\documentclass[11pt]{article} 

\usepackage{amsmath,amssymb,amsfonts,mathabx,tikz,mathrsfs,float,amsthm,enumitem} 
\usepackage[all]{xy}

\numberwithin{equation}{section} 
\newtheorem{thm}{Theorem}[section]
\newtheorem{prp}[thm]{Proposition} 
\newtheorem{lmm}[thm]{Lemma}  

\newtheorem{cnj}[thm]{Conjecture}

\theoremstyle{remark}

\theoremstyle{definition}
\newtheorem{dfn}[thm]{Definition}
\newtheorem{eg}[thm]{Example}

\def\BE#1{\begin{equation}\label{#1}}  
\def\EE{\end{equation}}

\def\ti#1{\widetilde{#1}} 
\def\e_ref#1{(\ref{#1})} 
\def\lan{\langle} \def\ran{\rangle} 
\def\lr#1{\lan#1\ran} 
\def\llrr#1{\lan\!\lan#1\ran\!\ran}
\def\blr#1{\big\lan#1\big\ran}

\def\lra{\longrightarrow}
\def\sf#1{\textsf{#1}}
\def\ov#1{\overline{#1}}
\def\un#1{\underline{#1}}  
\def\wh#1{\widehat{#1}}
\def\wch#1{\widecheck{#1}}

\def\al{\alpha}
\def\be{\beta}
\def\ga{\gamma}
\def\de{\delta}
\def\ep{\epsilon} 
\def\io{\iota}
\def\ka{\kappa}
\def\la{\lambda}
\def\na{\nabla}
\def\om{\omega}
\def\si{\sigma}

\def\up{\upsilon}
\def\vt{\vartheta}
\def\ve{\varepsilon}
\def\vp{\varpi}
\def\vph{\varphi}

\def\ze{\zeta}

\def\A{\mathcal A}

\def\C{\mathbb C}

\def\cC{\mathcal C}

\def\cD{\mathcal D}
\def\fd{\mathfrak m}
\def\bE{\mathbb E}

\def\F{\mathcal F}
\def\fF{\mathfrak F}

\def\g{\mathfrak g}

\def\cJ{\mathcal J}

\def\M{\mathfrak M}
\def\fM{\mathfrak M}
\def\cM{\mathcal M}
\def\cN{\mathcal N}

\def\O{\mathcal O}
\def\Q{\mathbb Q}

\def\P{\mathbb P}
\def\cP{\mathcal P}

\def\R{\mathbb R}

\def\T{\mathcal T}
\def\ft{\mathfrak t}
\def\bT{\mathbb T}
\def\sT{\mathscr T}

\def\fU{\mathfrak U}
\def\X{\mathfrak X}

\def\cZ{\mathcal Z}
\def\Z{\mathbb Z}

\def\De{\Delta}
\def\Ga{\Gamma}
\def\La{\Lambda}
\def\Om{\Omega}
\def\Si{\Sigma}
\def\Th{\Theta}

\def\fc{\mathfrak c}
\def\d{\mathfrak d}
\def\f{\mathbf f}
\def\fI{\mathfrak i}
\def\fJ{\mathfrak j}
\def\u{\mathbf u}
\def\x{\mathbf x}
\def\fR{\mathfrak R}

\def\ff{\mathfrak f}
\def\fs{\mathfrak s}

\def\i{\infty}
\def\prt{\partial}
\def\dbar{\bar\partial}
\def\eset{\emptyset} 
\def\0{\mathbf 0}

\def\tn{\textnormal}
\def\tnd{\textnormal d}
\def\bfe{\mathbf e}

\def\E{\textnormal E}
\def\V{\textnormal V}
\def\GW{\textnormal{GW}}

\def\Aut{\tn{Aut}}

\def\Cntr{\textnormal{Cntr}}

\def\Edg{\textnormal{Edg}}

\def\ev{\textnormal{ev}}

\def\id{\textnormal{id}}

\def\Im{\textnormal{Im}}

\def\Obs{\tn{Obs}}

\def\reg{\textnormal{reg}}

\def\rk{\textnormal{rk}}
\def\Sym{\tn{Sym}}
\def\sgn{\textnormal{sgn}}

\def\st{\textnormal{s.t.}}
\def\v{\textnormal{w}}
\def\val{\tn{val}}
\def\Ver{\textnormal{Ver}}
\def\vir{\textnormal{vir}}

\def\tne{{e}}

\def\bc{\mathbf c}

\def\bu{\bullet}

\def\inn{\!\in\!}
\def\eq{\!=\!}

\begin{document}

\title{Lower Bounds for Enumerative Counts\\
of Positive-Genus Real Curves}

\author{Jingchen Niu\thanks{Partially supported by DMS Grants 0846978 and 1500875} 
~and Aleksey Zinger\thanks{Partially supported by DMS Grant 1500875 and MPIM}}
\date{\today}

\maketitle

\begin{abstract}
\noindent
We transform the positive-genus real  
Gromov-Witten invariants of many real-orientable symplectic threefolds
into signed counts of curves.
These integer invariants provide lower bounds for counts of real curves 
of a given genus that pass through conjugate pairs of constraints.
We conclude with some implications and related conjectures 
for one- and two-partition Hodge integrals.
\end{abstract}

\tableofcontents

\section{Introduction}
\label{intro_sec}

\noindent
Gromov-Witten invariants of a symplectic manifold~$(X,\om)$
are counts of $J$-holomorphic curves in~$X$; they are usually rational numbers.
However, it is well-known that the genus~0 Gromov-Witten invariants of
Fano manifolds are precisely counts of rational curves;
this observation is key to enumerating rational curves in the complex projective space~$\P^n$ 
in \cite[Section~5]{KoM} and \cite[Section~10]{RT}.
Positive-genus Gromov-Witten invariants of many symplectic threefolds 
are transformed into integer counts of $J$-holomorphic curves 
in~\cite{FanoGV} in the process of establishing the ``Fano" case of
the Gopakumar-Vafa prediction of \cite[Conjecure~2(i)]{P2}.
In this paper,  we obtain a similar result for the real Gromov-Witten invariants 
defined in~\cite{RealGWsI}; see Theorem~\ref{main_thm}.
The resulting integer invariants count genus~$g$ real curves with certain signs
and thus provide lower bounds for the actual counts of such curves.
For example, we find that there are at least~10 genus~2 degree~7 real curves
passing through 7~general pairs of conjugate points in~$\P^3$
and at least~40 genus~5 degree~8 real curves
passing through 8~general pairs of conjugate points in~$\P^3$.\\

\noindent 
A \sf{real symplectic manifold} is a triple $(X,\om,\phi)$ consisting 
of a symplectic manifold~$(X,\om)$ and an \sf{anti-symplectic} involution~$\phi$ such that 
$\phi^*\om\!=\!-\om$. 
A symmetric surface $(\Si,\si)$ is a closed connected oriented, possibly nodal, 
surface~$\Si$ ($\dim_{\R}\Si\!=\!2$) with an orientation-reversing involution~$\si$. 
The fixed locus of $\si$ is a disjoint union of circles. 
A \sf{real map}
$$u:(\Si,\si)\lra(X,\phi) $$
is a smooth map $u:\Si\!\lra\!X$ such that $u\!\circ\!\si\eq\phi\!\circ\!u$. 
Let $\cJ_\om$ be the space of $\om$-compatible almost complex structures on $X$ and define
$$\cJ^\phi_\om=\big\{J\!\in\!\cJ_{\om}\!:\,\phi^*J\!=\!-J\big\}.$$ 
For $g,l\!\in\!\Z^{\ge 0}$, $B\!\in\!H_2(X;\Z)$, and $J\!\in\!\cJ_\om^\phi$, 
we denote by $\ov\M_{g,l}(X,B;J)^{\phi}$ the moduli space of equivalence classes 
of stable degree~$B$ $J$-holomorphic real maps from genus~$g$ symmetric 
(possibly nodal) surfaces with $l$~pairs of conjugate marked points;
see \cite[Section~1.1]{RealGWsI}.
For each $i\!=\!1,\ldots,l$, let 
$$\ev_i\!: \ov\fM_{g,l}(X,B;J)^{\phi}\lra X, \qquad
\big[u,(z_1^+,z_1^-),\ldots,(z_l^+,z_l^-)\big]\lra u(z_i^+),$$
be the evaluation at the first point in the $i$-th pair of conjugate points.\\

\noindent
Suppose $(X,\om,\phi)$ is a compact real symplectic threefold.
By \cite[Theorem~1.3]{RealGWsI}, a real orientation on $(X,\om,\phi)$ 
in the sense of \cite[Definition~1.2]{RealGWsI} determines an orientation on
the moduli space $\ov\fM_{g,l}(X,B;J)^{\phi}$ and endows~it
with a virtual class of dimension
\BE{virdim_e}
\dim_\R\big[\ov\M_{g,l}(X,B;J)^\phi\big]^{\vir}= c_1(B)+2l,
\EE
where $c_1(B)\!\equiv\!\lr{c_1(TX),B}$.
If $\mu_1,\ldots,\mu_l\!\in\!H^*(X;\Z)$, let
\BE{GWdfn_e}
\GW_{g,B}^{X,\phi}\big(\mu_1,\ldots,\mu_l\big) \equiv 
\bigg\lan\prod_{i=1}^l
\ev_i^*\mu_i,\big[\ov\M_{g,l}(X,B;J)^\phi\big]^{\vir}\bigg\ran\EE
be the corresponding \sf{genus~$g$ degree~$B$ real  GW-invariant} of~$(X,\om,\phi)$.
By~\e_ref{virdim_e}, this number is zero unless
\BE{dimcond_e}
\sum_{i=1}^l\dim_\R\mu_i=c_1(B)+2l.\EE 
In general, this number is rational and includes contributions from lower-genus curves.\\

\noindent
In this paper, we determine the lower-genus contributions to the number~\e_ref{GWdfn_e}.
We find that the real genus~$h$ curves with $h\!<\!g$ do not contribute to this number
unless $g$ and~$h$ have the same parity; see Propositions~\ref{PrpNoContr} and~\ref{PrpOppPar}.
If  $h\!<\!g$ and $g$ and~$h$ have the same parity, the genus~$h$ curves do contribute 
to~\e_ref{GWdfn_e}; see Proposition~\ref{PrpSamPar}.
These contributions are closely related to the contributions from genus~$h$ curves 
to the genus $(g\!+\!h)/2$ GW-invariant in the complex setting computed in~\cite{P1};
see Theorem~\ref{main_thm}.
The integer invariants provided by Theorem~\ref{main_thm} extract direct lower bounds for 
the counts of real curves in~$X$ from the real GW-invariants of~\cite{RealGWsI}.\\

If $(X,\om,\phi)$ is a compact real symplectic threefold, $g,l\!\in\!\Z^{\ge 0}$, 
$B\!\in\!H_2(X;\Z)\!-\!0$, and $J\inn\cJ_\om^\phi$, let
$$\M_{g,l}^*(X,B;J)^\phi\subset\ov\M_{g,l}(X,B;J)^\phi$$
denote the subspace consisting of \sf{simple} real maps from smooth domains;
see~\cite[Section~2.5]{McS}.
If 
$$ f_i:\, Y_i\lra X,\qquad1\le i\le l, $$
are pseudocycle representatives for the Poincare duals of $\mu_1,\ldots,\mu_l$,
let
\begin{gather}
\label{MBarF_e}\begin{split}
 \ov\M_{g,\f}(X,B;J)^\phi=\Big\{\big([u,(z_i^+,z_i^-)_{i=1}^l],(y_i)_{i=1}^l\big)
\in \ov\M_{g,l}\big(X,B;J\big)^\phi\!\times\!\prod_{i=1}^l\!Y_i\!:~~&\\
u(z_i^+)\!=\!f_i(y_i)~\forall\,i&\Big\},
\end{split}\\
\notag
\M_{g,\f}^*(X,B;J)^\phi= \ov\M_{g,\f}(X,B;J)^\phi\cap
\bigg(\M_{g,l}^*(X,B;J)^\phi\times \prod_{i=1}^lY_i\bigg).
\end{gather}
If $\M_{g,\f}^*(X,B;J)^\phi$ is cut out transversely,
we denote its signed cardinality by $\E_{g,B}^{X,\phi}(J,\f)$.\\

The expected dimension of
$\ov\M_{g,l}(X,B;J)^\phi$ is independent of the genus~$g$ if the (real) 
dimension of $X$ is 6; see~\e_ref{virdim_e}.
Thus, one can mix curve counts of different genera passing through 
the same constraints.
If $c_1(B)\!<\!0$, then 
the moduli space of unmarked maps has a negative expected dimension
and all degree~$B$ GW-invariants vanish.
This leaves the ``Calabi-Yau" case, $c_1(B)\!=\!0$, and
the ``Fano" case, $c_1(B)\!>\!0$.
For $g,h\!\in\!\Z^{\ge 0}$, define $\ti C_{h,B}^X(g)\!\in\!\Q$ by
\BE{Cdfn_e} \sum_{g=0}^{\i}\ti C_{h,B}^X(g)t^{2g}
=\bigg(\frac{\sinh(t/2)}{t/2}\bigg)^{\!h-1+c_1(B)/2}\,.\EE
For example, 
\begin{equation*}\begin{split}
 &\ti C_{h,B}^X(0)\eq 1,\quad
 \ti C_{h,B}^X(1)\eq \frac{2h\!-\!2\!+\!c_1(B)}{48},\quad
 \ti C_{h,B}^X(2)\eq\frac{\big(2h\!-\!2\!+\!c_1(B)\big)
\big(5(2h\!-\!2\!+\!c_1(B))\!-\!4\big)}{23040}.
\end{split}\end{equation*}

\begin{thm}\label{main_thm}
Suppose $(X,\om,\phi)$ is a compact real-orientable symplectic threefold with a choice of real orientation, $g,l\!\in\!\Z^{\ge 0}$, $B\!\in\!H_2(X;\Z)$,
and $\mu_i\!\in\!H^*(X;\Z)$ for $1\!\le\!i\!\le\!l$ are
such that~$c_1(B)\!>\!0$ and~\e_ref{dimcond_e} is satisfied.
\begin{enumerate}[label=(\arabic*),leftmargin=*]
\item There exists a subset $\cJ^\phi_{\reg}(g,B)\!\subset\!\cJ_\om^\phi$ of the second category such that for all $h\!\le\!g$:
\begin{enumerate}[label=$\bu$,leftmargin=*]
\item the moduli space $\M_{h,l}^*(X,B;J)^\phi$ consists of regular maps;
\item for a generic choice of pseudocycle representatives
$f_i\!:Y_i\!\lra\!X$ for $\mu_i$, $\M_{h,\f}^*(X,B;J)^\phi$ is a finite set of
regular pairs $([u,(z_i^+,z_i^-)_{i=1}^l],(y_i)_{i=1}^l)$ such that $u$ is an embedding.
\end{enumerate}

\item The numbers $\E_{h,B}^{X,\phi}(\f,J)$, with $h\!\le\!g$, are independent of the choice of 
$J\!\in\!\cJ^\phi_{\reg}(g,B)$ and~$f_i$ and can thus be denoted 
$\E_{h,B}^{X,\phi}(\mu_1,\ldots,\mu_l)$.

\item With $\ti C_{h,B}^{X,\phi}(g)$ defined by~\e_ref{Cdfn_e},
\BE{FanoGV_e} 
\GW_{g,B}^{X,\phi}\big(\mu_1,\ldots,\mu_l\big)=
\sum_{\begin{subarray}{c}0\le h\le g\\ g-h\in 2\Z\end{subarray}}\!\!
\ti C_{h,B}^X\big(\tfrac{g-h}{2}\big) \E_{h,B}^{X,\phi}\big(\mu_1,\ldots,\mu_l\big).\EE

\end{enumerate}
\end{thm}

\noindent
With $\ep\!=\!0,1$, \e_ref{FanoGV_e} gives
\BE{lowgenusGV_e}\begin{split}
\GW_{\ep,B}^{X,\phi}\big(\un\mu\big)&=\E_{\ep,B}^{X,\phi}\big(\un\mu\big),\quad
\GW_{2+\ep,B}^{X,\phi}(\un\mu)=\E_{2+\ep,B}^{X,\phi}\big(\un\mu\big)
+\frac{c_1(B)\!-\!2\!+\!2\ep}{48}\,\E_{\ep,B}^{X,\phi}\big(\un\mu\big),\\
\GW_{4+\ep,B}^{X,\phi}\big(\un\mu\big)&=\E_{4+\ep,B}^{X,\phi}\big(\un\mu\big)
+\frac{c_1(B)\!+\!2\!+\!2\ep}{48}\,\E_{2+\ep,B}^{X,\phi}\big(\un\mu\big)\\
&\hspace{1.5in}+\frac{(c_1(B)\!-\!2+2\ep)(5c_1(B)\!-\!14+10\ep)}{23040}\,\E_{\ep,B}^{X,\phi}\big(\un\mu\big),
\end{split}\EE
where $\un\mu\!=\!(\mu_1,\ldots,\mu_l)$.
Different of the $g\eq 0$ case of~\e_ref{lowgenusGV_e} go back to \cite{Wel4,Wel6,Ge2,Teh};
the $g\eq 1$ case is~\cite[Theorem 1.4]{RealGWsII}.\\

\noindent
The analogue of Theorem~\ref{main_thm} in the complex setting is 
\cite[Theorem~1.5]{FanoGV};
it establishes the ``Fano" case of Gopakumar-Vafa prediction of \cite[Conjecure~2(i)]{P2}.
The GW- and enumerate invariants are then related by the formula
\BE{FanoGVC_e} 
\GW_{g,B}^{X}\big(\mu_1,\ldots,\mu_l\big)=
\sum_{h=0}^gC_{h,B}^X(g\!-\!h)\E_{h,B}^{X}\big(\mu_1,\ldots,\mu_l\big),\EE
with the coefficients defined~by
$$\sum_{g=0}^{\i} C_{h,B}^X(g)t^{2g}
=\bigg(\frac{\sin(t/2)}{t/2}\bigg)^{\!2h-2+c_1(B)}\,.$$
The relation~\e_ref{FanoGVC_e} is also invertible and determines the complex enumerative invariants from  
the GW-invariants.\\

The standard conjugation on $\P^3$ is given by
$$\tau_4:\, \P^3\lra\P^3,\quad
 [Z_1,\ldots,Z_4]\mapsto[\Bar Z_1,\ldots,\Bar Z_4].$$
Let $\ell\inn H_2(\P^3;\Z)$ denote the homology class of a line.
By~\cite[Theorem 1.6]{RealGWsII} and~\e_ref{FanoGV_e},
$\E_{g,d\ell}^{\P^3,\tau_4}\eq 0$ if $d\!-\!g\inn 2\Z$.
In general, the genus~$g$ real  GW-invariants with conjugate pairs of 
constraints can be computed using the virtual equivariant localization theorem of~\cite{GP};
see~\cite[Section~4.2]{RealGWsIII}.
The standard complex structure $J_{\P^3}$ is regular in the sense of Theorem~\ref{main_thm}
and \cite[Theorem~1.5]{FanoGV}
if $d\!\ge\!2g\!-\!1$;
it may be regular for some smaller degrees as well.
Thus, Theorem~\ref{main_thm} provides lower bounds for the number of 
positive-genus real curves in~$\P^3$.
Tables~\ref{TabInv} and~\ref{TabInvReal} 
show some complex and real GW- and enumerative invariants of~$\P^3$;
the complex numbers are obtained from~\cite{growi} and~\e_ref{FanoGVC_e}. 
The numbers in the two tables are consistent with each other as well as
with the Castelnuovo bounds.\\

\begin{table}[htp]
\renewcommand{\arraystretch}{1.5} \renewcommand{\tabcolsep}{0.2cm}
\begin{center}
\begin{tabular}{|c|c|c|c|c|c|c|c|c|} 
\hline $d$&1&2&3&4&5&6&7&8\\
\hline $\GW_{0,d}$&1&0&1&4&105&2576&122129&7397760\\
\hline $\GW_{1,d}$&$-\frac{1}{12}$&0&$-\frac{5}{12}$&$-\frac{4}{3}$&$-\frac{147}{4}$&$\frac{1496}{3}$&$\frac{1121131}{12}$&14028960\\
\hline $\GW_{2,d}$&$\frac{1}{360}$&0&$\frac{1}{12}$&$-\frac{1}{180}$&$-\frac{49}{8}$&$-\frac{7427}{5}$&$-\frac{4905131}{45}$&$-7022780$\\
\hline $\GW_{3,d}$&$-\frac{1}{20160}$&0&$-\frac{43}{4032}$&$\frac{103}{1080}$&$\frac{473}{64}$&$\frac{206873}{270}$&$\frac{283305113}{8640}$&$-\frac{110089487}{63}$\\
\hline $\GW_{4,d}$&$\frac{1}{1814400}$&0&$\frac{713}{725760}$&$-\frac{26813}{907200}$&$-\frac{833}{320}$&$-\frac{12355247}{56700}$&$-\frac{1332337}{34560}$&$\frac{117632950}{63}$\\
\hline $\E_{0,d}$&1&0&1&4&105&2576&122129&7397760\\
\hline $\E_{1,d}$&0&0&0&1&42&2860&225734&23276160\\
\hline $\E_{2,d}$&0&0&0&0&0&312&83790&18309660\\
\hline $\E_{3,d}$&0&0&0&0&0&11&10800&6072960\\
\hline $\E_{4,d}$&0&0&0&0&0&0&605&980100\\
\hline 
\end{tabular}
\end{center}
\caption{Complex GW- and enumerative invariants of $\P^3$ with point constraints \label{TabInv}}
\ \\
\begin{center}
\begin{tabular}{|c|c|c|c|c|c|c|c|c|} 
\hline $d$&1&2&3&4&5&6&7&8\\
\hline $\GW^\phi_{0,d}$&1&0&$-1$&0&5&0&$-85$&0\\
\hline $\GW^\phi_{1,d}$&0&0&0&$-1$&0&$-4$&0&$-1000$\\
\hline $\GW^\phi_{2,d}$&$\frac{1}{24}$&0&$-\frac{5}{24}$&0&$\frac{15}{8}$&0&$-\frac{1345}{24}$&0\\
\hline $\GW^\phi_{3,d}$&0&0&0&$-\frac{1}{3}$&0&$-3$&0&$-\frac{2840}{3}$\\
\hline $\GW^\phi_{4,d}$&$\frac{1}{1920}$&0&$-\frac{23}{1152}$&0&$\frac{43}{128}$&0&$-\frac{2475}{128}$&0\\
\hline $\GW^\phi_{5,d}$&0&0&0&$-\frac{19}{360}$&0&$-\frac{16}{15}$&0&$-\frac{1400}{3}$\\
\hline $\E^\phi_{0,d}$&1&0&$-1$&0&5&0&$-85$&0\\
\hline $\E^\phi_{1,d}$&0&0&0&$-1$&0&$-4$&0&$-1000$\\
\hline $\E^\phi_{2,d}$&0&0&0&0&0&0&$-10$&0\\
\hline $\E^\phi_{3,d}$&0&0&0&0&0&$-1$&0&$-280$\\
\hline $\E^\phi_{4,d}$&0&0&0&0&0&0&$-1$&0\\
\hline $\E^\phi_{5,d}$&0&0&0&0&0&0&0&$-40$\\
\hline 
\end{tabular}
\end{center}
\caption{Real GW- and enumerative invariants of $(\P^3,\tau_4)$ with conjugate pairs of point constraints \label{TabInvReal}}
\end{table}

The first claim of Theorem~\ref{main_thm} is standard; see the beginning of Section~\ref{SecNotation}.
The second claim follows immediately from~\e_ref{FanoGV_e} and 
$\ti C_{h,B}^X(0)\eq 1$ for all $h\inn\Z^+$.
Thus, the key part of Theorem~\ref{main_thm} is the last one.
Its analogue in the complex setting, i.e.~the last part of \cite[Theorem~1.5]{FanoGV},
is deduced~from 
\begin{enumerate}[label=$\bu$,leftmargin=*]

\item  \cite[Theorem~1.2]{FanoGV}, which compares the GW-invariants of a manifold
with the GW-invariants of a submanifold which contains all relevant curves,~and 

\item  \cite[Theorem~3]{P1}, which integrates the Euler classes of the relevant obstruction
bundles over the moduli spaces of stable maps into smooth curves.

\end{enumerate}
It should be fairly straightforward to establish the  analogue of \cite[Theorem~1.2]{FanoGV}
in the real setting using the approach of~\cite{FanoGV}.
It does not appear so straightforward to establish  the real analogue of 
\cite[Theorem~3]{P1} using the approach of~\cite{P1} though, 
as the topology of the Deligne-Mumford moduli spaces
is less well understood.\\

\noindent
We instead adapt a direct approach  to the last part of Theorem~\ref{main_thm} 
and determine the contribution of every 
(usually non-compact) stratum of the moduli space of the constrained maps to
the number~\e_ref{GWdfn_e}.
We show that only the simplest possible boundary strata in fact contribute.
The contributions of these boundary strata are given by the number of zeros 
of affine bundle maps from the associated virtual normal bundles to 
the obstruction bundles.
In the process, we establish the real analogue of \cite[Theorem~3]{P1}, 
but using a  different approach (as far as the applications of Hodge integrals~go).\\

Propositions~\ref{PrpNoContr}-\ref{PrpSamPar} in the next section
 decompose~\e_ref{GWdfn_e} into contributions 
from individual strata of the moduli space~\e_ref{MBarF_e}.
We deduce~\e_ref{FanoGV_e} from these propositions and~\cite[(1)]{P1}
at the end of Section~\ref{SecNotation}. 
In Section~\ref{SecPrps}, we describe the contribution of each curve~$\cC$ of genus~$h\!<\!g$
to~\e_ref{GWdfn_e} in terms of the deformation-obstruction complex for each stratum
of~\e_ref{MBarF_e} associated with~$\cC$; see Proposition~\ref{PrpSection}.
We obtain Propositions~\ref{PrpNoContr}-\ref{PrpSamPar}  from Proposition~\ref{PrpSection}
in Section~\ref{SecPfPrps} and establish the latter in Section~\ref{SecAnalytic}.
In Section~\ref{LowDegree_subs}, we compute  low-degree real GW-invariants of~$(\P^3,\tau_4)$ 
by equivariant localization.
As a corollary of Theorem~\ref{main_thm} and these computations,
we obtain closed formulas for one- and two-partition Hodge integrals
and formulate related conjectures;
see Proposition~\ref{Hodge_prp} and Conjectures~\ref{F1_cnj} and~\ref{F2_cnj}.\\ 

We would like to thank J.~Bryan, R.~Cavalieri, C.-C.~Liu, and A.~Oblomkov for discussions 
on Hodge integrals. 
The second author is grateful to MPIM for its hospitality while this project 
was being completed.

\section{Overview of the proof}
\label{SecNotation}

Let $(X,\om,\phi)$ be a compact real-orientable symplectic threefold and $J\inn\cJ^\phi_\om$.
If $(\Si,\si)$ is a symmetric surface,
$\fJ$ is a complex structure on $\Si$ such that $\si^*\fJ\eq-\fJ$,
and $u\!:(\Si,\si)\!\lra\!(X,\phi)$ is a real map,
we define
\begin{equation*}\begin{split}
 &\Ga(u)^{\phi,\si}\equiv\big\{\xi\inn \Ga(\Si;u^*TX)\!:\,\tnd\phi\!\circ\!\xi=\xi\!\circ\!\si\big\},\\
 &\Ga^{0,1}(u)^{\phi,\si}\equiv\big\{\eta\inn\Ga\big(\Si;(T^*\Si,\fJ)^{0,1}\!\otimes_\C\!u^*TX\big)\!:\,
 \tnd\phi\!\circ\!\eta=\eta\!\circ\!\tnd\si\big\}.
\end{split}\end{equation*}
If in addition $u$ is $(J,\fJ)$-holomorphic, let
\BE{Ddfn_e} D_u^{\phi}:\,\Ga(u)^{\phi,\si}\lra\Ga^{0,1}(u)^{\phi,\si}\EE
be the linearization of the real $\Bar\prt_J$-operator at $u$.\\

We take the subset $\cJ_{\reg}^\phi(g,B)$ to consist of~$J$ in $\cJ_\om^\phi$ 
so that $D_u^{\phi}$ is surjective for every simple real $J$-holomorphic 
map $u$ from a genus $h$ symmetric surface $(\Si,\si)$ such that
$$h\le g\qquad\tn{and}\qquad \blr{\om,u_*[\Si]}\le\blr{\om,B}. $$
By the proof of~\cite[Theorem 3.1.5]{McS}, $\cJ_{\reg}^\phi(g,B)$ 
is of the second category in $\cJ_\om^\phi$.\\

Let $g,$ $B$, and $\mu_1,\ldots,\mu_l$ be as in Theorem~\ref{main_thm},
and $J\inn\cJ^\phi_{\reg}(g,B)$.
With $\f\!\equiv\!(f_1,\ldots,f_l)$ as before, define
\BE{cMhfdfn_e}\cM_{h,\f}^{\phi}(B;J)\equiv \big\{\Im\,u\!:\,
\big([u,(z_i^+,z_i^-)_{i=1}^l],(y_i)_{i=1}^l\big)\!\in\!\M_{h,\f}^*(X,B;J)^\phi\big\}.\EE
If $\f$ is chosen generically, 
then $\M_{h,\f}^*(X,B;J)^\phi$ is cut out transversely and is 
a finite set for all $h\!\le\!g$ by~\cite[Proposition 3.2]{FanoGV};
see also~\cite[Section 1.2]{FanoGV}.
Furthermore, the collections~\e_ref{cMhfdfn_e} with $h\!\le\!g$ consist of 
smooth disjoint curves meeting the pseudocycles~$f_i$ at distinct non-real points.
Under these assumptions, we denote by $\sgn(\u)$ the sign~of 
\BE{u_e}[\u]\equiv\big([u,(z_i^+,z_i^-)_{i=1}^l],(y_i)_{i=1}^l\big)\in\M_{h,\f}^*(X,B;J)^\phi\EE
and define $\sgn_{\f}(\Im\,u)\!=\!\sgn(\u)$.\\

\noindent
Under the assumptions of the previous paragraph,
$$\ov\M_{g,\f}(X,B;J)^\phi=
\bigsqcup_{h=0}^{g}\,\bigsqcup_{\cC\in\cM_{h,\f}^{\phi}(B;J)}
\hspace{-0.25in} \ov\M_{g,\f}(\cC)^\phi,
\qquad\tn{where}\quad \ov\M_{g,\f}(\cC)^\phi\equiv\ov\M_{g,\f}(\cC,[\cC];J)^\phi.$$
The domain $(\Si_u,\si_u)$ of an element~$[\u]$ of $\ov\M_{g,\f}(\cC)^\phi$
has a unique smooth irreducible component~$\Si_{u;0}$,
called the \sf{center component}, so that $u_0\!\equiv\!u|_{\Si_{u;0}}$ is non-constant.
Each of the remaining components, called \sf{bubble components},
has at most one node in common with~$\Si_{u;0}$ 
because the elements of~\e_ref{cMhfdfn_e} are smooth curves in~$X$.
We denote by $\T_{\u}\!\equiv\!(\ft_{\u},\fd_{\u})$ the pair consisting of the topological type
of the real symmetric surface~$(\Si_u,\si_u)$  and by $\fd_{\u}$
the assignment of $l$ conjugate pairs of marked points to the irreducible components 
of~$\Si_u$.\\

\noindent
For each pair $\T\!\equiv\!(\ft,\fd)$ consisting of a topological type of  
genus~$g$ symmetric surfaces~$(\Si,\si)$ 
and an assignment of $l$ conjugate pairs of marked points to the irreducible components 
of~$\Si$, let
\BE{Cstrata_e}
 \M_{\T,\f}(\cC)^\phi\equiv\big\{[\u]\!\in\!\ov\M_{g,\f}(\cC)^\phi\!:\,\T_{\u}=\T\big\}.\EE
The subspaces $\M_{\T,\f}(\cC)^\phi$ then partition $\ov\M_{g,\f}(\cC)^\phi$ 
into smooth, but generally non-compact, strata.
We denote by $\sT_{g,l}(\cC)^\phi$ the collection of all pairs $\T\!\equiv\!(\ft,\fd)$ 
for which the space~\e_ref{Cstrata_e} is non-empty.\\

We will call a topological type $\ft$ as above \sf{basic}
if it describes nodal symmetric surfaces consisting of smooth bubble components 
attached directly to the center component.
For any topological type~$\ft$, let $\fd_0(\ft)$ be the distribution 
that assigns all marked points to the center component.
We denote by $\sT_{g,l}^*(\cC)^\phi\!\subset\!\sT_{g,l}(\cC)^\phi$ the set of all tuples 
$\T\eq\big(\ft,\fd_0(\ft)\big)$ such that $\ft$ is basic.\\

Given $J\inn\cJ^\phi_\reg(g,B)$, 
let $\A^\phi_{g,l}(J)$ be the set of $\phi$-invariant $\Bar\prt_J$-admissible perturbations~$\nu$ 
of $\Bar\prt_J$, as in \cite[Section~1]{Ge2} and \cite[Section~3.1]{RealGWsIII}.
For $\nu\inn\A^\phi_{g,l}(J)$,
we denote by $\M^*_{g,l}(X,B;J,\nu)^\phi$
the moduli space of equivalence classes 
of simple degree~$B$ $(J,\nu)$-holomorphic real maps from smooth genus~$g$ symmetric surfaces with $l$~pairs of conjugate marked points.
Let
\BE{MJnu_e}\begin{split}
\M^*_{g,\f}(X,B;J,\nu)^\phi\eq
 \Big\{\big([u,(z_i^+,z_i^-)_{i=1}^l],(y_i)_{i=1}^l\big)
\in \M^*_{g,l}\big(X,B;J,\nu\big)^\phi\!\times\!\prod_{i=1}^l\!Y_i\!:&\,\\
u(z_i^+)\!=\!f_i(y_i)~\forall\,i&\Big\}.
\end{split}\EE
If $\f$ and $\nu$ are generic, this is a compact 0-dimensional manifold.
A real orientation on $(X,\om,\phi)$ determines a sign for each element of~\e_ref{MJnu_e}.
The signed cardinality $^\pm\!|\M^*_{g,\f}(X,B;J,\nu)^\phi|$ of~\e_ref{MJnu_e}
is the  number~\e_ref{GWdfn_e}.
We establish Theorem~\ref{main_thm} by splitting~\e_ref{MJnu_e}  into contributions from
the strata~$\M_{\T,\f}(\cC)^\phi$ of $\ov\M_{g,\f}(X,B;J)^\phi$ 
as described by Definition~\ref{DfnContri} and Proposition~\ref{PrpNoContr} below.\\

\noindent
Let $\X_{g,l}(X,B)^{\phi}$ be the configuration space of equivalence classes 
of stable degree~$B$ real maps from genus~$g$ symmetric 
(possibly nodal) surfaces with $l$~pairs of conjugate marked points.
This space is topologized as in \cite[Section~3]{LiT}. 
Define 
$$\X_{g,\f}(X,B)^{\phi}=
 \Big\{\big([u,(z_i^+,z_i^-)_{i=1}^l],(y_i)_{i=1}^l\big)
\in \X_{g,l}\big(X,B\big)^\phi\!\times\!\prod_{i=1}^l\!Y_i\!:\,
u(z_i^+)\!=\!f_i(y_i)~\forall\,i\Big\}.$$

\begin{dfn}[{\cite[Definition 2.4]{LiZ}}]\label{DfnContri}
Let $(X,\om,\phi)$, $B$, $g,$ $l$, and $\f$ be as in Theorem~\ref{main_thm}
and $J\inn\cJ^\phi_{\reg}(g,B)$.
A subset $\cZ\!\subset\!\ov\M_{g,\f}(X,B;J)^\phi$ is \textsf{$\dbar_J$-regular}
if there exist $\Cntr_{\f}^{\phi}(\cZ)\!\in\!\Q$ and 
a dense open subset $\A_{\cZ}^{\phi}(J)\!\subset\!\A_{g,l}^{\phi}(J)$ with 
the following property.
For every $\nu\!\in\!\A_{\cZ}^{\phi}(J)$, there exist 
\begin{enumerate}[label=(\alph*),leftmargin=*]
\item a compact subset $K_\nu\!\subset\!\cZ$,
\item  an open neighborhood $U_\nu(K)\!\subset\!\X_{g,\f}(X,B)^{\phi}$
of each compact subset $K\!\subset\!\cZ$, and 
\item $\ep_\nu(U)\inn\R^+$ for each open subset $U$ of $\X_{g,\f}(X,B)^{\phi}$
\end{enumerate}
such that 
$$ {}^\pm\big|\M^*_{g,\f}(X,B;J,t\nu)^\phi\cap U\big|=\Cntr_\f^\phi(\cZ)
 \qquad
 \tn{if}~t\inn\big(0,\ep_\nu(U)\big),~ K_\nu\!\subset\!K\!\subset\!U\!\subset\!U_\nu(K).$$
If $\cZ\!\subset\!\ov\M_{g,\f}(X,B;J)^\phi$ is $\dbar_J$-regular,
the corresponding number $\Cntr_\f^\phi(\cZ)$ is the \sf{$\dbar_J$-contribution} 
of~$\cZ$ to~\e_ref{GWdfn_e}.
\end{dfn}

\begin{prp}\label{PrpNoContr}
Let $(X,\om,\phi)$, $B$, $g,l$, and $\f$ be as in Theorem~\ref{main_thm}, 
$J\inn\cJ^\phi_{\reg}(g,B)$, 
and $\cC$ be an element of $\cM_{h,\f}^{\phi}(B;J)$.
For each $\T\!\in\!\sT_{g,l}(\cC)^\phi$, $\M_{\T,\f}(\cC)^\phi$ is a $\dbar_J$-regular
subspace of $\ov\M_{g,\f}(X,B;J)^\phi$.
Furthermore,  $\Cntr_\f^\phi(\M_{\T,\f}(\cC)^\phi)\!=\!0$ if 
$\T\!\not\in\!\sT_{g,l}^*(\cC)^\phi$.
\end{prp}

\noindent
The implication of Proposition~\ref{PrpNoContr} is that 
\BE{ContrSplit_e}\begin{split}
\GW_{g,B}^{X,\phi}\big(\mu_1,\ldots,\mu_l\big) 
&= \sum_{h=0}^g\sum_{\cC\in\cM_{h,\f}^{\phi}(B;J)}
\sum_{\T\in\sT_{g,l}(\cC)^\phi}
\!\!\!\!\!\!\!\Cntr_\f^\phi\big(\M_{\T,\f}(\cC)^\phi\big)\\
&= \sum_{h=0}^g\sum_{\cC\in\cM_{h,\f}^{\phi}(B;J)}
\sum_{\T\in\sT_{g,l}^*(\cC)^\phi}
\!\!\!\!\!\!\!\Cntr_\f^\phi\big(\M_{\T,\f}(\cC)^\phi\big)\,.
\end{split}\EE
It thus remains to reduce the last expression in~\e_ref{ContrSplit_e} 
to the right-hand side of~\e_ref{FanoGV_e}.\\

For a tuple $I\!\equiv\!(g_1,\ldots,g_m)$ of positive integers,
let 
$$\ell(I)\equiv m,\qquad\quad|I|\equiv g_1\!+\!\ldots\!+\!g_m\,.$$
For $I\eq(g_1,\ldots,g_m)$ and $I'\eq(g'_1,\ldots,g'_{m'})$,
denote by $\ft_{I,I'}(\cC)$ the topological type of nodal symmetric surfaces consisting 
of the symmetric surface $(\cC,\phi)$ as the center component,
\begin{enumerate}[label=$\bu$,leftmargin=*]
\item $m$ pairs of smooth  surfaces of genera $g_1,\ldots,g_m$ (the same genus in each pair) attached at conjugate points of $\cC$, and
\item $m'$ smooth symmetric surfaces attached at real points of $\cC$.
\end{enumerate}
The set $\sT_{g,l}^*(\cC)^\phi$ then consists of the pairs
$$ \T_{I,I'}\equiv \Big(\ft_{I,I'}(\cC),\fd_0\big(\ft_{I,I'}(\cC)\big)\Big) $$
such that $2|I|\!+\!|I'|\eq g\!-\!h$, where $h$ is the genus of~$\cC$.

\begin{prp}\label{PrpOppPar}
Let $(X,\om,\phi)$, $B$, $g,l$, $\f$, $J$, and~$\cC$ be as in Proposition~\ref{PrpNoContr}. If
\hbox{$I\inn(\Z^+)^m$} and $I'\inn(\Z^+)^{m'}$ are such that $2|I|\!+\!|I'|\eq g\!-\!h$ and $m'\!>\!0$, 
then $\Cntr_\f^\phi(\M_{\T_{I,I'},\f}(\cC)^\phi)\eq 0$.
\end{prp}

For each $I\eq(g_1,\ldots,g_m)$,
let $\T_I\!\equiv\!\T_{I,(\,)}$ 
and denote by $\Aut(I)\!\subset\!S_m$ the stabilizer of~$I$.
For each $g'\inn\Z^+$, let 
$$\cP(g')=\big\{(g_1,\ldots,g_m)\inn(\Z^+)^m\!:\,
 g_1\!+\!\ldots\!+\!g_m\eq g',~m\inn\Z^{\ge0}\big\}.$$
By~\e_ref{ContrSplit_e} and Proposition~\ref{PrpOppPar},
\BE{ContrSplit_e2}\begin{split}
  \GW_{g,B}^{X,\phi}\big(\mu_1,\ldots,\mu_l\big) 
 =\sum_{\begin{subarray}{c}0\le h\le g\\ g-h\in 2\Z\end{subarray}}\,
 \sum_{\cC\in\cM_{h,\f}^{\phi}(B;J)}
\sum_{I\in\cP((g-h)/2)}\!\!\!\!\!\frac{|\Aut(I)|}{\ell(I)!}
\Cntr_\f^\phi\big(\M_{\T_I,\f}(\cC)^\phi\big)\,.
\end{split}\EE
The last contribution is described Proposition~\ref{PrpSamPar} below.\\

For $g'\inn\Z^+$,
we denote by $\ov\cM_{g',1}$ the Deligne-Mumford moduli space of stable 
genus~$g'$ one-marked curves.
Let 
\BE{bELdfn_e}\bE,L\lra\ov\cM_{g',1}\EE
be the Hodge vector bundle of harmonic differentials
and the universal tangent line bundle at the marked point,
respectively.
Let $\la_k\eq c_k(\bE)$, $\psi\eq c_1(L^*)$,  and
\BE{pidfn_e}\pi_{g'},\pi_\cC\!:\,\ov\cM_{g',1}\!\times\!\cC\lra\ov\cM_{g',1},\,\cC\EE
be the component projection maps.
Let
$$\al_{g'} = \int_{\ov\cM_{g',1}}\!\!\!\!\!\la_{g'-1}\la_{g'}\!
\bigg(\sum_{r=0}^{g'-1}(-1)^r\la_r\psi^{g'-1-r}\bigg).$$
For each $g_c\!\in\!\Z^{\ge0}$,  define
\BE{tiCdfn_e} 
\wh C^X_{h,B}(g_c)= 
\sum_{(g_1,\ldots,g_m)\in\cP(g_c)}\!\!\!\!\!\!\!\!\!\!
\frac{(2\!-\!2h\!-\!c_1(B))^m}{2^m m!}\prod_{i=1}^m\!\big((-1)^{g_i}\al_{g_i}\big)\,.\EE

\begin{prp}\label{PrpSamPar}
Let $(X,\om,\phi)$, $B$, $g,l$, $\f$, $J$, and~$\cC$ be as in Proposition~\ref{PrpNoContr}.
For each element $I\!\equiv\!(g_1,\ldots,g_m)$ of $\cP((g\!-\!h)/2)$,
\begin{equation*}\begin{split}
 &\Cntr_\f^\phi\big(\M_{\T_I,\f}(\cC)^\phi\big)\\
 &\qquad= \frac{\sgn_{\f}(\cC)}{2^m|\Aut(I)|}
 \prod_{i=1}^{m}\!\!\Bigg(\!\!(-1)^{g_i}\!\!
\int_{\ov\cM_{g_i,1}\times\cC}\!\!
 c_{2g_i}\big(\pi_{g_i}^*\bE^*\!\otimes\!\pi_\cC^*\cN_X\cC\big)
\bigg(\sum_{r=0}^{g_i-1}\!(-1)^r\la_r\psi^{g_i-1-r}\bigg)\!\!\Bigg)\,,
\end{split}\end{equation*}
where $\cN_X\cC$ is the normal bundle of $\cC$ in~$X$.
\end{prp}

We now proceed as at the end of~\cite[Section 2.3]{P1}.
Since $\la_{g'}^2\eq 0$ by~\cite[(5.3)]{Mu},
\BE{ECexpand_e}c_{2g'}\big(\pi_{g'}^*\bE^*\!\otimes\!\pi_\cC^*\cN_X\cC\big)
 =-\pi_{g'}^*\big(\la_{g'-1}\la_{g'}\big)\cdot \pi_{\cC}^*\,c_1(\cN_X\cC).\EE
Since
\BE{ECexpand_e2}\int_\cC c_1(\cN_X\cC)=c_1(B)-(2\!-\!2h),\EE
Proposition~\ref{PrpSamPar} thus gives
$$\Cntr_\f^\phi\big(\M_{\T_I,\f}(\cC)^\phi\big)
 =\sgn_{\f}(\cC)
 \frac{(2\!-\!2h\!-\!c_1(B))^m}{2^m|\Aut(I)|}
\prod_{i=1}^m\!\big((-1)^{g_i}\al_{g_i}\big)\,.$$
Combining this with~\e_ref{ContrSplit_e2} and~\e_ref{tiCdfn_e}, we obtain
\begin{equation*}\begin{split}
  \GW_{g,B}^{X,\phi}\big(\mu_1,\ldots,\mu_l\big) 
 =\sum_{\begin{subarray}{c}0\le h\le g\\ g-h\in 2\Z\end{subarray}}\,
 \sum_{\cC\in\cM_{h,\f}^{\phi}(B;J)}
\hspace{-0.25in}\wh C_{h,B}^X\big(\tfrac{g-h}{2}\big)\sgn_{\f}(\cC)
 =\sum_{\begin{subarray}{c}0\le h\le g\\ g-h\in 2\Z\end{subarray}}\!\!\!\!
\wh C_{h,B}^X\big(\tfrac{g-h}{2}\big) \E_{h,B}^{X,\phi}(\f,J)\,.
\end{split}\end{equation*}
We show below that $\wh C^X_{h,B}(g)$ satisfies~\e_ref{Cdfn_e}.
This implies~\e_ref{FanoGV_e} and  thus establishes Theorem~\ref{main_thm}.\\

By the $g\eq 0$ case of~\cite[(26)]{P1} and~\cite[(1)]{P1},
$$\exp\bigg(\sum_{g'=1}^{\i}\!\al_{g'}t^{2g'}\bigg)=\bigg(\frac{\sin(t/2)}{t/2}\bigg)^{-1}.$$
Thus, by~(\ref{tiCdfn_e}),
\begin{equation*}\begin{split}
 \sum_{g_c=0}^{\i}\!\!\wh C^X_{h,B}(g_c)t^{2g_c}&=
 \exp\bigg(\!\big(1\!-\!h\!-\!c_1(B)/2\big)\!\sum_{g'=1}^{\i}\!\al_{g'}(\fI t)^{2g'}\bigg)\\
 &=\bigg(\frac{\sin(\fI t/2)}{\fI t/2}\bigg)^{\!h-1+c_1(B)/2}
 =\bigg(\frac{\sinh(t/2)}{t/2}\bigg)^{\!h-1+c_1(B)/2}.
\end{split}\end{equation*}
Thus, $\wh C^X_{h,B}(g_c)\eq\ti C^X_{h,B}(g_c)$.

\section{Stratawise contributions}
\label{SecPrps}

Proposition~\ref{PrpSection} below relates the contribution of each stratum 
$\M_{\T,\f}(\cC)^\phi$
in Proposition~\ref{PrpNoContr} to the zeros of a bundle map from
the space~$\ti\F\T^{\si}$ of smoothing parameters
to the obstruction bundle~$\Obs_{\T}^{\phi}$.
For the purposes of Proposition~\ref{PrpSection}, we describe the structure of $\M_{\T,\f}(\cC)^\phi$
in detail in terms of graphs.

\subsection{Rooted decorated graphs}
\label{RDG_subs}

A \sf{graph} $(\Ver,\Edg)$ is pair consisting of a finite set~$\Ver$ of \sf{vertices}
and an element 
$$\Edg\in \Sym^m(\Sym^2\Ver)$$ 
for some $m\!\in\!\Z^{\ge0}$.
We will view~$\Edg$ as a collection of two-element subsets of~$\Ver$, called \sf{edges},
but some of these subsets may contain the same element of~$\Ver$ twice and 
$\Edg$ may contain several copies of the same two-element subset.
Hereafter we use $w$ and $e$ to denote vertices and edges, respectively.
\\

An \sf{$S$-marked decorated graph} or simply \sf{decorated graph}
\BE{Dual_e} \T=\big(\Ver,\Edg,S,\g,\fd\big) \EE
consists of a graph $(\Ver,\Edg)$, a finite set $S$, and~maps
$$\g\!:\,\Ver\lra\Z^{\ge 0}\qquad\hbox{and}\qquad \fd\!:\,S\lra\Ver.$$
We define \sf{the arithmetic genus $g_a(\T)$ of~$\T$} as in~\e_ref{Dual_e}~by
$$g_a(\T)= p_a+\sum_{w\in\Ver}\!\g(w),$$
where $p_a$ is the arithmetic genus of the graph $(\Ver,\Edg)$.
For each $w\!\in\!\Ver$, let
$$\val_w(\T)=2\g(w)+\big|\fd^{-1}(w)\big|+\big|\{e\!\in\!\Edg\!:\,w\!\in\!e\big\}\big|,$$
with each $e\!=\!(w,w)$ counted twice above.
A decorated graph~$\T$ as in~\e_ref{Dual_e} is  called \sf{trivalent} if $\val_w(\T)\!\ge\!3$
for every $w\!\in\!\Ver$.\\

\noindent
A  \sf{decorated subgraph} of a decorated graph~$\T$ as in~\e_ref{Dual_e} is 
a decorated graph 
\BE{SubDual_e}\T'\eq(\Ver',\Edg',S',\g',\fd')\EE
such that $\Ver'\!\subset\!\Ver$ and
\begin{enumerate}[label=$\bu$,leftmargin=*]

\item $\Edg'\!\subset\!\Edg$ is the subcollection of the edges with both vertices in~$\Ver'$,

\item $S'$ is the disjoint union of $\fd^{-1}(\Ver')$ and 
the subcollection $\Edg'^{\bu}\!\subset\!\Edg$ of the  edges with one vertex in~$\Ver'$
and the other in $\Ver\!-\!\Ver'$, 

\item $\g'=\g|_{\Ver'}$,

\item $\fd'|_{\fd^{-1}(\Ver')}=\fd|_{\fd^{-1}(\Ver')}$ and
$\fd'(e)\!=\!w$ 
if $e\!\in\!\Edg'^{\bu}$ and $e\!\cap\!\Ver'\!=\!\{w\}$.
\end{enumerate}

Thus,
we choose the vertices $\Ver'\!\subset\!\Ver$ to be contained in~$\T'$ and
then cut the edges starting in~$\Ver'$, but ending in $\Ver\!-\!\Ver'$, in half
and thus convert them to marked points.
We define \sf{the complement} of a decorated subgraph~$\T'$ as in~\e_ref{SubDual_e} to be 
the decorated subgraph
\BE{Complement_e}
 (\T')^c= (\Ver'^c,\Edg'^c,S'^c,\g'^c,\fd'^c) \EE
of $\T$ with $\Ver'^c\!=\!\Ver\!-\!\Ver'$.\\

An \sf{involution} $\si$ on a decorated graph~$\T$ as in~\e_ref{Dual_e} is an automorphism of 
the graph~$(\Ver,\Edg)$ and the set~$S$ such~that
\BE{Invol_e}
 \si\circ\si=\id,\qquad 
 \g\circ\si=\g,\qquad \si\circ\fd=\fd\circ\si.\EE
In such a case, let $\V_{\R}^{\si}(\T)\!\subset\!\Ver$ and $\E_\R^{\si}(\T)\!\subset\!\Edg$ be 
the subsets consisting of the fixed points of~$\si$ and 
$$\V_\C^{\si}(\T)\equiv\Ver\!-\!\V_\R^{\si}(\T),\qquad  \E_\C^{\si}(\T)\equiv\Edg\!-\!\E_\R^{\si}(\T).$$
A \sf{rooted decorated graph} 
\BE{RDG_e} \T=\big(\Ver,\Edg,S,\g,\fd;\si,\v_0\big) \EE
is a connected decorated graph with an involution~$\si$ and  a special vertex $\v_0\!\in\!\Ver$, 
called the \sf{root}, such~that $\si(\v_0)\eq\v_0$ and 
\begin{enumerate}[label=$\bu$,leftmargin=*]

\item there are no loops in~$(\Ver,\Edg)$  passing through~$\v_0$
(i.e.~removing any edge containing $\v_0$ disconnects this graph), and 

\item $\val_w(\T)\!\ge\!3$ for all $w\!\in\!\Ver\!-\!\{\v_0\}$;

\end{enumerate}
see Figure~\ref{FigDualGraph}.
In such a case, let $\E_0(\T)\!\subset\!\Edg$ be the subset of edges containing the root~$\v_0$ and
$$\E_{0;\C}^\si(\T)=\E_0(\T)\cap\E_{\C}^{\si}(\T), \qquad 
\E_{0;\R}^\si(\T)=\E_0(\T)\cap\E_{\R}^{\si}(\T)\,.$$
Since $\E_0(\T)$ is preserved by the action of~$\si$, so is~$\E_{0;\C}^\si(\T)$.\\

A \sf{rooted decorated subgraph} of a rooted decorated graph~$\T$ as in~\e_ref{RDG_e} 
is a rooted decorated graph
\BE{subRDG_e} \T'=\big(\Ver',\Edg',S',\g',\fd';\si',\v_0\big) \EE
such that $(\Ver',\Edg',S',\g',\fd')$ is a decorated subgraph 
of $(\Ver,\Edg,S,\g,\fd)$ and
$$\si'=\si|_{\Ver'\sqcup\Edg'\sqcup S'}.$$
The \sf{complement}~$(\T')^c$ of a rooted decorated subgraph~$\T'$ as in~\e_ref{subRDG_e} 
is the complement of the decorated graph
$(\Ver',\Edg',S',\g',\fd')$ in $(\Ver,\Edg,S,\g,\fd)$.
This is a decorated subgraph preserved by~$\si$.\\

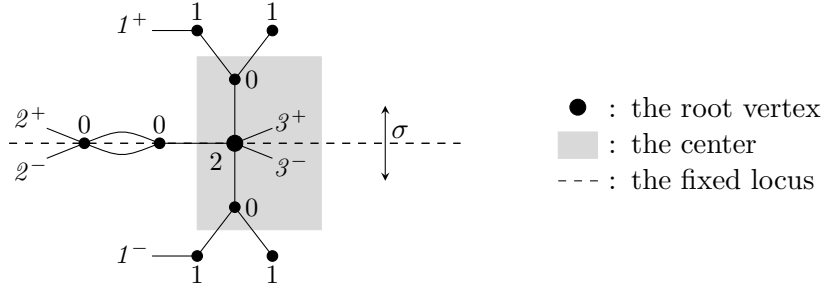
\begin{figure}
\begin{center}
\begin{tikzpicture}
 \filldraw[color=black!15!white] (1.5,1.15) rectangle (3.15,-1.15)
  (6.3,0.15) rectangle (6.82,-0.2);
 \filldraw (0,0) circle (2pt)
 (1,0) circle (2pt)
 (2,0) circle (3pt)
 (2,0.85) circle (2pt)
 (2,-0.85) circle (2pt)
 (1.5,1.5) circle (2pt)
 (2.5,1.5) circle (2pt)
 (1.5,-1.5) circle (2pt)
 (2.5,-1.5) circle (2pt)
 (6.56,0.48) circle (3pt);
 \draw[dashed] (-1,0)--(5,0);
 \draw[>=stealth,<->] (4,0.5)--(4,-0.5);
 \draw (4.2,0.15) node {$\si$};
 \draw (0,0) node[above] {\small $0$}
 (1,0) node[above] {\small $0$}
 (1.75,-0) node[below] {\small $2$}
 (2,0.85) node[right] {\small $0$}
 (2,-0.85) node[right]{\small $0$}
 (1.5,1.5) node[above]{\small $1$}
 (2.5,1.5) node[above]{\small $1$}
 (1.5,-1.5) node[below]{\small $1$}
 (2.5,-1.5) node[below]{\small $1$}
 (1,1.6) node[left] {\small $\mathit 1^+$}
 (1,-1.45) node[left] {\small $\mathit 1^-$}
 (-0.35,0.35) node[left] {\small $\mathit 2^+$}
 (-0.35,-0.35) node[left] {\small $\mathit 2^-$}
 (2.75,0.3) node {\small $\mathit 3^+$}
 (2.75,-0.3) node {\small $\mathit 3^-$};
 \draw[dashed] (-1,0)--(5,0);
 \draw (0,0)..controls (0.5,0.2)..(1,0)
 (0,0)..controls (0.5,-0.2)..(1,0)
 (1,0)--(2,0)
 (2,0)--(2,0.85)
 (2,0)--(2,-0.85)
 (1.5,1.5)--(2,0.85)
 (2.5,1.5)--(2,0.85)
 (1.5,-1.5)--(2,-0.85)
 (2.5,-1.5)--(2,-0.85);
 \draw (0,0)--(-0.5,0.2)
 (0,0)--(-0.5,-0.2)
 (2,0)--(2.5,0.2)
 (2,0)--(2.5,-0.2)
 (1.5,1.5)--(0.9,1.5)
 (1.5,-1.5)--(0.9,-1.5);
 \draw[dashed] (6.3,-0.5)--(6.9,-0.5);
 \draw (6,0) node[right,text width=5cm]
 {~~~~~\,~: the root vertex\\
 ~~~~~~\,: the center\\ 
 ~~~~~~\,: the fixed locus};
\end{tikzpicture}
\end{center}
\caption{A rooted decorated graph representing an element $\T$ of $\sT_{7,3}(\cC)^\phi$}
\label{FigDualGraph}
\end{figure}

\noindent
We define \sf{the center} of a rooted decorated graph~$\T$ as in~\e_ref{RDG_e}
to be the maximal rooted  decorated subgraph
\BE{CsubRDG_e} \T^0\equiv\big(\Ver^0,\Edg^0,S^0,\g^0,\fd^0;\si^0,\v_0\big) \EE
such that 
\BE{CsubRGDcond_e} g_a(\T)=\g(\v_0)+ g_a\big((\T^0)^c\big). \EE
The assumption that no loop in~$(\Ver,\Edg)$ contains~$\v_0$ implies that 
each rooted decorated graph has a well-defined center.
The complement of the one-vertex decorated subgraph 
\BE{Cs0ubRDG_e} \T_0\equiv\big(\{\v_0\},\eset,S_0,\g_0,\{S_0\!\to\!\v_0\};\si_0,\v_0\big) \EE
of~$\T^0$ containing~$\v_0$
consists of connected decorated subgraphs $\T'_{0;e}$
of $\T^0$ indexed by the edges $e$ in $\E_0(\T)$.
If $e$ is in~$\E_0(\T)\!\cap\!S^0$, we define $\T'_{0;e}\!=\!\eset$.
The assumption that $\val_w(\T)\!\ge\!3$ for all $w\!\in\!\Ver\!-\!\{\v_0\}$ 
and~\e_ref{CsubRGDcond_e} imply that each $\T'_{0;e}$ 
is a trivalent graph and $g(\T'_{0;e})\!=\!0$.
For each $e\inn\E_0(\T)$, let 
\BE{T0_e}
 \T^\si_{0;e} = \big(\Ver^\si_{0;e},\Edg^\si_{0;e},S^\si_{0;e},\g^\si_{0;e},\fd^\si_{0;e}\big)
\EE
be the decorated subgraph determined by $\Ver^\si_{0;e}\eq\Ver(\T'_{0;e})\!\cup\!\si(\Ver(\T'_{0;e})).$
\\

\noindent
Let $\T$ be a rooted decorated graph as in~\e_ref{RDG_e},
$\T^0$ be its center as in~\e_ref{CsubRDG_e}, and
\BE{TComplement_e}\T^c\equiv \big(\Ver^c,\Edg^c,S^c,\g^c,\fd^c\big)
=\big((\Ver^0)^c,(\Edg^0)^c,(S^0)^c,(\g^0)^c,(\fd^0)^c\big)
\equiv (\T^0)^c\EE
be its complement.
Let
$$ \E_{\bu}(\T)\equiv\Edg-\big(\Edg^0\!\sqcup\!\Edg^c\big),\quad
\E_{\bu;\C}^\si(\T)\equiv\E_{\bu}(\T)\!\cap\E^\si_\C(\T),\quad
\E_{\bu;\R}^\si(\T)\equiv\E_{\bu}(\T)\!\cap\E^\si_\R(\T)$$
be the edges separating the center $\T^0$ of $\T$ from its complement $\T^c$. Thus,
\BE{EdgDecomp_e}
\Edg=\Edg^0\sqcup\E_\bullet(\T)\sqcup\Edg^c.
\EE
The decorated subgraph $\T^c$ is the disjoint union of the connected decorated subgraphs $\T_{\bu;e}'$ indexed by the edges $e\inn\E_{\bu}(\T)$.
The assumption that $\val_w(\T)\!\ge\!3$ for all $w\!\in\!\Ver\!-\!\v_0$ implies
that each $\T'_{\bu;e}$ is a trivalent graph.
For each $e\inn\E_{\bu}(\T)$, let
\BE{Tbu_e}
\T_{\bu;e}^\si\eq \big(\Ver_{\bu;e}^\si,\Edg_{\bu;e}^\si,S_{\bu;e}^\si,\g_{\bu;e}^\si,\fd_{\bu;e}^\si\big)
\EE
be the decorated subgraph determined by 
$\Ver_{\bu;e}^\si\eq\Ver(\T_{\bu;e}')\!\cup\!\Ver(\si(\T_{\bu;e}'))$.
Define
\begin{gather}
 \lr{\cdot}\!: \,\E_\bullet(\T)\lra \E_0(\T)\qquad\quad
 \tn{by}\qquad
 \tne\in S_{0;\lr{e}}^\si\!\cup\!\{\lr{e}\}\quad\forall~e\inn\E_\bu(\T)\,,\nonumber\\
 \Aut_{\bu}(\T)\eq\big\{g\inn\Aut(\T)\!: g\!\cdot\! e\eq e~\forall 
e\inn\E_\bu(\T)\big\},\quad\!
 \Aut^*(\T)\eq\Aut(\T)\big/\Aut_{\bu}(\T)\,.\label{Aut*_e}
\end{gather}

\subsection{The strata of $\ov\fM_{g,\f}(\cC)^{\phi}$}
\label{Strata_subs}

We recall the notion of a \sf{nodal $g_0$-doublet} introduced in~\cite{RealGWsIII}.
It is
a two-component nodal symmetric surface~$(\Si,\si)$ of the~form
\BE{SymSurfDbl_e}\Si\equiv \Si_1\!\sqcup\!\Si_2
\equiv \{1\}\!\times\!\Si_0 \sqcup \{2\}\!\times\!\ov\Si_0, \qquad
\si(i,z)=\big(3\!-\!i,z\big)~~\forall~(i,z)\!\in\!\Si,\EE
where $\Si_0$ is a connected oriented, possibly nodal, genus~$g_0$ surface 
and  $\ov\Si_0$ denotes $\Si_0$ with the opposite orientation.
The arithmetic genus of a $g_0$-doublet is $2g_0\!-\!1$.
The components of~$\Si$ are ordered.\\

\noindent
If $S_1$ and $S_2$ are finite sets with a fixed bijection~$\si_S$ between them,
let $\ov\cM_{2g_0-1,(S_1\sqcup S_2,\si_S)}^{\bu}$
denote the moduli space of stable
nodal $g_0$-doublets with the first component carrying the $S_1$-marked points 
and with the marked points interchanged by the involution~$\si_S$.
Similarly, if $S$ is a finite set with an involution~$\si_S$ and~$g\!\in\!\Z^{\ge0}$ is such that 
$2g\!+\!|S|\!\ge\!3$, we denote by $\R\ov\cM_{g,(S,\si_S)}$
the Deligne-Mumford space of real genus~$g$ curves with marked points indexed by~$S$
and their type (real vs.~conjugate) characterized by~$\si_S$.
Let
\BE{LE_e}
L_i,~\bE\,\lra\,\ov\cM_{2g_0-1,(S_1\sqcup S_2,\si_S)}^{\bu},~\R\ov\cM_{g,(S,\si_S)}
\EE
be the universal tangent line bundle for the $i$-th marked point for each
$i\inn S_1\!\sqcup\!S_2,S$ and the Hodge vector bundles of holomorphic differentials, respectively.\\

\noindent
For each $e\!\in\!\E^\si_{0;\C}(\T),\E^\si_{\bu;\C}(\T)$, 
we denote~by
$$\cM_{\T,\si;e}\subset\ov\cM^\bu_{g_a(\T^\si_{0;e}),(S^\si_{0;e},\si)},
\ov\cM^\bu_{g_a(\T^\si_{\bu;e}),(S^\si_{\bu;e},\si)} $$
the open stratum consisting of the symmetric surfaces $(\Si,\si)$ with $\si$-compatible complex 
structure whose dual graph is $\T^\si_{0;e}$ or $\T^\si_{\bu;e}$ with involution~$\si$.
For each $e\!\in\!\E^\si_{0;\R}(\T),\E^\si_{\bu;\R}(\T)$, let
$$\cM_{\T,\si;e}\subset
 \R\ov\cM_{g_a(\T^\si_{0;e}),(S^\si_{0;e},\si)},\R\ov\cM_{g_a(\T^\si_{\bu;e}),(S^\si_{\bu;e},\si)}$$
be the analogous open stratum.
If~$e\inn\E_0(\T)$ and~$\T^\si_{0;e}\eq\eset$,
we define $\cM_{\T,\si;e}$ to consist of one point.\\

\noindent
We now return to the setting of Proposition~\ref{PrpNoContr}. 
If $\f$ is generic, the smooth embedded real curve~$\cC$ in~$(X,\phi)$
intersects the images of $f_1,\ldots,f_l$ at distinct non-real points
$$z_1^+(\cC)\equiv\cC\cap f_1(Y_1),\qquad\ldots,\qquad z_l^+(\cC)\equiv\cC\cap f_l(Y_l).$$
Let $z_i^-(\cC)\eq\phi(z_i^+(\cC))$.
The dual graph of an element~$[\u]$ of $\M_{\T,\f}(\cC)^\phi$
is a rooted decorated graph  which is completely determined
by the pair $\T\!=\!(\ft,\fd)$.
The center~$\T^0$ of~$\T$ corresponds to the maximal connected union~$\Si_u^0$ 
of the irreducible components 
of the domain~$\Si_u$ of each element of $\M_{\T,\f}(\cC)^\phi$ 
such that $\Si_u^0$ contains the center component~$\Si_{u;0}$.
The topological components of the complement of $\Si_{u;0}$ in~$\Si_u^0$
are the trees~$\wh\Si_{u;e}^0$ of spheres
corresponding to the graphs $\T'_{0;e}$ with $e\!\in\!\E_0(\T)\!-\!S^0$;
we define $\wh\Si_{u;e}^0\!=\!\eset$ if $e\!\in\!\E_0(\T)\!\cap\!S^0$.
The topological components of the complement of~$\Si_u^0$ in~$\Si$
are the unions~$\wh\Si_{u;e}^c$ of the irreducible 
components of~$\Si_u$ corresponding to the graphs $\T'_{\bu;e}$
with $e\!\in\!\E_{\bu}(\T)$;
the sum of their arithmetic genera is~$g\!-\!h$.
For now on,  we will not distinguish a pair $\T\!=\!(\ft,\fd)$ from 
the associated rooted decorated graph.\\

For each \hbox{$e\!\in\!\E_0(\T)$}, let
$S_e(\T)\!\subset\!S$ be the subset of the marked points of~$\T$ carried by the connected component
of the graph obtained by cutting~$\T$ at the edge~$e$ which does not contain~$\v_0$.
If $\T\!\in\!\sT_{g,l}(\cC)^\phi$ is as in~\e_ref{RDG_e},
then
\begin{gather}
\label{SeTcond_e0}
S=\big\{1^+,1^-,\ldots,l^+,l^-\big\}, \qquad
g_a(\T)=g, \qquad \g(\v_0)=h, \\
\label{SeTcond_e}
\big|S_e(\T)\big|\le1~~~\forall~e\!\in\!\E_{0;\C}^\si(\T), \quad
S_e(\T)=\eset~~\forall~e\!\in\!\E_{0;\R}^\si(\T);
\end{gather}
the last two properties hold because the points $z_1^+(\cC),\ldots,z_l^+(\cC)$ 
are non-real and distinct.
Define
$$\cC^*_\T\equiv\big\{(z_e)_{e\in\E_0(\T)}\inn\cC^{\E_0(\T)}\!:~
 z_e\!\ne\! z_{e'}~\tn{if}~e\!\ne\! e',~
 z_e\eq z_i^\pm(\cC)~\tn{if}~S_e(\T)\eq\{i^\pm\}\big\}\subset\cC^{\E_0(\T)}.$$
The involution $\si$ of $\T$ acts on $\cC^*_\T$ by
\BE{Csidfn_e}\big(\si(\mathbf{z})\big)_e=\phi(z_{\si(e)})\qquad
\forall~\mathbf{z}\eq(z_e)_{e\in\E_0(\T)}\,.\EE
We denote by $\cC^{*,\si}_\T\!\subset\!\cC^*_\T$ the fixed locus.\\

\noindent
By~\e_ref{SeTcond_e}, we can choose $\E_{0;+}^\si(\T)\!\subset\!\E_{0;\C}^\si(\T)$ such~that 
$$\E^\si_{0;\C}(\T)=\E_{0;+}^\si(\T)\sqcup\si\big(\E_{0;+}^\si(\T)\big)
\qquad\tn{and}\qquad
S_e(\T)\!\cap\!\{1^-,\ldots,l^-\}=\eset \quad\forall~e\inn\E_{0;+}^\si(\T).$$
We then choose $\E_{\bu;+}^\si(\T)\!\subset\!\E_{\bu;\C}^\si(\T)$ such that
$$\E_{\bu;\C}^\si(\T)=\E_{\bu;+}^\si(\T)\sqcup\si\big(\E_{\bu;+}^\si(\T)\big)
\qquad\tn{and}\qquad
\lr{e}\in \E_{0;+}^\si(\T)\cup\E_{0;\R}^\si(\T)\quad\forall\,e\inn\E_{\bu;+}^\si(\T)\,.$$
Let
$$\E^\si_{0;\ge}(\T)\equiv\E^\si_{0;+}(\T)\sqcup\E^\si_{0;\R}(\T),\qquad
 \E^\si_{\bu;\ge}(\T)\equiv\E^\si_{\bu;+}(\T)\sqcup\E^\si_{\bu;\R}(\T)\,.$$
We take 
\BE{tiM_e}
  \ti\M_{\T,\f}(\cC)^\phi = 
	\prod_{e\in\E_{0;\ge}^\si(\T)\sqcup\E^\si_{\bu;\ge}(\T)}\hspace{-.45in}\cM_{\T,\si;e}
\times\cC_{\T}^{*,\si}.\EE
The actions of the group $\Aut^*(\T)$ on $\E_{0}(\T)$ and~$\E_{\bu}(\T)$
and the possible changes in  the subsets
$$\E_{0;+}^\si(\T)\subset\E_{0;\C}^\si(\T)  \qquad\hbox{and}\qquad
\E_{\bu;+}^\si(\T)\subset\E_{\bu;\C}^\si(\T)$$ 
induce an action of a group $\ti\Aut^*(\T)$ on~\e_ref{tiM_e}.
The natural node-identifying immersion
\BE{Iota_e}
\io_{\T,\f}\!:\,\ti\M_{\T,\f}(\cC)^\phi\lra\M_{\T,\f}(\cC)^\phi\EE
is $\ti\Aut^*(\T)$-invariant.\\

\noindent
For each $\tne\inn\E_{0;\ge}^\si(\T),\E^\si_{\bu;\ge}(\T)$ and $e\!\in\!\E_{0}^\si(\T)$, 
denote~by
$$\pi_{\T;\tne}\!: \ti\M_{\T,\f}(\cC)^\phi\lra \cM_{\T,\si;e}
\qquad\tn{and}\qquad
\pi_{\cC;e}\!: \ti\M_{\T,\f}(\cC)^\phi\lra\cC$$
the corresponding projection maps.
For $e\!\in\!\si(\E_{0;+}^\si(\T)),\si(\E^\si_{\bu;+}(\T))$, let
$$\pi_{\T;\tne}\!=\!\pi_{\T;\si(\tne)}\!: \ti\M_{\T,\f}(\cC)^\phi\lra \cM_{\T,\si;\si(e)}.$$
For each $e\inn\E^\si_{\bu;\C}(\T)$, let
\BE{EeC}\bE_{e} \lra \ti\M_{\T,\f}(\cC)^\phi\EE
be the subbundle of $\pi^*_{\T;e}\bE$ consisting of the holomorphic differentials 
supported on~$\wh\Si_{u;e}^c$.
For each $e\inn\E^\si_{\bu;\R}(\T)$, let
\BE{EeR}\bE_{e}\!=\!\pi^*_{\T;e}\bE \lra \ti\M_{\T,\f}(\cC)^\phi\,.\EE
The involution~$\si$ lifts to an $\R$-linear isomorphism
$$\ti\si_e\!:\bE_{e}^*\lra\bE_{\si(e)}^*, \qquad
\big\{\ti\si_e(\eta)\big\}(\ka)=\ov{\eta(\fc_{\C}\!\circ\!\ka\!\circ\!\tnd\si)}\in\C
\qquad\forall~\eta\!\in\!\bE_e^*,\,\ka\!\in\!\bE_{\si(e)},$$
where $\fc_\C\!:\C\!\lra\!\C$ is the standard conjugation.
Let
\BE{ObsTCdfn_e}\Obs_\T\equiv
\bigoplus_{\tne\in\E_{\bu}(\T)}\!\!\!\bE^{*}_e\!\otimes\!\pi_{\cC;\lr{\tne}}^*TX \lra
 \ti\M_{\T,\f}(\cC)^\phi\,.\EE
The involution $\phi$ acts on $\Obs_\T$ by
$$\phi\big(\eta_e\!\otimes\! Y_e\big)=
 \big(\ti\si_e(\eta_e)\big)\!\otimes\!\big(\tnd_{z_{\lr{e}}}\!\phi(Y_e)\big)\inn
 \bE_{\si(e)}^*\!\!\otimes\!\pi^*_{\cC;\si\lr{e}}TX\quad
 \forall\,e\inn\E_\bu(\T),\,\eta_e\inn\bE^*_e,\,Y_e\inn\pi^*_{\cC;\lr{e}}TX.$$
In other words,
$$\big\{\phi(\eta)\big\}\big((\ka_e)_{e\in\E_\bu(\T)}\big)
=\big( \tnd_{z_{\si(\lr{e})}}\!\phi
\big(\eta(\fc_{\C}\!\circ\!\ka_e\!\circ\!\tnd\si)\big)\big)_{e\in\E_\bu(\T)}
\quad\forall\,\eta\!\in\!\Obs_\T,\,\ka_e\!\in\!\E_e,\,e\!\in\!\E_{\bu}(\T).$$
We take the \sf{obstruction bundle} 
\BE{Obs_e}\Obs_\T^{\phi}\subset\Obs_\T\lra\ti\M_{\T,\f}(\cC)^\phi\EE
to be the fixed locus under this action.\\

Let $\ti\u\!\in\!\ti\M_{\T,\f}(\cC)^\phi$ with $\u\!=\!\io_{\T,\f}(\ti\u)$ 
as in~\e_ref{u_e}. 
For each $e\!\in\!\E_{\bu}(\T)$,
$u$ is constant on the support of the elements of~$\bE_e|_{\ti\u}$ 
(which is contained in~$\wh\Si_{u;e}^c$) and 
sends it to $\pi_{\cC;\lr{e}}(\ti\u)\!\in\!\cC$.
Thus, we can define 
\BE{Thudfn_e}\begin{split}
&\hspace{.8in}\Th_{\ti\u}\!: \Ga\big(\Si;T^*\Si_u^{0,1}\!\otimes\!u^*(TX,J)\big)
\lra\Obs_{\T}\big|_{\ti\u}
\quad\hbox{by}\quad\\
&\big\{\Th_{\ti\u}(\eta)\big\}\big((\ka_e)_{e\in\E_\bu(\T)}\big)
=\bigg(\frac{\fI}{2\pi}\!\int_{\wh\Si_{u;e}^c}\!\!\!\!
\ka_e\!\wedge\!\!\eta\!\bigg)_{\!\!e\in\E_\bu(\T)}
\qquad\forall~\ka_e\!\in\!\bE_e|_{\ti\u},~e\!\in\!\E_\bu(\T)\,.
\end{split}\EE
If $\eta$ is $(\phi,\si_u)$-invariant, then $\Th_{\ti\u}(\eta)\!\in\!\Obs_\T^{\phi}|_{\ti\u}$.
A Ruan-Tian deformation~$\nu$ for maps from genus~$g$ curves with $2l$ markings
into~$(X,J)$ as in \cite[Section~2]{GWsAbsRel} determines elements 
$$\nu_{\ti\u}\in\Ga\big(\Si;T^*\Si_u^{0,1}\!\otimes\!u^*(TX,J)\big)
\qquad\hbox{and}\qquad
\ov\nu_{\T}(\ti\u)\equiv \Th_{\ti\u}(\nu_{\ti\u})\in \Obs_{\T}\big|_{\ti\u}.$$
If $\nu\inn\A^\phi_{g,l}(J)$,
then $\ov\nu_{\T}$ takes values in~$\Obs_{\T}^{\phi}$.
The homomorphism $\nu\!\lra\!\ov\nu_{\T}$ surjects onto the
subspace of $\ti\Aut^*(\T)$-invariant sections.\\

For each edge $e\eq(w_1,w_2)$ in $\Edg$,
we denote by $\ti L_{e}$ the tensor product of the two complex line bundles corresponding 
to the tangent spaces of the irreducible components $\Si_{u;w_1}$ and $\Si_{u;w_2}$ 
of~$\Si_u$ at the node~$e$.
Let
\BE{FT_e}
\pi_{\ti\F\T}\!:\, \ti\F\T=\Big(\bigoplus_{e\in\Edg}\!\!\!
  \ti L_{e}\Big)\Big/\Aut_{\bu}(\T) \lra \ti\M_{\T,\f}(\cC)^\phi.\EE
We fix an $\ti\Aut^*(\T)$-invariant and involution-invariant metric on $\ti\F\T$
and for each $\de\!\in\!\R$ define
$$\ti\F\T_{\de}=\big\{\up\!\in\!\ti\F\T\!:\,|\up|\!<\!\de\big\}.$$
The involution $\si$ acts on $\ti\F\T$ by
$$\big(\si\big((\up_{e'})_{e'\in\Edg}\big)\big)_e=
 \tnd\si(\up_{\si(e)})   \qquad\forall~e\!\in\!\Edg\,.$$
We denote the corresponding fixed locus by $\ti\F\T^{\si}$.\\

If $\T\!=\!(\ft,\fd)$ with $\ft$ basic, then $\Edg\!=\!\E_\bu(\T)$ and
$$\ti\F\T= \bigoplus_{e\in\E_\bu(\T)}\!\!\!\!
 \pi^*_{\T;e}L_e\!\otimes\!\pi^*_{\cC;\lr{e}}T\cC\,.$$
Define
\BE{DT_e}
\cD_\T\!:\ti\F\T\lra\Obs_\T,\quad
 \big\{\cD_\T\big((\up_e\!\otimes\!\up_{\lr{e}})_{e\in\E_\bu(\T)}\big)\big\}
 \big((\ka_e)_{e\in\E_\bu(\T)}\big)
 =\big(\ka_e(\up_e)\up_{\lr{e}}\big)_{e\in\E_\bu(\T)},\EE
with $\ka_e\inn\bE_e$.
The homomorphism~$\cD_\T$ restricts to a homomorphism
$$\cD_\T^{\phi}\!: \ti\fF\T^{\si}\lra \Obs_\T^{\phi}$$
between the fixed loci of the two bundles.\\

The bundles $\ti\F\T^\si$ and $\Obs^\phi_\T$ form 
a deformation-obstruction complex for 
$\ov\M_{g,\f}(X,B;J)^\phi$ over $\ov\M_{\T,\f}(\cC)^\phi$.
The choice of real orientation on $(X,\om,\phi)$ orients this moduli space and thus the total space of the vector bundle
\BE{ObsoverFT_e} \pi^*_{\ti\F\T^\si}\Obs^\phi_\T\lra\ti\F\T^\si\,,\EE
where $\pi_{\ti\F\T^\si}\!:\ti\F\T^\si\!\lra\!\ti\M_{\T,\f}(\cC)^\phi$
is the bundle projection map.
Since the dimension of the total space of $\ti\F\T^\si$ and the rank of $\Obs_\T^\phi$ are the same,
every transverse zero of a section $\Psi$ of~\e_ref{ObsoverFT_e} thus has a well defined sign.
The next structural proposition will be proved in Section~\ref{SecAnalytic}.

\begin{prp}\label{PrpSection}
Let $(X,\om,\phi)$, $B$, $g,l$, and $\f$ be as in Theorem~\ref{main_thm}, 
$J\inn\cJ^\phi_{\reg}(g,B)$, $\cC$ be an element of $\cM_{h,\f}^{\phi}(B;J)$,
$\nu\!\in\!\A_{g,l}^{\phi}(J)$, $\T\!\equiv\!(\ft,\fd)\!\in\!\sT_{g,l}(\cC)^\phi$, and
$K\!\subset\!\M_{\T,\f}(\cC)^\phi$ be a precompact open subset.
\begin{enumerate}[label=(\arabic*),leftmargin=*]

\item\label{notbasic_it} If $\ft$ is not basic and $\nu$ is generic,
there exist $\de_{\nu}(K)\!\in\!\R^+$ and an open neighborhood 
$U_{\nu}(K)$ of~$K$ in $\X_{g,\f}(X,B)^\phi$ such that 
\BE{esetinter_e} \ov\M_{g,\f}(X,B;J,t\nu)^\phi\cap U_{\nu}(K)=\eset
\qquad\forall~t\in\big(\!-\!\de_{\nu}(K),\de_{\nu}(K)\big).\EE

\item\label{basic_it}  If $\ft$ is basic, 
there exist $\de_{\nu}(K)\inn\R^+$, a family
$$U_{\nu;\de}(K)\subset\X_{g;\f}(X,B)^{\phi} , 
\qquad \de\!\in\!\big(0,\de_{\nu}(K)\big),$$
of neighborhoods of $K$, continuous families of continuous $\ti\Aut^*(\T)$-invariant maps
$$\Phi_{\T;t\nu}\!: 
\ti\F\T^{\si}_{\de_{\nu}(K)}|_{\io_{\T,\f}^{-1}(K)}\lra U_{\nu;\de_{\nu}{K}}(K),  
\quad \ve_{\T;t\nu}\!:\ti\F\T^{\si}\lra\R^+, \quad 
t\in\big(-\!\de_{\nu}(K),\de_{\nu}(K)\big),$$
and a continuous family of  $\ti\Aut^*(\T)$-sections
\BE{PsiDfn_e} \Psi_{\T;t\nu}\in
 \Ga\big(\ti\F\T^{\si}_{\de_{\nu}(K)}|_{\io^{-1}_{\T,\f}(K)};\,
 \pi^*_{\ti\F\T^\si}\Obs_\T^{\phi}\,\big), 
\quad t\in\big(-\!\de_{\nu}(K),\de_{\nu}(K)\big),\EE
such that $\Phi_{\T;t\nu}$ is a degree $|\ti\Aut^*(\T)|$ covering of its image,
\begin{gather}
\label{BdCond_e}
\Phi_{\T;0}|_{\io_{\T,\f}^{-1}(K)}=\io_{\T,\f}, \qquad
\ve_{\T;0}|_{\io^{-1}_{\T,\f}(K)}=0,\\
\label{PhiEstimate_e}
\Big\|\Psi_{\T;t\nu}(\up)-\Big(\cD^{\phi}_{\T}(\up)
\!+\!t\Bar\nu_{\T}\big(\pi_{\ti\F\T^{\si}}(\up)\big)\Big)\Big\|
 \le \ve_{\T;t\nu}(\up)\big(\,|\up|\!+\!|t|\big),\\
\label{UnuK_e}
\Phi_{\T;t\nu}\big(\Psi_{\T;t\nu}^{-1}(0)\!\cap\!\ti\F\T^{\si}_{\de}\big)
=\ov\M_{g,\f}(X,B;J,t\nu)^\phi\!\cap\!U_{\nu;\de}(K)\quad\forall~
t\!\in\!(-\de,\de),~\de\!\in\!\big(0,\de_{\nu}(K)\big).
\end{gather}
If in addition $\nu$ is generic, $t\!\in\!(-\de,\de)$,
and $\de\!\in\!(0,\de_{\nu}(K))$, then  
\BE{PhiZeros_e}
\Phi_{\T;t\nu}\!:\Psi_{\T;t\nu}^{-1}(0)\!\cap\!\ti\F\T^{\si}_{\de}
\lra \M^*_{g,\f}(X,B;J,t\nu)^\phi\!\cap\!U_{\nu;\de}(K)\EE
preserves the signs of all points.
\end{enumerate}
\end{prp}

\section{Proofs of Propositions~\ref{PrpNoContr}-\ref{PrpSamPar}}
\label{SecPfPrps}

\noindent
We will next deduce the three propositions of Section~\ref{SecNotation} from Proposition~\ref{PrpSection}
following the principles laid out in \cite[Sections~3]{g2n2and3}
which relate stratawise degenerate contributions to zeros of affine maps between vector bundles.
These principles readily apply with the manifold~$\cM$ 
(accidentally denoted by~$X$ in Lemma~3.5(3) and Corollary~3.6(5) in~\cite{g2n2and3})
replaced by an orbifold
(in the sense specified after \cite[Definition~2.10]{g0pr}).
All bundles in \cite[Sections~3]{g2n2and3} are assumed to be  complex 
so that all transverse isolated zeros of bundles sections are canonically oriented.
While the former is no longer the case, the latter is;
see the paragraph preceding Proposition~\ref{PrpSection}.
The rank over~$\C$ appearing in the statements of 
\cite[Sections~3]{g2n2and3} should be replaced by half the rank over~$\R$.
We will apply the reasoning behind \cite[Corollary~3.6]{g2n2and3}
with 
$$\cM=\ti\M_{\T,\f}(\cC)^\phi, \qquad
F=F^-\!=\ti{F}^-=\ti\F\T^{\si},  \qquad \O=\O^-=\Obs_\T^{\phi}, $$
$F^+\!\subset\!F$ and $\O^+\!\subset\!\O$ being the trivial (zero) subbundles,
and $\al\!=\!\cD_{\T}^{\phi}$.
We continue with the assumptions and the notation as 
in Propositions~\ref{PrpNoContr} and~\ref{PrpSection}
and assume that $\nu$ is generic.

\subsection{The regularity of the boundary strata}
\label{RegStrataContr_subs}

The crucial difference between the main boundary strata, i.e.~those corresponding
to $\T\eq(\ft,\fd)$ with $\ft$ basic, and the remaining boundary strata
is that the latter are $\dbar_J$-hollow in the sense of
\cite[Definition~3.11]{g2n2and3} and thus automatically do not contribute
to the number~\e_ref{GWdfn_e}.
In order to complete the proof of Proposition~\ref{PrpNoContr},
we show that  the main boundary strata with $\fd\!\neq\!\fd_0(\ft)$
do not contribute either;
the reason is more delicate in this case.

\begin{proof}[{\bf\emph{Proof of Proposition~\ref{PrpNoContr}}}]
Let $\T\eq(\ft,\fd)$.\\

Suppose $\ft$ is not basic.
For the purposes of Definition~\ref{DfnContri}, we then take $K_{\nu}\!=\!\eset$.
Given a precompact open subset $K\!\subset\!\M_{\T,\f}(\cC)^\phi$,
let $\de_{\nu}(K)\!\in\!\R^+$ and $U_{\nu}(K)\!\subset\!\X_{g,\f}(X,B)^\phi$
be as in Proposition~\ref{PrpSection}\ref{notbasic_it}.
Let $U\!\subset\!U_{\nu}(K)$ be an open neighborhood of~$K$.
By~\e_ref{esetinter_e}, 
$$\M_{g,\f}^*(X,B;J,t\nu)^\phi\cap U
\subset \ov\M_{g,\f}(X,B;J,t\nu)^\phi\cap U_{\nu}(K)
=\eset.$$
This establishes Proposition~\ref{PrpNoContr} for $\T\eq(\ft,\fd)$
with $\ft$ not~basic.\\

\noindent
Suppose $\ft$ is basic. In particular,
$$\E_0(\T)=\E_{\bu}(\T)=\Edg, \quad \ti\F\T^{\si}=\ti\fF\T^{\si}, \quad
\dim\,\ti\M_{\T,\f}(\cC)^\phi+\rk\,\ti\F\T^{\si}=\rk\,\Obs_\T^{\phi}\,.$$
We assume that~$\T$ is given by~\e_ref{RDG_e} and satisfies~\e_ref{SeTcond_e0} 
and~\e_ref{SeTcond_e}.
For each $e\!\in\!\Edg$, we denote by $e/w_0\!\in\!e$ the vertex different from $w_0\!\in\!e$.
Let 
$$S_0(\T)=\big\{i\!\in\!\{1,\ldots,l\}\!:\,i^+\!\in\!\d^{-1}(w_0)\big\}$$
be the pairs of marked points carried by the non-contracted curve in $\ti\M_{\T,\f}(\cC)^\phi$.\\

\noindent
The vector bundles $\ti\F\T^{\si}$ and $\Obs_\T^{\phi}$,
the bundle homomorphism~$\cD_{\T}^{\phi}$, and the bundle section~$\ov\nu_{\T}$
naturally extend over the compactification
\BE{whMdfn_e}\wh\M_{\T,\f}(\cC)^\phi= 
\prod_{e\in\E^\si_{\bu;+}(\T)}\hspace{-.2in}\ov\cM_{2\g(e/w_0)-1,(S_{\bu;e}^{\si},\si)}^{\bu}
\times \prod_{e\in\E^\si_{\bu;\R}(\T)}\hspace{-.2in}\R\ov\cM_{\g(e/w_0),(S_{\bu;e}^{\si},\si)}
\times\wh\cC_{\T}\EE
of $\ti\M_{\T,\f}(\cC)^\phi$, where $\wh\cC_{\T}$ is the closure of
$$\big\{\big((z_i^+(\cC),z_i^-(\cC))_{i\in S_0(\T)},(z_e)_{e\in\E_{\bu}(\T)}\big)\!:\,
(z_e)_{e\in\E_{\bu}(\T)}\inn\cC_{\T}^*\big\}$$
in the corresponding real moduli space of points on~$\cC$.
We denote these extensions in the same way.
The compactification~\e_ref{whMdfn_e} is a union of the quotients 
of the spaces $\ti\M_{\T',\f}(\cC)^\phi$ by subgroups of $\ti\Aut^*(\T')$
for some rooted decorated graphs~$\T'$.
The restriction of~$\ov\nu_{\T}$ to the quotient of $\ti\M_{\T',\f}(\cC)^\phi$ lifts
 to~$\ov\nu_{\T'}$
(as the sections $\ov\nu_{\T}$ and $\ov\nu_{\T'}$ 
are pullbacks from the closure of $\M_{\T,\f}(\cC)^\phi$ 
in~$\ov\fM_{g,\f}(\cC)^{\phi}$).\\

\noindent
The homomorphism~$\cD_{\T}^{\phi}$ does not vanish over $\ti\M_{\T,\f}(\cC)^\phi$.
The pairings of tangent and cotangent vectors in~\e_ref{DT_e} 
correspond to algebraic bundle sections.
Thus, the homomorphism~$\cD_{\T}^{\phi}$ is a regular polynomial (linear map in this case)
in the sense of \cite[Definition~3.9]{g2n2and3}.
By the first two parts of the proof of \cite[Lemma~3.10]{g2n2and3}, the zero set of 
the affine bundle~map
$$\psi_{\T;\nu}\!: \ti\F\T^{\si}\lra\Obs_\T^{\phi}, \qquad
\psi_{\T;\nu}(\up)=\cD_{\T}^{\phi}(\up)+\ov\nu_{\T}(\u)
\quad\forall~\up\!\in\!\ti\F\T^{\si}|_{\u},~\u\!\in\!\ti\M_{\T,\f}(\cC)^\phi,$$
thus has finitely many zeros.
The continuous extension of this map over $\wh\M_{\T,\f}(\cC)^\phi$
does not vanish over the complement of~$\ti\M_{\T,\f}(\cC)^\phi$.
By the third part of the proof of \cite[Lemma~3.10]{g2n2and3}, 
the signed cardinality of the affine bundle map
$$\psi_{\T;\vp}\!: \ti\F\T^{\si}\lra\Obs_\T^{\phi}, \qquad
\psi_{\T;\vp}(\up)=\cD_{\T}^{\phi}(\up)+\vp(\u)
\quad\forall~\up\!\in\!\ti\F\T^{\si}|_{\u},~\u\!\in\!\wh\M_{\T,\f}(\cC)^\phi,$$
does not depend on a generic choice of the section~$\vp$ of $\Obs_\T^{\phi}$
over~$\wh\M_{\T,\f}(\cC)^\phi$.
In this case, the cobordism in proof of \cite[Lemma~3.10]{g2n2and3} may cross 
the fibers of $\ti\F\T^{\si}$ over the codimension-one boundary strata of~$\wh\M_{\T,\f}(\cC)^\phi$.
However, the orientation on such a cobordism extends across the zeros in these fibers
based on the same considerations as above Proposition~\ref{PrpSection}.
We denote the  signed cardinality of $\psi_{\T;\vp}^{-1}(0)$ by~$N(\cD_{\T}^{\phi})$.\\

\noindent
By the previous paragraph, \e_ref{PhiEstimate_e}, and the proofs of 
Lemma~3.5 and Corollary~3.6 in~\cite{g2n2and3},
there exist a compact subset $\ti{K}_{\T;\nu}\!\subset\!\ti\M_{\T,\f}(\cC)^\phi$
and $C_{\T;\nu}\!\in\!\R^+$ with the following property.
If $\ti{K}\!\subset\!\ti\M_{\T,\f}(\cC)^\phi$ is a precompact open subset
containing~$\ti{K}_{\T;\nu}$, then there exists $\de_{\nu}(\ti{K})\!\in\!\R^+$ so~that 
the bundle~map
$$\Psi_{\T;t\nu}\!: \ti\F\T^{\si}_{\de_{\nu}(\ti{K})}|_{\ti{K}}\lra \Obs_\T^{\phi}$$ 
is defined and 
\BE{PsiTnum_e}
\Psi_{\T;t\nu}^{-1}(0)\subset  \ti\F\T^{\si}_{C_{\T;\nu}t}|_{\ti{K}_{\T;\nu}}\,,
\quad
{}^\pm\big|\Psi_{\T;t\nu}^{-1}(0)\big|=N\big(\cD_{\T}^{\phi}\big)
\qquad\forall~t\!\in\!\big(0,\de_{\nu}(\ti{K})\big).\EE
Let $K_{\nu}\!\subset\!\M_{\T,\f}(\cC)^\phi$ be a compact subset such~that 
$$\ti{K}_{\T;\nu}\subset \ti{K}_{\nu}\!\equiv\!\io_{\T;\f}^{-1}(K_{\nu})$$ 
and set
\BE{BasicContr_e}\Cntr_\f^\phi\big(\M_{\T,\f}(\cC)^\phi\big)=
\frac{1}{|\ti\Aut^*(\T)|}N\big(\cD_{\T}^{\phi}\big)\,.\EE
We verify below that $K_{\nu}$ and the number~\e_ref{BasicContr_e} satisfy 
the conditions of Definition~\ref{DfnContri}.\\

Let $K\!\subset\!\M_{\T,\f}(\cC)^\phi$ be a precompact open subset containing $K_{\nu}$,
$\ti{K}\!=\!\io^{-1}_{\T,\f}(K)$,  and
$$\de_{\nu}(K)\in\big(0,\de_{\nu}(\ti K)\big) \qquad\hbox{and}\qquad
U_{\nu}(K)\equiv U_{\nu;\de_{\nu}(K)}(K)$$ 
be as in Proposition~\ref{PrpSection}\ref{basic_it},
and $U\!\subset\!U_{\nu}(K)$ be a neighborhood of~$K$.
Since $K_{\nu}$ is compact, there exists $\ep_{\nu}(U)\!\in\!(0,\de_{\nu}(K))$ 
such~that 
$$\ti\F\T^{\si}_{C_{\T;\nu}\ep_{\nu}(U)}\big|_{\ti{K}_{\nu}}
\subset \Phi_{\T;t\nu}^{-1}(U)\,.$$
Suppose  $t\!\in\!(0,\ep_{\nu}(U))$.
By~\e_ref{UnuK_e}, the first statement in~\e_ref{PsiTnum_e}, and the last condition,
\begin{equation*}\begin{split}
\Psi_{\T;t\nu}^{-1}(0)\!\cap\!\Phi_{\T;t\nu}^{-1}(U)
&\subset \Psi_{\T;t\nu}^{-1}(0)\!\cap\!\Phi_{\T;t\nu}^{-1}\big(U_{\nu}(K)\big)
=\Psi_{\T;t\nu}^{-1}(0)\!\cap\!\ti\F\T^{\si}_{\de_{\nu}(K)}|_{\ti{K}}\\
&\subset \Psi_{\T;t\nu}^{-1}(0)\!\cap\! \ti\F\T^{\si}_{C_{\T;\nu}\ep_{\nu}(U)}|_{\ti{K}_{\T;\nu}}
\subset \Psi_{\T;t\nu}^{-1}(0)\!\cap\!\Phi_{\T;t\nu}^{-1}(U).
\end{split}\end{equation*}
This implies that 
$$\Psi_{\T;t\nu}^{-1}(0)\!\cap\!\Phi_{\T;t\nu}^{-1}(U)
=\Psi_{\T;t\nu}^{-1}(0)\!\cap\!\ti\F\T^{\si}_{\de_{\nu}(K)}|_K\,.$$
Combining this conclusion with~\e_ref{PsiTnum_e}, we obtain 
$$ {}^\pm\big|\Psi_{\T;t\nu}^{-1}(0)\!\cap\!\Phi_{\T;t\nu}^{-1}(U)\big|
= {}^\pm\big|\Psi_{\T;t\nu}^{-1}(0)\!\cap\!\ti\F\T^{\si}_{\de_{\nu}(K)}|_K  \big|
=N\big(\cD_{\T}^{\phi}\big)\,.$$
The last two statements and~\e_ref{PhiZeros_e} give 
the first statement of Proposition~\ref{PrpNoContr} for $\T\eq(\ft,\fd)$
with $\ft$~basic and the contribution given by~\e_ref{BasicContr_e}.\\

\noindent
It remains to show that $N(\cD_{\T}^{\phi})\!=\!0$ if $\T\eq(\ft,\fd)$
with $\ft$ basic and $\fd\!\neq\!\fd_0(\ft)$.
The last condition implies that at least one conjugate pair $(z_i^+,z_i^-)$ of  
marked points is carried by the contracted components of the elements in
$\ti\M_{\T,\f}(\cC)^\phi$.
For $e\inn \E^{\si}_{\bu;\ge}(\T)$, let 
$$S_{\bu;e}'^{\si}\equiv\big\{e,\si(e)\big\} \subset S_{\bu;e}^{\si}$$
be the complement of the markings $\{1^+,1^-,\ldots,l^+,l^-\}$.
Let 
$$\wh\M_{\T,\f}'(\cC)^\phi=
\prod_{e\in\E^\si_{\bu;+}(\T)}\hspace{-.2in}\ov\cM_{2\g(e/w_0)-1,(S_{\bu;e}'^{\si},\si)}^{\bu}
\times \prod_{e\in\E^\si_{\bu;\R}(\T)}\hspace{-.2in}\R\ov\cM_{\g(e/w_0),(S_{\bu;e}'^{\si},\si)}
\times\wh\cC_{\T}\,.$$
We denote by 
$$\cD_{\T}'^{\phi}\!: \ti\F'\T^{\si}\lra\Obs_\T'^{\phi}$$
the vector bundle homomorphism over  $\wh\M_{\T,\f}'(\cC)^\phi$ defined analogously 
to the vector bundle homomorphism $\cD_{\T}^{\phi}$ over $\wh\M_{\T,\f}(\cC)^\phi$.\\

\noindent
Let
$$\ff\!: \wh\M_{\T,\f}(\cC)^\phi\lra\wh\M_{\T,\f}'(\cC)^\phi$$
be the morphism dropping the points  marked by $1^+,1^-,\ldots,l^+,l^-$
that are carried by the contracted components.
Thus,
$$\ti\F\T^{\si}\big|_{\ti\M_{\T,\f}(\cC)^\phi}=
\ff^*\ti\F'\T^{\si}\big|_{\ti\M_{\T,\f}(\cC)^\phi}, \quad
\Obs_\T^{\phi}=\ff^*\Obs_\T'^{\phi}\,,\quad
\cD_{\T}^{\phi}\big|_{\ti\M_{\T,\f}(\cC)^\phi}
=\ff^*\cD_{\T}'^{\phi}\big|_{\ti\M_{\T,\f}(\cC)^\phi}\,.$$
This implies that $N(\cD_{\T}^{\phi})\!=\!N(\ff^*\cD_{\T}'^{\phi})$,
with the second number defined analogous to the former.
If $\fd\!\neq\!\fd_0(\ft)$, then
$$\dim\,\wh\M_{\T,\f}'(\cC)^\phi+\rk\,\ti\F'\T^{\si}
<\dim\,\ti\M_{\T,\f}(\cC)^\phi+\rk\,\ti\F\T^{\si}
=\rk\,\Obs_\T^{\phi}=\rk\,\Obs_\T'^{\phi}\,.$$
Thus, there exists a section $\vp'$ of $\Obs_\T'^{\phi}$ so that the affine bundle~map
$$\psi_{\T;\vp}'\!: \ti\F'\T^{\si}\lra\Obs_\T'^{\phi}, \qquad
\psi_{\T;\vp}'(\up)=\cD_{\T}'^{\phi}(\up)+\vp(\u)
\quad\forall~\up\!\in\!\ti\F'\T^{\si}|_{\u},~\u\!\in\!\wh\M_{\T,\f}'(\cC)^\phi,$$
over $\wh\M_{\T,\f}'(\cC)^\phi$ does not vanish.
This implies that the affine bundle~map
$$\ff^*\psi_{\T;\vp}'\!: \ti\F\T^{\si}\lra\Obs_\T^{\phi}, \quad
\big\{\ff^*\psi_{\T;\vp}'\big\}(\up)=\big\{\ff^*\cD_{\T}'^{\phi}\big\}(\up)+\big\{\ff^*\vp\}(\u)
~~\forall~\up\!\in\!\ti\F\T^{\si}|_{\u},~\u\!\in\!\wh\M_{\T,\f}(\cC)^\phi,$$
does not vanish and so 
$$N(\cD_{\T}^{\phi})=N(\ff^*\cD_{\T}'^{\phi})=0.$$
This establishes the last claim of Proposition~\ref{PrpNoContr}.
\end{proof}

\subsection{The contributions of the main boundary strata}
\label{MainStrataContr_subs}

We continue with the setup and notation in the proof of Proposition~\ref{PrpNoContr}
and~set
$$\wch\M_{\T,\f}(\cC)^\phi= 
\prod_{e\in\E^\si_{\bu;+}(\T)}\hspace{-.2in}\ov\cM_{2\g(e/w_0)-1,(S_{\bu;e}^{\si},\si)}^{\bu}
\times \prod_{e\in\E^\si_{\bu;\R}(\T)}\hspace{-.2in}\R\ov\cM_{\g(e/w_0),(S_{\bu;e}^{\si},\si)}
\times\cC_{\T},$$
where $\cC_{\T}\!\subset\!\cC^{\E_{\bu}(\T)}$ is the closure of $\cC_{\T}^*$.
Thus, 
$$\cC_{\T}= \big(\cC^{\phi}\big)^{\E_{\bu;\R}^{\si}(\T)}\times\cC_{\T;\C}\,,$$
where $\cC_{\T;\C}\!\subset\!\cC^{\E_{\bu;\C}^{\si}(\T)}$ is the fixed locus
of the involution given by~\e_ref{Csidfn_e} with $\E_0(\T)$ replaced by~$\E_{\bu;\C}^{\si}(\T)$.
We denote by 
$$\wch\cD_{\T}^{\phi}\!: \wch\F\T^{\si}\lra\wch\Obs_\T^{\phi}$$
the vector bundle homomorphism over  $\wch\M_{\T,\f}(\cC)^\phi$ defined analogously 
to the vector bundle homomorphism $\cD_{\T}^{\phi}$ over $\wh\M_{\T,\f}(\cC)^\phi$.\\

Let 
$$q\!:\wh\M_{\T,\f}(\cC)^\phi\lra \wch\M_{\T,\f}(\cC)^\phi$$
be the map induced by the natural projection $\wh\cC_{\T}\!\lra\!\cC_{\T}$.
Since
$$\ti\F\T^{\si}\big|_{\ti\M_{\T,\f}(\cC)^\phi}=
q^*\wch\F\T^{\si}\big|_{\ti\M_{\T,\f}(\cC)^\phi}, \quad
\Obs_\T^{\phi}=q^*\wch\Obs_\T^{\phi}\,,\quad
\cD_{\T}^{\phi}\big|_{\ti\M_{\T,\f}(\cC)^\phi}
=q^*\wch\cD_{\T}^{\phi}\big|_{\ti\M_{\T,\f}(\cC)^\phi},$$
and $q$ is the identity on $\ti\M_{\T,\f}(\cC)^\phi$,
\BE{eqnums_e}N\big(\cD_{\T}^{\phi}\big)=N\big(q^*\wch\cD_{\T}^{\phi}\big)
=N\big(\wch\cD_{\T}^{\phi}\big).\EE
Since the contributions of the boundary strata of Propositions~\ref{PrpOppPar} and~\ref{PrpSamPar}
are given by~\e_ref{BasicContr_e}, it remains to determine the number~$N(\wch\cD_{\T}^{\phi})$.

\begin{proof}[{\bf\emph{Proof of Proposition~\ref{PrpOppPar}}}]
Let 
$$\ff\!:\wch\M_{\T,\f}(\cC)^\phi\lra
\wch\M_{\T,\f}'(\cC)^\phi\!\equiv\!
\prod_{e\in\E^\si_{\bu;+}(\T)}\hspace{-.2in}\ov\cM_{2\g(e/w_0)-1,(S_{\bu;e}^{\si},\si)}^{\bu}
\times \prod_{e\in\E^\si_{\bu;\R}(\T)}\hspace{-.2in}\R\ov\cM_{\g(e/w_0),(S_{\bu;e}^{\si},\si)}
\times\cC_{\T;\C}$$
be the natural projection.
We denote by 
$$\wch\cD_{\T}'^{\phi}\!: \wch\F'\T^{\si}\lra\wch\Obs_\T'^{\phi}$$
the vector bundle homomorphism over  $\wch\M_{\T,\f}'(\cC)^\phi$ defined analogously 
to the vector bundle homomorphism $\cD_{\T}^{\phi}$ over $\wh\M_{\T,\f}(\cC)^\phi$,
but for each $e\!\in\!\E^\si_{\bu;\R}(\T)$ we replace 
the factor $\pi^*_{\cC;\lr{e}}T\cC$ appearing in the definition
of $\ti\F\T\!=\!\ti\fF\T$ in~\e_ref{FT_e} by the trivial complex line bundle
(with the standard conjugation)
and the factor $\pi_{\cC;\lr{\tne}}^*TX$ in~\e_ref{ObsTCdfn_e}
by the trivial rank~3 complex vector bundle.
Since $\cC^{\phi}$ is a disjoint union of circles and the restriction of every orientable vector 
bundle to a circle is trivial,
$$\wch\F\T^{\si}=
\ff^*\wch\F'\T^{\si}\big|_{\ti\M_{\T,\f}(\cC)^\phi}, \quad
\wch\Obs_\T^{\phi}=\ff^*\wch\Obs_\T'^{\phi}\,,\quad
\wch\cD_{\T}^{\phi}=\ff^*\wch\cD_{\T}'^{\phi}$$
and thus $N(\wch\cD_{\T}^{\phi})\!=\!N(\ff^*\wch\cD_{\T}'^{\phi})$.
If $\E_{\bu;\R}^{\si}(\T)\!\neq\!\eset$, then
$$\dim\,\wch\M_{\T,\f}'(\cC)^\phi+\rk\,\wch\F'\T^{\si}
<\dim\,\ti\M_{\T,\f}(\cC)^\phi+\rk\,\ti\F\T^{\si}
=\rk\,\Obs_\T^{\phi}=\rk\,\wch\Obs_\T'^{\phi}.$$
As in the last paragraph of the proof of Proposition~\ref{PrpNoContr},
this implies that $N(\ff^*\wch\cD_{\T}'^{\phi})\!=\!0$.
Proposition~\ref{PrpOppPar} now follows from~\e_ref{BasicContr_e} and~\e_ref{eqnums_e}.
\end{proof}

\begin{proof}[{\bf\emph{Proof of Proposition~\ref{PrpSamPar}}}]
We continue with the setup and notation in the proof of Proposition~\ref{PrpOppPar}.
We now assume that $\ft$ is a basic topological type, $\T\eq(\ft,\fd_0(\ft))$,
and $\E_{\bu;\R}^{\si}(\T)\!=\!\eset$.  
Let $m\!=\!|\E_{\bu;+}^{\si}(\T)|$ so that 
\BE{tiAutAut_e} \big|\ti\Aut^*(\T)\big|=2^m \big|\Aut^*(\T)\big|,\EE
with the groups $\Aut^*(\T)$ and $\ti\Aut^*(\T)$ as in Section~\ref{Strata_subs}.
With the notation as in~\e_ref{bELdfn_e} and~\e_ref{pidfn_e}, we define
the bundle homomorphism
$$\cD_{g'}\!: \pi_{g'}^*L\!\otimes\!\pi_{\cC}^*T\cC\lra
\pi_{g'}^*\bE^*\!\otimes\!\pi_{\cC}^*TX$$ 
over  $\ov\cM_{g',1}\!\times\!\cC$  analogously to~\e_ref{DT_e}.\\

\noindent
For each $e\!\in\!\E^\si_{\bu;+}(\T)$, let
$$\ov\cM_{\T;e}^{\bu}(\cC)
=\ov\cM_{2\g(e/w_0)-1,(\{e,\si(e)\},\si)}^{\bu}
\times\big\{(z^+,z^-)\!\in\!\cC^2\!:\,z^+\!=\!\phi(z^-)\big\}$$
and denote by 
$$\pi_{\T;e},\pi_{\cC;+},\pi_{\cC;-}\!: \ov\cM_{\T;e}^{\bu}(\cC)
\lra\ov\cM_{2\g(e/w_0)-1,(\{e,\si(e)\},\si)}^{\bu},\cC,\cC$$
the three component projection maps.
Let
$$L_+,L_-,\bE_+,\bE_-\lra \ov\cM_{\T;e}^{\bu}(\cC) $$
be the universal tangent line bundles at the marked points indexed by~$e$ and~$\si(e)$
and the subbundles of the Hodge bundle~$\bE$ consisting of the holomorphic differentials 
supported on the first and second component of each doublet
(the first component carries the marked point indexed by~$e$).
Denote~by
\BE{FeObsdfn_e}\begin{split}
\F_e^{\si}&\subset \pi_{\T;e}^*L_+\!\otimes\!\pi_{\cC;+}^*T\cC 
\oplus \pi_{\T;e}^*L_-\!\otimes\!\pi_{\cC;-}^*T\cC  \qquad\hbox{and}\\
\Obs_e^{\phi}&\subset \pi_{\T;e}^*\bE_+^*\!\otimes\!\pi_{\cC;+}^*TX 
\oplus \pi_{\T;e}^*\bE_-^*\!\otimes\!\pi_{\cC;-}^*TX
\end{split}\EE
the fixed loci of the conjugations over the identity induced by the involution~$\phi$ 
as in Section~\ref{Strata_subs} and~by
$$\cD_e^{\phi}\!:\F_e^{\si}\lra \Obs_e^{\phi}$$
the vector bundle homomorphism over $\ov\cM_{\T;e}^{\bu}(\cC)$ defined in the same way
as~before.\\

The projections
$$\ov\cM_{2\g(e/w_0)-1,(\{e,\si(e)\},\si)}^{\bu}\lra\ov\cM_{\g(e/w_0),1}
\quad\hbox{and}\quad
 \big\{(z^+,z^-)\!\in\!\cC^2\!:\,z^+\!=\!\phi(z^-)\big\}\lra\cC$$
on the first component in both cases are diffeomorphisms and induce
orientations on their domains from the complex orientations on their targets;
we will call the former orientations \sf{the complex orientations}.
The~map 
$$ \ov\cM_{\T;e}^{\bu}(\cC)\lra 
\ov\cM_{\g(e/w_0),1}\!\times\!\cC$$
induced by the above projections naturally lifts to vector bundle isomorphisms
\BE{FeIsom_e}\F_e^{\si}\lra \pi_{\g(e/w_0)}^*L\!\otimes\!\pi_{\cC}^*T\cC
\qquad\hbox{and}\qquad
\Obs_e^{\phi}\lra \pi_{\g(e/w_0)}^*\bE^*\!\otimes\!\pi_{\cC}^*TX\EE
obtained by projecting to the first components in~\e_ref{FeObsdfn_e}.
These projections induce  orientations on their domains from the complex orientations on their targets;
we will call the former orientations \sf{the complex orientations}.
Since the isomorphisms~\e_ref{FeIsom_e}  intertwine
the vector bundle homomorphisms~$\cD_e^{\phi}$ and~$\cD_{\g(e/w_0)}$,
\BE{NCeq_e} N\big(\cD_e^{\phi}\big)_{\C}=N\big(\cD_{\g(e/w_0)}\big),\EE
where the two sides denote the numbers of zeros of generic affine bundle maps
associated with $\cD_e^{\phi}$ and~$\cD_{\g(e/w_0)}$ with respect to
the complex orientations on their domain and target vector bundles
and the bases of these vector bundles.\\

In the case of Proposition~\ref{PrpSamPar},
\BE{wchMsplit_e}\wch\M_{\T,\f}(\cC)^\phi= 
\prod_{e\in\E^\si_{\bu;+}(\T)}\hspace{-.2in}\ov\cM_{2\g(e/w_0)-1,(\{e,\si(e)\},\si)}^{\bu}
\times\cC_{\T}
=\prod_{e\in\E^\si_{\bu;+}(\T)}\hspace{-.2in}\ov\cM_{\T;e}^{\bu}(\cC)\,.\EE
With 
$$\pi_e\!: \wch\M_{\T,\f}(\cC)^\phi \lra \ov\cM_{\T;e}^{\bu}(\cC)$$
denoting the component projection map,
\BE{cDTsplit_e}
\wch\cD_{\T}^{\phi}\!=\!\bigoplus_{e\in\E^\si_{\bu;+}(\T)}\!\!\!\!\!\!\!\pi_e^*\cD_e^{\phi}\!:
\wch\F\T^{\si}\!=\!\bigoplus_{e\in\E^\si_{\bu;+}(\T)}\!\!\!\!\!\!\!\pi_e^*\F_e^{\si}
\lra
\wch\Obs_\T^{\phi}\!=\!\bigoplus_{e\in\E^\si_{\bu;+}(\T)}\!\!\!\!\!\!\!\pi_e^*\Obs_e^{\phi}\,.\EE
The identifications in~\e_ref{wchMsplit_e} and~\e_ref{cDTsplit_e} induce orientations on
their domains from the complex orientations of their targets defined in the previous paragraph.
By~\e_ref{NCeq_e} and the first equality in~\e_ref{cDTsplit_e}, 
\BE{NCeq_e2} N\big(\wch\cD_{\T}^{\phi}\big)_{\C}=
\prod_{e\in\E^\si_{\bu;+}(\T)}\hspace{-.2in}N\big(\cD_{\g(e/w_0)}\big),\EE
where the left-hand side denotes the number of zeros of a generic affine bundle map
associated with $\wch\cD_{\T}^{\phi}$
with respect to
the complex orientations on its domain and target vector bundles
and the base of these vector bundles.\\

\noindent
The curve $\cC\!\subset\!X$ corresponds to an element $[\u]$ of $\M_{h,\f}^*(X,B;J)^\phi$
as around~\e_ref{u_e}.
The real orientation on~$(X,\om,\phi)$ fixed in Section~\ref{intro_sec} determines the sign,
$$\sgn_{\f}(\cC)=\sgn(\u)\in\big\{\pm1\big\},$$ 
of~$[\u]$.
The open subspace~$\cC_{\T}^*$ of~$\cC_{\T}$ corresponds to a family of maps
with $m$~additional conjugate pairs of marked points.
The real orientation on~$(X,\om,\phi)$ and the choice of 
the subset~$\E^\si_{\bu;+}(\T)$ of $\E_{\bu}(\T)$ determine an orientation on~$\cC_{\T}$.
It differs from the complex orientation by~$\sgn_{\f}(\cC)$.\\

\noindent
For each $g'\!\in\!\Z^+$,
the real orientation on $(X,\om,\phi)$ determines an orientation on the moduli space
$\ov\M_{2g'-1,(1^+,1^-)}^{\bu}(X,0;J)^\phi$ of $J$-holomorphic degree~0 maps
from $g'$-doublets with one conjugate pair of marked points.
As topological spaces
\BE{Mdoubsplit_e}\ov\M_{2g'-1,(1^+,1^-)}^{\bu}(X,0;J)^\phi \approx 
\ov\cM_{2g'-1,(1^+,1^-)}^{\bu}\!\times\!
\big\{\big(x^+,x^-)\!\in\!X^2\!:\,x_+\!=\!\phi(x_-)\big\},\EE
but the moduli space on the left-hand side comes with an associated obstruction bundle
\BE{Obs0dfn_e} \Obs_X^{\phi}\lra 
\ov\cM_{2g'-1,(1^+,1^-)}^{\bu}\!\times\!\big\{\big(x,\phi(x)\big)\!:\,x\!\in\!X\big\};\EE
the latter is the fixed locus of the involution~on 
$$\pi_{g'}^*\bE_+^*\!\otimes\!\pi_+^*TX \oplus  \pi_{g'}^*\bE_-^*\!\otimes\!\pi_-^*TX
\lra \ov\cM_{2g'-1,(1^+,1^-)}^{\bu}\!\times\!
\big\{\big(x^+,x^-)\!\in\!X^2\!:\,x_+\!=\!\phi(x_-)\big\}$$
induced by~$\phi$ as in Section~\ref{Strata_subs}.
In particular, the projection to the first component
\BE{Obs0dfn_e2} \Obs_X^{\phi}\lra \pi_{g'}^*\bE_+^*\!\otimes\!\pi_+^*TX\EE
is an isomorphism.
The last factors in~\e_ref{Mdoubsplit_e} and~\e_ref{Obs0dfn_e2}
are oriented from the (almost) complex orientation of~$X$ 
(by projecting to the first component in the case of~\e_ref{Mdoubsplit_e}).
Along with the complex orientations on the first factors on the right-hand sides
of~\e_ref{Mdoubsplit_e} and~\e_ref{Obs0dfn_e2}, these orientations on the last factors
induce an orientation on the total space of the vector bundle in~\e_ref{Obs0dfn_e};
we will call this orientation \sf{the complex orientation}.
By \cite[Theorem~1.4]{RealGWsII}, the orientation on  the total space of this vector bundle
induced by the real orientation on $(X,\om,\phi)$ differs 
from the complex orientation by~$(-1)^{g'-1}$.\\ 

\noindent
By \cite[Theorem~1.2]{RealGWsII}, the immersion~\e_ref{Iota_e} 
is orientation-preserving with respect to the orientations on its domain and target
induced  by a real orientation on $(X,\om,\phi)$ if and only if $m\!\in\!2\Z$.
Combining this with~\e_ref{NCeq_e2} and the last two paragraphs, we conclude that 
\begin{equation*}\begin{split}
N\big(\wch\cD_{\T}^{\phi}\big)
&= (-1)^m\sgn_{\f}(\cC)\!\!\!\!\!\!\!\!  \prod_{e\in\E^\si_{\bu;+}(\T)}\hspace{-.2in}
\big((-1)^{\g(e/w_0)-1}\!N\big(\cD_{\g(e/w_0)}\big)\big)\\
&= \sgn_{\f}(\cC)\!\!\!\!\!\!\!\! 
\prod_{e\in\E^\si_{\bu;+}(\T)}\hspace{-.2in}
\big((-1)^{\g(e/w_0)}\!N\big(\cD_{\g(e/w_0)}\big)\big).
\end{split}\end{equation*}
The claim now follows from~\e_ref{BasicContr_e}, \e_ref{tiAutAut_e}, and Lemma~\ref{NcDg_lmm} below.
\end{proof}

\begin{lmm}\label{NcDg_lmm}
With notation as above,
$$N(\cD_{g'})=\int_{\ov\cM_{g',1}\times\cC}\!\!
 c_{2g'}\big(\pi_{g'}^*\bE^*\!\otimes\!\pi_\cC^*\cN_X\cC\big)
\bigg(\sum_{r=0}^{g'-1}(-1)^r\la_r\psi^{g'-1-r}\bigg)
.$$
\end{lmm}

\begin{proof}
By definition $N(\cD_{g'})$ is the signed cardinality of the zero set
of the affine bundle~map
\begin{gather*}
\psi_{g';\vp}\!: \pi_{g'}^*L\!\otimes\!\pi_{\cC}^*T\cC\lra
\pi_{g'}^*\bE^*\!\otimes\!\pi_{\cC}^*TX, \\
\psi_{g';\vp}(\up)=\cD_{g'}(\up)+\vp(\u)
\quad\forall~\up\!\in\!\pi_{g'}^*L\!\otimes\!\pi_{\cC}^*T\cC\big|_{\u},
~\u\!\in\!\ov\cM_{g',1}\!\times\!\cC,
\end{gather*}
with respect to the complex orientations on the domain and target vector bundles
and the complex orientation on their base, for a generic section~$\vp$
of the target bundle.
Since $\vp$ is generic, it spans a trivial complex subbundle $\C\vp$ 
in $\pi_{g'}^*\bE^*\!\otimes\!\pi_{\cC}^*TX$.
The composition of~$\vp$ with the tensor product of the identity and the natural
projection $TX|_{\cC}\!\lra\!\cN_X\cC$ induces a section~$\vp'$ of 
the vector bundle
$$\pi_{g'}^*\bE^*\!\otimes\!\pi_{\cC}^*\cN_X\cC
\lra \ov\cM_{g',1}\!\times\!\cC\,.$$
Since~$\vp$ is generic, the zero set $\cZ_{\vp'}$ of this section is 
a smooth oriented sub-orbifold.\\

\noindent
We denote by $\De_{g'}\!\subset\!\ov\cM_{g',1}$ the subspace of marked curves so that 
the irreducible component containing the marked point is a smooth rational curve.
The vector bundle homomorphism~$\cD_{g'}$ corresponds to a section of the bundle
$$\big(\pi_{g'}^*L\!\otimes\!\pi_{\cC}^*T\cC\big)^*\otimes
\pi_{g'}^*\bE^*\!\otimes\!\pi_{\cC}^*TX\lra  \ov\cM_{g',1}\!\times\!\cC,$$
which we denote in the same way.
We note~that 
\BE{cDg0_e}\cD_{g'}^{-1}(0)= \De_{g'}\!\times\!\cC \subset \ov\cM_{g',1}\!\times\!\cC\,.\EE
Along with the projection 
$$\pi_{g'}^*\bE^*\!\otimes\!\pi_{\cC}^*TX\lra 
\pi_{g'}^*\bE^*\!\otimes\!\pi_{\cC}^*TX\big/\C\vp\,,$$
the section~$\cD_{g'}$  induces a section of the bundle
$$\big(\pi_{g'}^*L\!\otimes\!\pi_{\cC}^*T\cC\big)^*\otimes
(\pi_{g'}^*\bE^*\!\otimes\!\pi_{\cC}^*TX\big/\C\vp\big)\lra \ov\cM_{g',1}\!\times\!\cC.$$
Since the image of $\cD_{g'}$ lies in $\pi_{g'}^*\bE^*\!\otimes\!\pi_{\cC}^*T\cC$,
the latter restricts to a section of the bundle
$$V\equiv \big(\pi_{g'}^*L\!\otimes\!\pi_{\cC}^*T\cC\big)^*\otimes
\big(\pi_{g'}^*\bE^*\!\otimes\!\pi_{\cC}^*T\cC\big/\C\vp\big)\lra \cZ_{\vp'}\,;$$
we denote it by $\cD_{g'}^{\perp}$.\\ 

\noindent
Since the image of $\cD_{g'}$ lies in $\pi_{g'}^*\bE^*\!\otimes\!\pi_{\cC}^*T\cC$,
 $N(\cD_{g'})$ is the signed cardinality of the zero set of the affine bundle~map
$$\psi_{g';\vp}'\!: \pi_{g'}^*L\!\otimes\!\pi_{\cC}^*T\cC\lra
\pi_{g'}^*\bE^*\!\otimes\!\pi_{\cC}^*T\cC, \quad
\psi_{g';\vp}'(\up)=\psi_{g';\vp}(\up)
~~\forall~\up\!\in\!\pi_{g'}^*L\!\otimes\!\pi_{\cC}^*T\cC\big|_{\cZ_{\vp'}}.$$
Since the linear part of this bundle map is holomorphic, 
it is a regular polynomial  in the sense of \cite[Definition~3.9]{g2n2and3}.
By \cite[Lemma~3.14]{g2n2and3},
\BE{cDg_e} N(\cD_{g'})=\blr{e(V),\big[\cZ_{\vp'}\big]}-
\Cntr_{\cD_{g'}^{-1}(0)\cap\cZ_{\vp'} }\!\big(\cD_{g'}^{\perp}\big).\EE
The last term above is the \sf{$\cD_{g'}^{\perp}$-contribution} to the middle term from 
the subspace $\cD_{g'}^{-1}(0)\!\cap\!\cZ_{\vp'}$ of~$\cZ_{\vp'}$;
this notion is the direct analogue of Definition~\ref{DfnContri} in this setting.
By \cite[Propositions~2.18A,B]{g0pr} and~\e_ref{cDg0_e}, 
\BE{ContrSum_e}\Cntr_{\cD_{g'}^{-1}(0)\cap\cZ_{\vp'} }\!\big(\cD_{g'}^{\perp}\big)
=\sum_i d_i\blr{e(V_i),\big[(\De_{g';i}\!\times\!\cC)\!\cap\!\cZ_{\vp'}\big]}\EE
for some $d_i\!\in\!\Z$, closed subvarieties $\De_{g';i}\!\subset\!\De_{g'}$, and vector bundles~$V_i$
over $\De_{g';i}\!\times\!\cC$; these are determined by the scheme structure of 
$\cD_{g'}^{-1}(0)\!\subset\!\ov\cM_{g',1}$.\\

\noindent 
By the Poincare Duality, \e_ref{ECexpand_e}, and~\e_ref{ECexpand_e2},
\BE{ContrSum_e2}\begin{split}
\blr{e(V_i),\big[(\De_{g';i}\!\times\!\cC)\!\cap\!\cZ_{\vp'}\big]}
&=\blr{e(V_i)e\big(\pi_{g'}^*\bE^*\!\otimes\!\pi_{\cC}^*\cN_X\cC\big),\big[\De_{g';i}\!\times\!\cC\big]}\\
&=\big(2\!-\!2h\!-\!c_1(B)\big)\blr{e(V_i)\la_{g'-1}\la_{g'},\big[\De_{g';i}\big]}\,.
\end{split}\EE
By \cite[Lemma~1]{Faber}, $\la_{g'-1}\la_{g'}$ vanishes on $\ov\cM_{g',1}\!-\!\cM_{g',1}$ 
and thus on~$\De_{g';i}$.
Combining this observation with~\e_ref{cDg_e},  \e_ref{ContrSum_e}, and~\e_ref{ContrSum_e2}, 
we conclude~that 
\begin{equation*}\begin{split}
N(\cD_{g'})&=\blr{e\big((\pi_{g'}^*L\!\otimes\!\pi_{\cC}^*T\cC)^*\!\otimes\!
(\pi_{g'}^*\bE^*\!\otimes\!\pi_{\cC}^*T\cC/\C\vp)\big),\big[\cZ_{\vp'}\big]}\\
&=\blr{e\big((\pi_{g'}^*\bE^*\!\otimes\!\pi_{\cC}^*T\cC)/
(\pi_{g'}^*L\!\otimes\!\pi_{\cC}^*T\cC)\big),\big[\cZ_{\vp'}\big]}
=\blr{e\big((\pi_{g'}^*\bE^*)/(\pi_{g'}^*L)\big),\big[\cZ_{\vp'}\big]}\\
&=\sum_{r=0}^{g'-1}\!\!\blr{(-1)^r\la_r\psi^{g'-1-r},\big[\cZ_{\vp'}\big]}
=\sum_{r=0}^{g'-1}\int_{\ov\cM_{g',1}\times\cC}\!\!(-1)^r\la_r\psi^{g'-1-r}
 c_{2g'}\big(\pi_{g'}^*\bE^*\!\otimes\!\pi_\cC^*\cN_X\cC\big).
\end{split}\end{equation*}
The third equality above holds because $\pi_{\cC}^*c_1(T\cC)$ vanishes on~$\cZ_{\vp'}$
by~\e_ref{ECexpand_e} and the Poincare Duality;
the last equality holds by the Poincare Duality.
\end{proof}

\section{Proof of Proposition~\ref{PrpSection}}
\label{SecAnalytic}

Continuing with the notation of Section~\ref{SecPrps},
we set up a two-step gluing construction to prove Proposition~\ref{PrpSection}.
For each $\T\inn\sT_{g,l}(\cC)^\phi$ as in~\e_ref{RDG_e},
we first smooth out all nodes corresponding to~$\Edg^0$ in~\e_ref{CsubRDG_e} 
and~$\Edg^c$ in~\e_ref{TComplement_e}.
This step is not obstructed.
At the second stage,
we smooth out the remaining nodes,
i.e.~the nodes corresponding to $\E_\bu(\T)$; see~\e_ref{EdgDecomp_e}.
This step is obstructed.
We will use the same strategy as in~\cite[Sections~3.2,4.2,4.3]{g1comp} 
and~\cite[Section~3.3]{g1comp2}.\\

For a generic choice of the pseudocycles~$\f$, the projection 
$$  \ov\M_{g,\f}(X,B;J)^\phi\lra \ov\M_{g,l}\big(X,B;J\big)$$
to the moduli space component in~\e_ref{MBarF_e} is an embedding. 
Thus, we can view  
$\fM_{\T,\f}(\cC)^\phi$ as the subset of the elements of
$\ov\M_{g,l}(X,B;J)$ which map to the curve $\cC\!\subset\!X$
and meet the pseudocycles~$\f$ at the marked points.\\

\noindent
Let $(S_0,\si_0)$ be as in~\e_ref{Cs0ubRDG_e}.
We denote by $\fM_{h,(S_0,\si_0)}(B)^{\phi}$ the moduli space of 
degree~$B$ $J$-holomorphic real maps from smooth genus~$h$ symmetric 
surfaces with the marked points indexed by the set~$S_0$ with the involution~$\si_0$.
Analogously to~\e_ref{tiM_e}, let
\BE{tifMT_e}\ti\fM_{\T}(X)^\phi= 
	\prod_{e\in\E_{0;\ge}^\si(\T)\sqcup\E^\si_{\bu;\ge}(\T)}\hspace{-.45in}\cM_{\T,\si;e}
\times\fM_{h,(S_0,\si_0)}(B)^{\phi}\EE
and 
$$\io_{\T,\f}\!: \ti\fM_{\T}(X)^\phi
\lra  \ov\M_{g,l}(X,B;J)$$
be the natural $\ti\Aut^*(\T)$-invariant node-identifying immersion extending~\e_ref{Iota_e}.
The image of this immersion is the stratum
$$\fM_{\T}(X)^\phi\subset \ov\M_{g,l}(X,B;J)$$
of stable maps of the topological type~$\T$
(not necessarily passing through the constrains).
For each element $\u$ of $\ti\fM_{\T}(X)^\phi$,
a vertex $w\!\in\!\Ver$ corresponds to 
an irreducible component~$\Si_{\u;w}$ of the domain~$\Si_{\u}$ of
the map~$u$ in the image of this immersion.
Each edge $e\eq(w_e^0,w_e^c)$ in $\E_{\bu}(\T)\!\subset\!\Edg$ corresponds to a node
$x_e(\u)\!\in\!\Si_u$ formed by joining a point 
$x_e^0(\u)$ in $\Si_{\u;e}^0\!\equiv\!\Si_{\u;w_e^0}$ with
a point $x_e^c(\u)$  in $\Si_{\u;e}^c\!\equiv\!\Si_{\u;w_e^c}$.\\

\noindent
We denote the natural extensions of the vector bundles in~\e_ref{ObsTCdfn_e}, \e_ref{Obs_e}, 
and~\e_ref{FT_e} and of the bundle homomorphism in~\e_ref{Thudfn_e}  
to $\ti\fM_{\T}(X)^\phi$ in the same way.
The center component $\Si_{\u;0}\!=\!\Si_{\u;\v_0}$ is no longer identified with~$\cC$.
The bundle $\pi^*_{\cC;\lr{e}}TX$ in~\e_ref{ObsTCdfn_e} should thus be replaced by
 $\pi_{\fM}^*\ev_{\lr{e}}^*TX$,  where $\pi_{\fM}$ is
the projection onto the last factor in~\e_ref{tifMT_e}. 
With notation as in~\e_ref{FT_e}, let
\begin{equation*}\begin{split}
\ti\F_1\T&=\Big(\bigoplus_{e\in\Edg^0\sqcup\Edg^c}\!\!\!\!\!\!\!\!\!\!
  \ti L_{e}~~\Big)\Big/\Aut_{\bu}(\T) \lra \ti\M_{\T}(X)^\phi\,,\\
\ti\fF\T&=
\bigoplus_{e\in\E_\bu(\T)}\!\!\!\!
 \pi^*_{\T;e}L_e\!\otimes\!\pi_{\fM}^*L_{\lr{e}} \lra \ti\M_{\T}(X)^\phi\,.
\end{split}\end{equation*}
We denote by $\up_1\!\in\!\ti\F_1\T$ the image of an element~$\up$ 
of~$\ti{\F}\T$ under the natural projection
$$\ti\F\T \lra \ti\F_1\T\subset \ti\F\T\,.$$
The involution $\si$ acts on $\ti\fF\T$ by
$$\big(\si\big((\up_{e'}\!\otimes\!\up_{\lr{e'}})_{e'\in\E_\bu(\T)}\big)\big)_e
=\tnd\si(\up_{\si(e)})\!\otimes\!\tnd\si(\up_{\si\lr{e}}).$$
We denote the corresponding fixed locus by~$\ti\fF\T^{\si}$.\\

For each $e\inn\E_\bu(\T)$,
let $\Edg'_{0,\bu;e}\!\subset\!\Edg'_{0;\lr{e}}$
be the subset of edges whose removal separates $e$ and~$\lr{e}$.
This subset is empty if $e\eq\lr{e}$.
Let
$$\Edg_{0,\bu;e}=\Edg'_{0,\bu;e}\cup\{e\}\cup\{\lr{e}\}\subset
 \Edg\qquad \forall~e\inn\E_\bu(\T)\,.$$
Since $\Si_{\u;w}\!\approx\!\P^1$ for every $w\!\in\!e\!$ with $e\!\in\!\Edg'_{0,\bu;e}$,
there is a natural isomorphism
\BE{Fedfn_e}
F_e\!:
\bigotimes_{e'\in\Edg_{0,\bu;e}}\!\!\!\!\!\!\!\!\ti L_{e'} 
  \,\stackrel{\cong}{\lra} \pi^*_{\T;e}L_{e}\!\otimes\!\pi^*_{\cC;\lr{e}}T\cC 
	\,;\EE
see~\cite[(2.3)]{desing}.
Define 
\begin{alignat*}{2}
\rho_\T\!:\ti\F\T&\lra \ti\fF\T, &\quad 
\rho_\T\big(\up\big) \big(\rho_e(\up)\big)_{e\in\E_{\bu}(\T)}
 &= \bigg(\!F_e\Big(\!\bigotimes_{e'\in\Edg_{0,\bu;e}}
 \!\!\!\!\!\!\!\up_{e'}\Big)\!\!\bigg)_{\!\!e\in\E_\bu(\T)}\,,\\
\cD_\T\!:\ti\fF\T&\lra\Obs_\T, &\quad
 \big\{\cD_\T\big((\up_e\!\otimes\!\up_{\lr{e}})_{e\in\E_\bu(\T)}\big)\big\}
 \big((\ka_e)_{e\in\E_\bu(\T)}\big)
 &=
\Big(\!\ka_e(\up_e)\big(\tnd_{x_{\lr{e}}^0(\u)}u_0(\up_{\lr{e}})\!\big)\!\!
\Big)_{\!e\in\E_\bu(\T)}\,,
\end{alignat*}
where $u_0\!=\!u|_{\Si_{\u;0}}$ and $x_{\lr{e}}^0(\u)\!\in\!\Si_{\u;0}$ is the point
forming the node of~$\Si_{\u}$ corresponding to the edge $\lr{e}\!\in\!\E_0(\T)$.
The bundle maps~$\cD_\T$ and~$\rho_\T$ restrict to
$$\cD_\T^{\phi}\!: \ti\fF\T^{\si}\lra \Obs_\T^{\phi}\qquad\tn{and}\qquad
 \rho_\T^{\si}\!: \ti\F\T^{\si}\lra \ti\fF\T^{\si},$$
respectively.
If $\T\!=\!(\ft,\fd)$ with $\ft$ basic, then
$\ti\fF\T^{\si}\!=\!\ti\F\T^{\si}$, $\rho_\T^{\si}\!=\!\id$,
and the restriction of $\cD_\T^{\phi}$ to $\ti\fM_{\T,\f}(\cC)^\phi$
is as below~\e_ref{DT_e}.

\subsection{The unobstructed gluing step}
\label{UnObsGlue_subs}

\noindent
Let $\fU\!\lra\!\De^{\!\si}$ be a family of deformations of the domains of the elements of 
$\ti\M_{\T}(X)^\phi$ for a small neighborhood $\De^{\!\si}\!\subset\!\ti\F\T^{\si}$ of the zero section.
It is the restriction to $\ti\F\T^{\si}\!\subset\!\ti\F\T$ of the analogous family
in the complex setting.
The fiber of $\fU\!\lra\!\De^{\!\si}$ over $\up\!\in\!\De^{\!\si}$ is a nodal symmetric 
surface~$(\Si_{\up},\si_{\up})$.
Its dual graph~$\T_{\up}$ is a rooted decorated graph obtained from~$\T$ 
by deleting the edges corresponding to 
the non-zero components of~$\up$ and identifying the vertices of each deleted edge.
In particular,
$$ \E_\bu(\T_{\up})=\E_\bu(\T) \qquad\forall~\up\in \De^{\!\si}\!\cap\!\ti\F_1\T;$$
we will not distinguish between these two sets below.
It can be assumed that $\up_1\!\in\!\De^{\!\si}$ for all $\up\!\in\!\De^{\!\si}$.\\

\noindent
We fix an involution-invariant Riemannian metric on~$\fU$ and denote its restriction 
to~$\Si_{\up}$ by~$g_{\up}$.
For any $\up\!\in\!\De^{\!\si}$, $e\!\in\!\Edg$ with $\up_e\!=\!0$, and $\de\!\in\!\R$, 
let $\Si_{\up;e}(\de)\!\subset\!\Si_{\up}$ denote
the $g_{\up}$-ball of radius~$\de$ centered at the node $x_e(\up)$ corresponding 
to~$e$.
If in addition $\up\!\in\!\ti\F_1\T$ and $e\!\in\!\E_{\bu}(\T)$, define
$$\Si_{\up;e}^0(\de)=\Si_{\up;e}(\de)\cap\Si_{\up;e}^0\,, \qquad
\Si_{\up;e}^c(\de)=\Si_{\up;e}(\de)\cap\Si_{\up;e}^c\,, \qquad
\wh\Si_{\up;e}^c(\de)=\wh\Si_{\up;e}^c\!\cup\!\Si_{\up;e}(\de)\,.$$
If $e\!\in\!\E_0(\T)$, let $\wh\Si_{\u;e}^0(\de)\!=\!\wh\Si_{\u;e}^0\!\cup\!\Si_{\u;e}(\de)$.
The subset $\De^{\!\si}\!\subset\!\ti\F\T^{\si}$ contains~$\ti\M_{\T}(X)^\phi$
as the zero section.
There are continuous fiber-preserving retractions~$q$ and~$q_{\bu}$ 
respecting the involutions so that the~diagram
$$\xymatrix{ \fU \ar@/^2pc/[rrrr]_q \ar[rr]^>>>>>>>>>>>{q_{\bu}}\ar[d]&&
 \fU|_{\De^{\!\si}\cap\ti\F_1\T^{\si}} \ar[rr]^>>>>>>>>>>>q\ar[d]&&
\fU|_{\ti\M_{\T}(X)^\phi} \ar[d]\\ 
\De^{\!\si} \ar[rr]&&  \De^{\!\si}\!\cap\!\ti\F_1\T^{\si} 
\ar[rr]&&\ti\M_{\T}(X)^\phi}$$
commutes.
We denote~by
\BE{qupdfn_e} q_{\up}\!:\,\Si_{\up}\lra\Si_\u
\quad\hbox{and}\quad
q_{\bu;\up}\!:\,\Si_{\up}\lra\Si_{\up_1},
\qquad\up\in\De^{\!\si}|_{\u},~\u\!\in\!\ti\M_{\T}(X)^\phi\,,\EE
the restrictions of~$q$ and~$q_{\bu}$ to~$\Si_{\up}$.\\

\noindent 
After possibly shrinking~$\De^{\!\si}$, 
the Riemannian metric on~$\fU$ and the map~$q$ can be chosen so~that 
\begin{enumerate}[label=$\bu$,leftmargin=*]

\item $q$ is smooth on each stratum of~$\fU$,

\item $q(z_i^{\pm}(\up))\!=\!z_i^{\pm}\big(\pi_{\ti\F\T^\si}(\up)\big)$ for all 
$\up\in\De^{\!\si}$ and $i\!=\!1,\ldots,l$,

\item for each $\up\in\De^{\!\si}$, $q_{\up}$ is biholomorphic on 
the complement of the subspaces $q_{\up}^{-1}(\Si_{\u;e}(2\sqrt{|\up_e|}))$
with $e\!\in\!\Edg$,

\item there exists a continuous function $\de_q\!:\ti\M_{\T}(X)^\phi\!\lra\!\R^+$
such that every restriction~\e_ref{qupdfn_e} is a $(g_{\up},g_{\u})$-isometry 
on the complement of the subspaces $q_{\up}^{-1}(\Si_{\u;e}(\de_q(\u)))$
with $e\!\in\!\Edg$ and~$\up_e\!\neq\!0$,

\item there exist  a continuous function $\de_{\dbar}\!:\ti\M_{\T}(X)^\phi\!\lra\!\R^+$
and holomorphic functions
$$z_e,z_e^c\!: \bigcup_{\u\in\ti\M_{\T}(X)^\phi}\hspace{-.2in}q^{-1}
\big(\Si_{\u;e}(\de_{\dbar}(\u))\big)\lra\C, \qquad e\inn\E_\bu(\T),$$
such that $z_{e;\u}\!\equiv\!z_e|_{\Si_{\u;e}^0(\de_{\dbar}(\u))}$ and 
$z_{e;\u}^c\!\equiv\!z_e^c|_{\Si_{\u;e}^c(\de_{\dbar}(\u))}$
are unitary coordinates centered at $x_e(\u)$ and
\begin{gather*}
\big(z_ez_e^c\big)|_{q_{\up}^{-1}(\Si_{\u;e}(\de_{\dbar}(\u)))}\, 
 \frac{\prt}{\prt z_{e;\u}}\bigg|_{x_e^0(\u)}\otimes 
\frac{\prt}{\prt z_{e;\u}^c}\bigg|_{x_e^c(\u)} =\up_e\,, \quad
\big|z_e|_{q_{\up}^{-1}(x_e(\u))}\big|=\big|z_e^c|_{q_{\up}^{-1}(x_e(\u))}\big|\,,\\
z_e\big|_{q_{\up}^{-1}(\Si_{\u;e}^0(\de_{\dbar}(\u))-\Si_{\u;e}(2\sqrt{|\up_e|}))}=
z_{e;\u}\!\circ\!q_{\up}\big|_{q_{\up}^{-1}(\Si_{\u;e}^0(\de_{\dbar}(\u))-\Si_{\u;e}(2\sqrt{|\up_e|}))}\,,\\
z_e^c\big|_{q_{\up}^{-1}(\Si_{\u;e}^c(\de_{\dbar}(\u))-\Si_{\u;e}(2\sqrt{|\up_e|}))}=
z_{e;\u}^c\!\circ\!q_{\up}\big|_{q_{\up}^{-1}(\Si_{\u;e}^c(\de_{\dbar}(\u))-\Si_{\u;e}(2\sqrt{|\up_e|}))}  
\end{gather*} 
for all $\up\!\in\!\De^{\!\si}\big|_{\u}$ and $\u\!\in\!\ti\M_{\T}(X)^\phi$.
\end{enumerate}
We can assume that $2\de_q(\u)$ is less than the minimal distance between 
the nodal and marked points
of~$\Si_{\u}$ for every $\u\!\in\!\ti\M_{\T}(X)^\phi$, $\de_{\dbar}\!<\!\de_q\!<\!1$, and
$$16|\up|< \de_{\dbar}\big(\pi_{\ti\F\T}(\up)\big)^2 \qquad\forall~\up\!\in\!\De^{\!\si}\,.$$
The functions~$\de_q$, $\de_{\dbar}$, $z_e$, and~$z_e^c$ can be chosen
compatibly with the $\ti\Aut^*(\T)$-actions and with the involution on~$\fU$.\\

\noindent
With~$e\!\in\!\E_{\bu}(\T)$ and $F_e$ as in~\e_ref{Fedfn_e}, define
\begin{gather*}
\rho_{e;\u}\!:\C\lra T_{x_{\lr{e}}(\u)}\Si_{\u;0} \qquad\hbox{by}\\
v_e\otimes \rho_{e;\u}(c)=
F_e\bigg(v_e\!\otimes\!c\frac{\prt}{\prt z_{e;\u}}\bigg|_{x_e^0(\u)}\otimes
\!\bigotimes_{\begin{subarray}{c}e'\in\Edg_{0,\bu;e}\\ e'\neq e\end{subarray}}
 \!\!\!\!\!\!\!\up_{e'}\bigg)
\quad\forall~c\!\in\!\C,~v_e\!\in\!\pi^*_{\T;e}L_{e}\big|_{\u}\,.
\end{gather*}
By the $\C$-linearity of~$F_e$,  $\rho_{e;\u}$ is a well-defined $\C$-linear isomorphism.\\

Let $g_X$ be a $\phi$-invariant metric on $X$,
$\na^X$ be its  Levi-Civita connection, and $\exp$ be the exponential map of~$\na^X$.
Fix a number $p\!>\!2$.
For each $\up\!\in\!\De^{\!\si}$ and a smooth real map $f\!:\Si_{\up}\!\lra\!X$, let 
\BE{Dphif_e}D^{\phi}_f\!: \Ga(f)^\phi\!\equiv\!\Ga\big(\Si_{\up};f^*TX\big)^{\phi,\si_{\up}}
\lra \Ga^{0,1}(f)^\phi\equiv
 \Ga^{0,1}\big(\Si_{\up};T^*\Si_{\up}^{0,1}\!\otimes\!f^*TX\big)^{\phi,\si_{\up}}\EE
denote the linearization of the $\dbar_J$-operator at~$f$
defined via the connection~$\na^X$ as in \cite[Section~3.1]{McS}.
The construction of \cite[Section~3]{LiT} provides modified Sobolev norms 
$\|\cdot\|_{\up,p,1}$ and $\|\cdot\|_{\up,p}$~on
the domain and target of the homomorphism~\e_ref{Dphif_e}
as well as $L^2$-inner products $\llrr{\cdot,\cdot}_{\up,2}$ on these spaces.
These homomorphisms  satisfy elliptic estimates involving 
the $\|\cdot\|_{\up,p,1}$ and $\|\cdot\|_{\up,p}$ norms
with coefficients
that depend only on $\pi_{\ti\F\T}(\up)\!\in\!\ti\M_{\T}(X)^\phi$
and $\|\tnd f\|_{\up,p}\!\in\!\R$.\\

For $\up\!\in\!\De^{\!\si}\!\cap\!\ti\F_1\T|_{\u}$, 
$\Si_{\up}$ is obtained from~$\Si_{\u}$ by smoothing the nodes corresponding
to the subsets $\Edg^0$ and~$\Edg^c$ of the edges of~$\T$.
Since the restrictions of~$u$ to the irreducible components of~$\Si_{\u}$ sharing the nodes
indexed by~$\Edg^c$ are constant, the smoothings of these nodes of~$\Si_{\u}$ extend to
the deformations of the map~$\u$ into~$X$.
The nodes indexed by~$\Edg^0$ are the nodes of the center~$\Si_{\u}^0$ of~$\Si_{\u}$,
which consists of the non-contracted smooth genus~$h$ curve~$\Si_{\u;0}$
with contracted rational tails~$\wh\Si_{\u;e}^0$ attached.
If $\u\!\in\!\ti\fM_{\T,\f}(\cC)^\phi$, then $u|_{\Si_{\u}^0}$ is a regular $J$-holomorphic map.
Therefore, there exists a neighborhood 
$$\ti{U}_{\T,\f}^{\phi}(\cC)\subset\ti\M_{\T}(X)^\phi$$
of $\ti\fM_{\T,\f}(\cC)^\phi$ such~that the family of the deformations
$$\fU\big|_{\De_{\cC}^{\!\si}\cap\ti\F_1\T}\lra\De_{\cC}^{\!\si}\!\cap\!\ti\F_1\T\,,
\qquad\hbox{where}\quad  
\De_{\cC}^{\!\si}= \De^{\!\si}\cap
\ti\F\T^{\si}\big|_{\ti{U}_{\T,\f}^{\phi}(\cC)}\,,$$
of the domains of the elements of $\ti{U}_{\T,\f}^{\phi}(\cC)$ extends to 
a continuous  $\ti\Aut^*(\T)$-invariant family 
\BE{ti1up_e}\ti{u}_{\up}\!:\Si_{\up}\lra X, \qquad 
\up\in \De_{\cC}^{\si}\!\cap\!\ti\F_1\T\,,\EE
of $J$-holomorphic maps.
This family can be chosen so that it is smooth on each stratum of 
$\De_{\cC}^{\!\si}\!\cap\!\ti\F_1\T$.
The smoothings of the maps~$\u$ with a fixed image curve arise algebraically.\\

By shrinking~$\de_{\dbar}$, we can assume that 
the diameter of $u(\Si_{\u;e}^0(\de_{\dbar}(\u)))$ is at most half the injectivity
radius of~$g_X$ for every $\u\!\in\!\ti\M_{\T}(X)^\phi$ and $e\!\in\!\E_0(\T)$.
Since $\ti{u}_{\up}\!\lra\!u$ as $\up\!\lra\!\u$, 
we can also assume  that the diameter of 
$$u\big(\Si_{\u;\lr{e}}^0\big(\de_{\dbar}(\u))\big)\cup
\ti{u}_{\up}\big(q_{\up}^{-1}\big(\wh\Si_{\u;\lr{e}}^0(\de_{\dbar}(\u))\big)\big)\subset X$$ 
is less than the injectivity radius for all $\up\!\in\!\De_{\cC}^{\!\si}\!\cap\!\ti\F_1\T|_{\u}$
and $e\!\in\!\E_\bu(\T)$.
For $\up\!\in\!\De_{\cC}^{\!\si}\!\cap\!\ti\F_1\T|_{\u}$ and $e\!\in\!\E_\bu(\T)$,
we can thus define a smooth function 
\BE{tize}
\ze_{\up;e}\!: q_{\up}^{-1}\big(\wh\Si_{\u;\lr{e}}^0(\de_{\dbar}(\u))\big)
\!\cup\!\wh\Si_{\up;e}^c \lra T_{u(x_{\lr{e}}(\u))}X \quad\hbox{s.t.}~~~
\exp_{u(x_{\lr{e}}(\u))}\!\ze_{\up;e}=\ti u_{\up}\EE
by requiring that $\ze_{\up;e}|_{\wh\Si_{\up;e}^c}=0$.\\

For each \hbox{$\up\!\in\!\De_{\cC}^{\!\si}|_{\u}\!\cap\!\ti\F_1\T$}, let
$$u_{\up}=u\!\circ\!q_{\up}\!:\Si_{\up}\lra X.$$
Since $\ti{u}_{\up}\!\lra\!u$ as $\up\!\lra\!\u$, 
\BE{zeup1_e} \ti u_{\up}\equiv\exp_{u_{\up}}\!\!\ze_{\up} \EE
for some small vector field $\ze_{\up}\inn\Ga(u_{\up})^\phi$,
provided $\De^{\!\si}$ and $\ti{U}_{\T,\f}^{\phi}(\cC)$
are sufficiently small. 
Since the map $q_{\up}$ is holomorphic on $\Si_{\up;e}(\de_{\dbar}(\u))$ 
for each $e\!\in\!\E_{\bu}(\T)$,
$$z_{e;\up}\equiv z_e\big|_{\Si_{\up;e}^0(\de_{\dbar}(\u))}=
z_{\u;e}\!\circ\!q_{\up}\big|_{\Si_{\up;e}^0(\de_{\dbar}(\u))}\!: 
\Si_{\up;e}^0\big(\de_{\dbar}(\u)\big)\lra \C$$
is a holomorphic coordinate 
centered at $x^0_e(\up)\!=\!q_{\up}^{-1}(x^0_e(\u))$.

\begin{lmm}\label{LMExpansion}
The families of $\fU\!\lra\!\De^\si$ and~\e_ref{ti1up_e}
can be parametrized so that the following holds. 
There exist smooth functions 
\begin{gather*}
C\!\!:\ti{U}_{\T,\f}^{\phi}(\cC)\lra\R^+ \qquad\hbox{and}\\
\fR_{\up;e}\!:  \Si_{\up;e}\big(\de_{\dbar}(\u)\big)  \lra T_{u(x_{\lr{e}}(\u))}X
~~\hbox{with}~\up\!\in\!\De_{\cC}^{\!\si}\!\cap\!\ti\F_1\T|_{\u},\,
\u\!\in\!\ti{U}_{\T,\f}^{\phi}(\cC),\,e\!\in\!\E_\bu(\T)
\end{gather*}
respecting the $\ti\Aut^*(\T)$-actions such~that 
\begin{gather}
\label{ZeExpan_e}
\ze_{\up;e}\big(z_{e;\up}\big)=\ze_{\up;e}(0)
+\tnd_{x_{\lr{e}}^0(\u)}u_0\big(\rho_{e;\u}(z_{e;\up})\big)
 +\fR_{\up;e}\big(z_{e;\up}\big)\,,\\
\notag
\|\ze_{\up}\|_{\up,p,1}\le C(\u)|\up|^{1/p},~~
 \big|\fR_{\up;e}(z_{e;\up})\big|\le C(\u)
\big|\rho_{e;\u}\big(z_{e;\up}\big)\big||z_{e;\up}|,~~
 \big|\tnd_{z_{e;\up}}\!\fR_{\up;e}\big|\le C(\u)
 \big|\rho_{e;\u}\big(z_{e;\up}\big)\big|
 \end{gather}
for all $z_{e;\up}\!\in\!\Si_{\up;e}(\de_{\dbar}(\u))$, $\up\!\in\!\De_{\cC}^{\!\si}\!\cap\!\ti\F_1\T|_{\u}$, 
$\u\!\in\!\ti{U}_{\T,\f}^{\phi}(\cC)$, and $e\!\in\!\E_\bu(\T)$.
\end{lmm}

\begin{proof}
Since $\ti{u}_{\up}|_{\wh\Si_{\up}^c}$ is constant, this lemma concerns only the smoothings of
the nodes of~$\Si_{\u}^0$ and the restrictions of~$q_{\up}$ and~$\ti{u}_{\up}$
to $\Si_{\up}^0\!\subset\!\Si_{\up}$.
The restriction $q_{\up}|_{\Si_{\up}^0}$ can be written as the basic gluing map 
of \cite[Section~2.2]{gluing}. 
Since $u|_{\Si_{\u}^0}$ is a regular map, \cite[Lemma~3.2]{g1comp}
provides a family of maps~\e_ref{ti1up_e} 
satisfying the first bound in Lemma~\ref{LMExpansion}.\\

\noindent
The smooth function $\fR_{\up;e}$ is defined by~\e_ref{ZeExpan_e}.  
The two bounds on~$\fR_{\up;e}$ follow directly from the proofs of \cite[Lemmas 3.4,3.5]{g1comp}.
The only difference is the presence of the 0-th order term $\ze_{\up;e}(0)$ in the current setting;
it does not appear in \cite[(3.27)]{g1comp} because of the vanishing condition imposed
on the elements of $\Ga_+(\up)$  in  \cite[(3.2)]{g1comp}.
By~\cite[Lemma 3.5]{gluing} and the first bound of Lemma~\ref{LMExpansion},
$$\big\|\ze_{\up}\big\|_{\up,C^0}\le C_1(\u)|\up|^{1/p}\,.$$
Along with~\e_ref{zeup1_e} and~\e_ref{tize}, this implies~that 
\BE{0thterm}
\big|\ze_{\up;e}(0)\big|\le C_2(\u)|\up|^{1/p}\,.\EE
Thus, all estimates from~\cite[Lemmas 3.4,3.5]{g1comp} apply in the current setting.
\end{proof}

\subsection{The obstructed gluing step}
\label{ObsGlue_subs}

We now move to the second stage of the gluing construction.
For each $\up\inn\De_{\cC}^{\!\si}$ sufficiently small, we will construct 
an approximately $J$-holomorphic map
\BE{dbarbnd_e}  u_{\bu;\up}\!:\Si_{\up}\lra X
\qquad\hbox{s.t.}\quad
\big\|\dbar_Ju_{\bu;\up}\big\|_{\up,p}\le 
C_0\big(\pi_{\ti\F\T}(\up)\big)\big|\rho^\si_\T(\up)\big|\EE
for some continuous function $C_0\!:\ti{U}_{\T,\f}^{\phi}(\cC)\!\lra\!\R^+$.
This map is analogous to 
the approximately $J$-holomorphic map constructed just before \cite[Lemma~3.5]{LiT}.
It can be written explicitly in this case as each map~$\ti{u}_{\up_1}$ 
is non-constant on only one irreducible component~$\Si_{\up_1;0}$ of its domain
and~$\Si_{\up_1;0}$ is smooth.\\

Let $\be\!:\R^+\!\lra\![0,1]$ be a smooth cutoff function such that
$$\be(r)=\begin{cases}
1,&\hbox{if}~r\!\le\!1/2;\\
0,&\hbox{if}~r\!\ge\!1.
\end{cases}$$
For each $\ep\!\in\!\R^+$, we define 
$\be_{\ep}\!\in\!C^{\i}(\R;\R)$ by $\be_{\ep}(r)\!=\!\be(r/\ep)$.
For $\up\!\in\!\De^{\!\si}|_{\u}$ and $e\!\in\!\E_{\bu}(\T)$, define
$$\ti q_{\up;e}\!:
q_{\up}^{-1}\big(\Si_{\u;e}(\de_{\dbar}(\u))\big)\lra 
q_{\up_1}^{-1}\big(\Si_{\u;e}^0(\de_{\dbar}(\u))\big) \quad\hbox{by}\quad
z_{e}\big(\ti q_{\up;e}(x)\big)=\be_{\de_{\dbar}(\u)}(|z_e^c(x)|)z_e(x);$$
in particular, 
$$\ti q_{\up;e}(x)=\begin{cases}
q_{\bu;\up}(x),&\hbox{if}~
x\!\in\!q_{\up}^{-1}(\Si_{\u;e}^0(\de_{\dbar}(\u))\!-\!\Si_{\u;e}(2\sqrt{|\up_e|}));\\
x_e(\up_1),&\hbox{if}~x\!\in\!q_{\up}^{-1}(\prt\Si_{\u;e}^c(\de_{\dbar}(\u))).
\end{cases}$$
If in addition $\u\!\in\!\ti{U}_{\T,\f}^{\phi}(\cC)$, define
$$u_{\bu;\up}\!:\Si_{\up}\lra X, \quad
u_{\bu;\up}(x)=\begin{cases}
\ti{u}_{\up_1}(q_{\bu;\up}(x)),&\hbox{if}~
x\!\in\!q_{\up}^{-1}(\Si_{\u;e}^0),~x\!\not\in\!
q_{\up}^{-1}(\Si_{\u;e}(2\sqrt{|\up_e|}))~\forall\,e\!\in\!\E_{\bu}(\T);\\
\ti{u}_{\up_1}(\ti q_{\up;e}(x)),&\hbox{if}~x\!\in\!q_{\up}^{-1}(\Si_{\u;e}(\de_{\dbar}(\u))),~
e\!\in\!\E_{\bu}(\T);\\
\ti{u}_{\up_1}(x_e(\up_1)),&\hbox{if}~
x\!\in\!q_{\up}^{-1}(\Si_{\u;e}^c\!-\!\Si_{\u;e}(\de_{\dbar}(\u))),~e\!\in\!\E_{\bu}(\T).
\end{cases}$$
The support of $\dbar_Ju_{\bu;\up}$ is contained in the annuli 
$q_{\up}^{-1}(\Si_{\u;e}^c(\de_{\dbar}(\u)))$ with $e\!\in\!\E_{\bu}(\T)$.
On these annuli, the map~$\ti q_{\up;e}$ is equivalent to 
 the modified gluing map~$\ti{q}_{\up_0;2}$  of \cite[Section~4.2]{g1comp}.
The bound in~\e_ref{dbarbnd_e} thus follows from Lemma~\ref{LMExpansion}
as in the proof of the first estimate in \cite[Lemma~4.4(3)]{g1comp}.\\

For each $\up\!\in\!\De^{\!\si}_{\cC}$ and $\xi\!\in\!\ker D_{\ti{u}_{\up_1}}^{\phi}$,
we define $R_{\up}\xi\!\in\!\Ga(u_{\bu;\up})^{\phi}$ as in the above construction
of~$u_{\bu;\up}$.
Let
$$\Ga_-\big(u_{\bu;\up}\big)^\phi=\big\{R_{\up}\xi\!:\,
\xi\!\in\!\ker D_{\ti{u}_{\up_1}}^{\phi}\big\}$$
and denote~by 
$$\Ga_+\big(u_{\bu;\up}\big)^\phi\subset\Ga\big(u_{\bu;\up}\big)^\phi$$
the $L^2$-orthogonal complement of $\Ga_-(u_{\bu;\up})^\phi$.
Analogously to the third estimate in \cite[Lemma~4.4(3)]{g1comp},
\BE{Dphiup_e}
C_1\big(\pi_{\ti\F\T}(\up)\big)^{-1}\|\ze\|_{\up,p,1}
\le \big\|D^{\phi}_{u_{\bu;\up}}\ze\big\|_{\up,p}
\le  C_1\big(\pi_{\ti\F\T}(\up)\big)\|\ze\|_{\up,p,1}
\qquad\forall~\ze\!\in\!\Ga_+\big(u_{\bu;\up}\big)^\phi\EE
for some continuous function $C_1\!:\ti{U}_{\T,\f}^{\phi}(\cC)\!\lra\!\R^+$.
The second estimate in \cite[Lemma~4.4(3)]{g1comp}
also applies in this case for the same reasons as before.

\begin{lmm}\label{LMnbhd}
If $\De^{\!\si}$, $\ti{U}_{\T,\f}^{\phi}(\cC)$, and 
$\de_{\bu}\!:\ti{U}_{\T,\f}^{\phi}(\cC)\!\lra\!\R^+$ are sufficiently small, 
then 
\begin{gather*}
 \Phi_{\T}\!:\big\{(\up,\ze)\!: \,
\up\inn\De_{\cC}^{\!\si},~\ze\inn\Ga_+\big(u_{\bu;\up}\big)^\phi,~
	\|\ze\|_{\up,p,1}\!<\!\de_{\bu}\big(\pi_{\ti\F\T^\si}(\up)\big)\,\big\}\lra 
\X_{g,l}(X,B)^{\phi},\\
(\up,\ze) \lra 
\big[\up(\ze)\big]\equiv\big[\exp_{u_{\bu};\up}\!\ze,\big(z^+_i(\up),z^-_i(\up)\big)_{i=1}^l\big],
\end{gather*}
is an $\ti{\Aut}^*(\T)$-invariant degree $|\ti{\Aut}^*(\T)|$ covering of 
an open neighborhood~$\X_{\T;\f}(\cC)$ of $\fM_{\T,\f}(\cC)^\phi$.
\end{lmm}

\begin{proof}
By the choice of $\Ga_-(u_{\bu;\up})^\phi$, this is essentially an inverse function theorem.
The continuity, injectivity, and surjectivity of the induced map on the quotient by  $\ti{\Aut}^*(\T)$
are proved 
by arguments similar to Sections~4.1, 4.2, and 4.3-4.5 in~\cite{gluing}, 
respectively; see also the paragraph following Lemma~4.4
in~\cite{g1comp}.
\end{proof}

\noindent
Let $\na^J$ be the $J$-linear connection corresponding to~$\na^X$. 
For each $\up\inn\De^{\!\si}_\cC$ and $\ze\inn\Ga(u_{\bu;\up})^\phi$, let
$$\Pi_\ze\!: \Ga^{0,1}\big(u_{\bu;\up}\big)^\phi\lra
 \Ga^{0,1}_{\fJ_{\up}}\big(\exp_{u_{\bu;\up}}\!\ze\big)^{\phi,\si_{\up}}$$
be the isomorphism induced by the $\na^J$-parallel transport along the $\na^X$-geodesics
$$s\lra \exp_{u_{\bu;\up}}\!\!\big(s\ze\big),\qquad s\inn[0,1].$$
With $\nu$ as in Proposition~\ref{PrpSection}, define
\BE{Phidfn_e}
\Xi_{\T;t\nu}(\up,\cdot)\!: \Ga\big(u_{\bu;\up}\big)^\phi\lra\Ga^{0,1}\big(u_{\bu;\up}\big)^\phi
\quad\hbox{by}\quad
 \Xi_{\T;t\nu}(\up,\ze)=\Pi^{-1}_{\ze}\circ
\big(\{\Bar\prt_J\!+\!t\nu\}\exp_{u_{\bu;\up}}\!\ze\big)\,.\EE
Similarly to \cite[(3.9),(3.10)]{gluing},
\BE{Quad_e}
\Xi_{\T;t\nu}(\up,\ze)=\dbar_Ju_{\bu;\up}
+t\nu_{u_{\bu;\up}}+D^{\phi}_{u_{\bu;\up}}\ze+N^\phi_{\up;t}(\ze),\EE
with the quadratic term $N_{\up;t}$ satisfying
\begin{gather}\label{N_e}
N^\phi_{\up;t}(0)=0, \quad
\big\|N^\phi_{\up;t}(\ze)-N^\phi_{\up;t}(\ze')\big\|_{\up,p} 
\le C_2\big(\pi_{\ti\F\T}(\up)\big)\big(|t|\!+\!\|\ze\|_{\up,p,1}\!+\!\|\ze'\|_{\up,p,1}\big)
\big\|\ze\!-\!\ze'\big\|_{\up,p,1}\\
\forall\quad t\!\in\R,~\ze,\ze'\in\!\Ga\big(u_{\bu;\up}\big)^{\phi} ~~\st~~ 
\|\ze\|_{\up,p,1},\|\ze'\|_{\up,p,1}\le\de_{\bu}\big(\pi_{\ti\F\T}(\up)\big)\notag
\end{gather}
for some continuous function $C_2\!:\ti{U}_{\T,\f}^{\phi}(\cC)\!\lra\!\R^+$.\\

We will show that $\Xi_{\T;t\nu}$ does not vanish if $\nu$ is generic, 
$\T\!=\!(\ft,\fd)$ with $\ft$ not basic, 
$[\up(\ze)]$ is sufficiently close to~$\fM_{\T,\f}(\cC)^\phi$,
and $t\!\in\!\R^+$ is sufficiently small. 
In the case $\ft$ is basic, 
we will construct a section~$\Psi_{\T;t\nu}$ as in~\e_ref{PsiDfn_e} so that 
its vanishing locus corresponds to the vanishing locus of~$\Xi_{\T;t\nu}$,
after the constraints~$\f$ are taken into~account.\\

For any functions $\de_{\T},\ep_{\T}\!: \ti{U}_{\T,\f}^{\phi}(\cC)\!\lra\!\R^+$, let 
$$\Om_{\de_{\T},\ep_{\T}}=
\big\{(\up,\ze)\!: \,\up\inn\De_{\cC}^{\!\si},~\ze\inn\Ga_+\big(u_{\bu;\up}\big)^\phi,~
	|\up|\!<\!\de_{\T}\big(\pi_{\ti\F\T^\si}(\up)\big),~
	 \|\ze\|_{\up,p,1}\!<\!\ep_{\T}\big(\pi_{\ti\F\T^\si}(\up)\big)\big\}.$$
By~\e_ref{Quad_e}, \e_ref{dbarbnd_e}, \e_ref{Dphiup_e}, and~\e_ref{N_e}, 
there exist continuous functions 
$$\de_{\T},\ep_{\T},C_{\T}\!: \ti{U}_{\T,\f}^{\phi}(\cC)\lra \R^+$$ 
such~that 
\BE{ZeBound}\begin{split}
&\big\{(t;\up,\ze)\!\in\!\R\!\times\!\Om_{\de_{\T},\ep_{\T}}\!:
	|t|\!<\!\de_{\T}\big(\pi_{\ti\F\T^\si}(\up)\big),~
	\Xi_{\T;t\nu}(\up,\ze)\!=\!0\big\}\\
&\hspace{.5in}\subset 
\big\{(t;\up,\ze)\!\in\!\R\!\times\!\Om_{\de_{\T},\ep_\T}\!:\,
	\|\ze\|_{\up,p,1}\!\le\!C_{\T}\big(\pi_{\ti\F\T^\si}(\up)\big)
	\big(|\rho_{\T}^{\si}(\up)|\!+\!|t|\big)
	\big\}.
\end{split}\EE
We can assume that 
$$\big\{\up\!\in\!\ti\F\T^{\phi}\big|_{\ti{U}_{\T,\f}^{\phi}(\cC)}\!:
|\up|\!<\!\de_{\T}\big(\pi_{\ti\F\T^\si}(\up)\big)\big\}\subset \De_{\cC}^{\!\si},\quad
\ep_{\T}\!<\!\de_{\bu}, \quad \de_{\T},\ep_{\T}<1, \quad 
2C_{\T}\de_{\T}<\ep_{\T}\,,$$ 
and that the functions $\de_{\T},\ep_{\T},C_{\T}$ are $\ti{\Aut}^*(\T)$-invariant. \\

\noindent
Let $e\inn\E_\bu(\T)$, $\u\!\in\!\ti{U}_{\T,\f}^{\phi}(\cC)$,
and $\up\!\in\!\De^{\!\si}_\cC|_{\u}$. 
We define a holomorphic map
\begin{gather*}
q_{\up;e}\!:\, \wh\Si_{\up;e}^c\big(\de_{\dbar}(\u)\big)\!\equiv\! 
q_{\up}^{\,-1}\big(\wh\Si_{\u;e}^c(\de_{\dbar}(\u))\big)  
\lra\wh\Si_{\up_1;e}^c\subset\Si_{\up_1} \qquad\hbox{by}\\
z_e^c\!\circ\!q_{\up;e}\big|_{q_{\up}^{-1}(\Si_{\u;e}(\de_{\dbar}(\u)))}
=z_e^c\big|_{q_{\up}^{-1}(\Si_{\u;e}(\de_{\dbar}(\u)))}\,, \quad
q_{\up;e}\big|_{\wh\Si_{\up;e}^c-q_{\up}^{-1}(\Si_{\u;e}(2\sqrt{|v_e|}))}
=q_{\bu;\up}\big|_{\wh\Si_{\up;e}^c-q_{\up}^{-1}(\Si_{\u;e}(2\sqrt{|v_e|}))}\,.
\end{gather*}
For each holomorphic $(1,0)$-differential $\ka\!\in\!\bE_e|_{\Si_{\up_1}}$ 
supported on $\wh\Si_{\up_1;e}^c$, 
we define a $(1,0)$-differential $R_{\bu;\up}\ka$ on~$\Si_{\up}$~by
$$R_{\bu;\up}\ka\big|_x
=\begin{cases}
\ka\!\circ\!\tnd_xq_{\up;e},&\hbox{if}~
x\!\in\!q_{\up}^{-1}(\wh\Si_{\u;e}^c\!-\!\Si_{\u;e}(\de_{\dbar}(\u)));\\
\be_{\de_{\dbar}(\u)}(|z_e(x)|)\ka\!\circ\!\tnd_xq_{\up;e},
&\hbox{if}~x\!\in\!q_{\up}^{-1}(\Si_{\u;e}(\de_{\dbar}(\u)));\\
0,&\hbox{if}~x\!\in\!\Si_{\up}\!-\!q_{\up}^{-1}(\wh\Si_{\u;e}^c(\de_{\dbar}(\u))).
\end{cases}$$
Combining this construction with a continuous collection of isomorphisms
$$R_{\up_1;e}\!: \bE_e|_{\Si_{\u}}\lra\bE_e|_{\Si_{\up_1}}$$
as above \cite[(4.9)]{g1comp}, for each $\ka\!\in\!\bE_e|_{\u}$
we obtain a $(1,0)$-differential $R_{\bu;\up}\ka$ on~$\Si_{\up}$.
Analogously to~\cite[(4.10)]{g1comp}, for each $q\inn[1,2)$
there exists a continuous function $C_q\!:\ti{U}_{\T,\f}^{\phi}(\cC)\!\lra\!\R^+$ 
such~that 
\BE{Rupka_e}
\big\|R_\up\ka\big\|_{\up,q}\le 
C_q\big(\u\big)\|\ka\|\qquad\quad
 \forall~\up\inn\De^{\!\si}_\cC|_{\u},~\ka\!\in\!\bE_e|_{\u},~e\!\in\!\E_{\bu}(\T)\,.\EE
The norm on the left-hand side of~\e_ref{Rupka_e} is the usual Sobolev norm with respect
to the metric~$g_{\up}$;
the norm on the right-hand side of~\e_ref{Rupka_e} is any fixed norm on the finite-rank vector 
bundle $\bE_e\!\lra\!\ti{U}_{\T,\f}^{\phi}(\cC)$.\\

With $e$, $\u$, and $\up\!\in\!\De^{\!\si}_\cC|_{\u}$ as above
and $\ze_{\up_1;e}$ as in~\e_ref{tize}, let
$$\ti\ze_{\up;e}\!=\!\ze_{\up_1;e}\!\circ\!\ti q_{\up;e}(x)\!:
\wh\Si_{\up;e}^c\big(\de_{\dbar}(\u)\big)\lra T_{u(x_{\lr{e}}(\u))}X.$$
Define 
\begin{gather*}
\Th_{\up}\!: \Ga^{0,1}\big(u_{\bu;\up}\big)^{\phi}\lra\Obs_{\T}^{\phi}\big|_{\u}
\qquad\hbox{by}\\
\big\{\Th_{\up}(\eta)\big\}\big((\ka_e)_{e\in\E_\bu(\T)}\big)
=\bigg(\frac{\fI}{2\pi}\!\int_{\wh\Si_{\up;e}^c\de_{\dbar}(\u)}\!\!\!\!
\big(\R_{\up}\ka_e\big)\!\!\wedge\!\!\big(\Pi_{\ti\ze_{\up;e}}^{-1}\eta\big)\!\!
\bigg)_{\!\!e\in\E_\bu(\T)}
\quad\forall~\ka_e\!\in\!\bE_e|_{\u},~e\!\in\!\E_\bu(\T)\,.
\end{gather*}
By H\"older's inequality and~\e_ref{Rupka_e}, the above integral is finite and
\BE{Thbd_e} \big\|\Th_{\up}(\eta)\big\|\le C_{\Th}(\u)\|\eta\|_{\up,p}
\qquad\forall~\eta\!\in\!\Ga^{0,1}\big(u_{\bu;\up}\big)^{\phi}\EE
for some continuous function $C_{\Th}\!:\ti{U}_{\T,\f}^{\phi}(\cC)\!\lra\!\R^+$.\\

It is immediate~that 
\BE{nubnd_e} 
\big\|\Th_{\up}\big(\nu_{u_{\bu;\up}}\big)-\ov\nu_{\T}\big(\pi_{\ti\F\T}(\up)\big)\big\|
\le\ep_\nu(\up)\EE
for some continuous function $\ep_{\nu}\!:\De^{\!\si}_\cC\!\lra\!\R^+$ that vanishes along
$\ti{U}_{\T,\f}^{\phi}(\cC)$.
By the proof of \cite[Lemma~4.4(7)]{g1comp},
there exists a continuous function $\ep_1\!:\De^{\!\si}_\cC\!\lra\!\R^+$ vanishing along
$\ti{U}_{\T,\f}^{\phi}(\cC)$ such that 
\BE{DFbnd_e}
\big\|\Th_{\up}\big(D^{\phi}_{u_{\bu;\up}}\ze\big)\big\| \le \ep_1(\up)\|\ze\|_{\up,p,1}
\qquad\forall~\ze\!\in\!\Ga(u_{\bu;\up})^{\phi}\,.\EE
By Lemma~\ref{LMExpansion} and the proof of~\cite[Lemma 4.4(6)]{g1comp},
there exists a continuous function $\ep_{\dbar}\!:\De^{\!\si}_\cC\!\lra\!\R^+$ vanishing along
$\ti{U}_{\T,\f}^{\phi}(\cC)$ such~that 
\BE{Term1_e}
\big\|\Th_{\up}\big(\Bar\prt_Ju_{\bu;\up}\big)-\cD_\T^\phi\rho^\si_\T(\up)\big\|
\le \ep_{\dbar}(\up)\big|\rho^\si_\T(\up)\big|\,;\EE
the 0-th order term $\ze_{\up;e}(0)$ appearing in the expansion of 
Lemma~\ref{LMExpansion} has no effect on~\cite[(4.24)]{g1comp}.
By~\e_ref{Quad_e}, \e_ref{N_e},  and \e_ref{Thbd_e}-\e_ref{Term1_e},
there exist continuous functions
$$C_{\Xi}\!:\ti{U}_{\T,\f}^{\phi}(\cC)\lra\R^+ \qquad\hbox{and}\qquad
\ep_{\Xi}\!:\De^{\!\si}_\cC\lra\R^+$$
with $\ep_{\Xi}$  vanishing along $\ti{U}_{\T,\f}^{\phi}(\cC)$ such~that 
\BE{ThEst_e}\begin{split}
\big\|\Th_{\up}\big(\Xi_{\T;t\nu}(\up,\ze)\big)-
\big(\cD_\T^\phi\rho^\si_\T(\up)\!+\!t\ov\nu_{\T}(\pi_{\ti\F\T^{\si}}(\up))\big)
\big\|
&\le
C_{\Xi}\big(\pi_{\ti\F\T}(\up)\big)\big(|t|\!+\!\|\ze\|_{\up,p,1}\big)\big\|\ze\big\|_{\up,p,1}\\
&\qquad+\ep_{\Xi}(t,\up)\big(\big|\rho^\si_\T(\up)\big|\!+\!|t|\!+\!\|\ze\|_{\up,p,1}\big)
\end{split}\EE
for all $t\!\in\!\R$, $\up\!\in\!\De_{\cC}^{\!\si}$, and 
$\ze\!\in\!\Ga(u_{\bu;\up})^{\phi}$ with 
$\|\ze\|_{\up,p,1}\!\le\!\de_{\bu}(\pi_{\ti\F\T}(\up))$.\\

By~\e_ref{ZeBound} and~\e_ref{ThEst_e},
there exists a continuous function 
$$\ep_{\T;\nu}\!:\R\!\times\!\De^{\!\si}_\cC\lra\R^+$$ 
vanishing along $\{0\}\!\times\!\ti{U}_{\T,\f}^{\phi}(\cC)$ such~that 
\BE{VanBd_e}\begin{split}
&\big\|\cD_\T^\phi\rho^\si_\T(\up)\!+\!t\ov\nu_{\T}\big(\pi_{\ti\F\T^{\si}}(\up)\big)\big\|
\le \ep_{\T;\nu}(t,\nu) \big(\big|\rho^\si_\T(\up)\big|\!+\!|t|\big)\\
&\hspace{1in}
\forall~
(t;\up,\ze)\!\in\!\R\!\times\!\Om_{\T;\de_{\T},\ep_{\T}}
~~\hbox{s.t.}~~
|t|\!<\!\de_{\T}\big(\pi_{\ti\F\T^\si}(\up)\big),~\Xi_{\T;t\nu}(\up,\ze)\!=\!0.
\end{split}\EE
Suppose $\T\!=\!(\ft,\fd)$ and $\ft$ is not basic so that  
$$\dim\,\ti\M_{\T,\f}(\cC)^\phi+\rk\,\ti\fF\T^{\si}   
<\dim\,\ti\M_{\T,\f}(\cC)^\phi+\rk\,\ti\F\T^{\si}
=\rk\,\Obs_\T^{\phi}.$$ 
Thus, the bundle map
$$\ti\F\T^{\si}\lra \Obs_\T^{\phi}, \qquad 
\up\lra \cD_\T^\phi\rho^\si_\T(\up)\!+\!\ov\nu_{\T}\big(\pi_{\ti\F\T}(\up)\big),$$
has no zeros over $\ti\M_{\T,\f}(\cC)^\phi$
for a generic choice of $\nu\!\in\!\A_{g,l}^{\phi}(J)$. 
By~\e_ref{VanBd_e} and the proof of \cite[Lemma~3.2]{g2n2and3},
for every precompact open subset $\ti{K}\!\subset\!\ti\M_{\T,\f}(\cC)^\phi$ 
there thus exist $\de_{\nu}(\ti{K})\!\in\!\R^+$ and an open neighborhood 
$\Om_{\nu}(\ti{K})$ of~$\ti{K}$ in $\Om_{\de_{\T},\ep_{\T}}$
such~that $\Xi_{\T;t\nu}$ does not vanish on~$\Om_{\nu}(\ti{K})$.
By Lemma~\ref{LMnbhd} and~\e_ref{Phidfn_e}, 
$$ U_{\nu}(\ti{K}) \equiv \X_{g,\f}(X,B)^\phi \cap
\bigg(\Phi_{\T}\big(\Om_{\nu}(\ti{K})\big)\!\times\!\prod_{i=1}^lY_i\bigg)$$
is an open neighborhood of $K\!=\!\io_{\T;\f}(\ti{K})$ in
 $\X_{g,\f}(X,B)^\phi$ satisfying~\e_ref{esetinter_e}
with $K$ replaced by~$\ti{K}$.
This establishes the first statement of Proposition~\ref{PrpSection}.\\

From now on we assume that $\T\!=\!(\ft,\fd)$ with $\ft$ basic.
Thus, the domain of each element of $\ti\fM_{\T}(X)^\phi$ consists 
of a smooth genus~$h$ curve~$\Si_{\u;0}$ with smooth positive-genus curves 
$$\wh\Si_{\u;e}^c=\Si_{\u;e}^c  \qquad\hbox{for}\quad e\in \E_{\bu}(\T)\!=\!\E_0(\T)$$ 
attached at distinct points.
Furthermore, $\ti\fF\T\!=\!\ti\F\T$ and $\rho_{\T}^{\si}(\up)\!=\!\up$.
The gluing construction now consists of only the second, obstructed step.\\

For each $\u\!\in\!\ti\fM_{\T}(X)^\phi$, let
$$\Ga_-^{0,1}(u)\subset \Ga\big(\Si;T^*\Si_u^{0,1}\!\otimes\!u^*(TX,J)\big)
\qquad\hbox{and}\qquad
\Ga_-^{0,1}(u)^{\phi}\equiv \Ga_-^{0,1}(u)\cap\!\Ga^{0,1}(u)^{\phi}$$
denote the subspace of harmonic $(0,1)$-forms and 
the subspace of $\phi$-invariant harmonic $(0,1)$-forms.
The former is generated by harmonic forms each of which is supported
on~$\wh\Si_{\u;e}^c$ for some $e\!\in\!\E_{\bu}(\T)$.
The homomorphism~\e_ref{Thudfn_e} restricts to an isomorphism
$$\Th_{\u}\!: \Ga_-^{0,1}(u)^{\phi}\lra \Obs_{\T}^{\phi}\big|_{\u}\,.$$
For each $\up\!\in\!\De_{\cC}^{\!\si}|_{\u}$, $e\!\in\!\E_{\bu}(\T)$, and 
$\eta\!\in\!\Ga_-^{0,1}(u)$ supported on $\Si_{\u;e}^c$, 
we define an element $R_{\up}\eta$ in $\Ga^{0,1}(u_{\bu;\up})$ by
$$R_{\up}\eta\big|_x=
\begin{cases}
\eta\!\circ\!\tnd_xq_{\up;e},&\hbox{if}~
x\!\in\!q_{\up}^{-1}(\Si_{\u;e}^c\!-\!\Si_{\u;e}(\de_{\dbar}(\u)));\\
\be_{\de_{\dbar}(\u)}(|z_e(x)|)\Pi_{\ti\ze_{\up;e}(x)}\!\circ\!\eta\!\circ\!\tnd_xq_{\up;e},
&\hbox{if}~x\!\in\!q_{\up}^{-1}(\Si_{\u;e}(\de_{\dbar}(\u)));\\
0,&\hbox{if}~x\!\in\!\Si_{\up}\!-\!q_{\up}^{-1}(\wh\Si_{\u;e}^c(\de_{\dbar}(\u))).
\end{cases}$$
This definition induces an injective homomorphism
$$R_{\up}\!:\Ga_-^{0,1}(u)^{\phi}\lra \Ga^{0,1}\big(u_{\bu;\up}\big)$$
that satisfies
\BE{Rvbnd_e}
\big\|R_{\up}\eta\big\|_{\up,p}\le C_R(\u)\|\eta\|,\quad
\big\|\Th_{\up}(R_{\up}\eta)-\Th_{\u}(\eta)\big\|
\le C_R(\u)|\up|^2\|\eta\| 
\qquad\forall~\eta\!\in\!\Ga_-^{0,1}(u)\EE
for some continuous function $C_R\!:\ti{U}_{\T,\f}^{\phi}(\cC)\!\lra\!\R$.\\

For each $\up\!\in\!\De_{\cC}^{\!\si}$, let
$$\Ga^{0,1}_-\big(u_{\bu;\up}\big)^{\phi}=
\big\{R_{\up}\eta\!:\,\eta\!\in\!\Ga_-^{0,1}(u)^{\phi}\big\}
\quad\hbox{and}\quad
\Ga^{0,1}_+\big(u_{\bu;\up}\big)^{\phi}=
\big\{\eta\!\in\!\Ga^{0,1}\big(u_{\bu;\up}\big)\!:\,\Th_{\up}(\eta)\!=\!0\big\}.$$
By~\e_ref{Rvbnd_e},
\BE{Ga01split_e}  
\Ga^{0,1}\big(u_{\bu;\up}\big)=
\Ga^{0,1}_-\big(u_{\bu;\up}\big)^{\phi}\oplus\Ga^{0,1}_+\big(u_{\bu;\up}\big)^{\phi}\EE
for all $\up$ sufficiently small;
by shrinking $\De_{\cC}^{\!\si}$, we can assume 
that \e_ref{Ga01split_e} holds for all $\up\!\in\!\De_{\cC}^{\!\si}$.
Denote~by 
$$\pi^{0,1}_{\up;+}\!: \Ga^{0,1}\big(u_{\bu;\up}\big)\lra \Ga^{0,1}_+\big(u_{\bu;\up}\big)^{\phi}$$
the  projection to the second component in the decomposition~\e_ref{Ga01split_e}.
By~\e_ref{Thbd_e} and \e_ref{Rvbnd_e}, there exists a continuous function
$C_+\!:\ti{U}_{\T,\f}^{\phi}(\cC)\!\lra\!\R^+$ such~that
\BE{piplusbd_e} \big\|\pi^{0,1}_{\up;+}\eta\big\|_{\up,p}\le 
C_+\big(\pi_{\ti\F\T^{\si}}(\up)\big)\|\eta\|_{\up,p}
 \qquad\forall~\eta\!\in\! \Ga^{0,1}\big(u_{\bu;\up}\big)\,.\EE
By~\e_ref{Dphiup_e}, \e_ref{DFbnd_e}, \e_ref{Rvbnd_e}, and~\e_ref{piplusbd_e}, 
 there exists a continuous function
such~that $\ti{C}_1\!:\ti{U}_{\T,\f}^{\phi}(\cC)\!\lra\!\R^+$ such~that 
\BE{Dphiup_e2}
\ti{C}_1\big(\pi_{\ti\F\T}(\up)\big)^{-1}\|\ze\|_{\up,p,1}
\le \big\|\pi^{0,1}_{\up;+}D^{\phi}_{u_{\bu;\up}}\ze\big\|_{\up,p}
\le  \ti{C}_1\big(\pi_{\ti\F\T}(\up)\big)\|\ze\|_{\up,p,1}
\quad\forall~\ze\!\in\!\Ga_+(\up)^\phi\EE
for all $\up$ sufficiently small;
by shrinking $\De_{\cC}^{\!\si}$, we can assume 
that \e_ref{Dphiup_e2} holds for all $\up\!\in\!\De_{\cC}^{\!\si}$.\\

By~\e_ref{Quad_e}, the identity $\Xi_{\T;t\nu}(\up,\ze)\!=\!0$ is equivalent 
to the system of equations
\BE{Quad_e2}
\begin{cases}
\pi^{0,1}_{\up;+}D^{\phi}_{u_{\bu;\up}}\ze+\pi^{0,1}_{\up;+}N^\phi_{\up;t}(\ze)
=-\pi^{0,1}_{\up;+}\dbar_Ju_{\bu;\up}-t\pi^{0,1}_{\up;+}\nu_{u_{\bu;\up}}
\in \Ga^{0,1}_+\big(u_{\bu;\up}\big)^{\phi} \,,\\
\Th_{\up}\big(\Xi_{\T;t\nu}(\up,\ze)\big)=0\in \Obs_{\T}^{\phi}\big|_{\u}.
\end{cases}\EE
By~\e_ref{Dphiup_e2} and index considerations, the homomorphism
$$\pi^{0,1}_{\up;+}D^{\phi}_{u_{\bu;\up}}\!:
\Ga_+\big(u_{\bu;\up}\big)^\phi\lra \Ga^{0,1}_+\big(u_{\bu;\up}\big)^{\phi}$$
is an isomorphism; 
its norm and the norm of its inverse are bounded depending only on~$\pi_{\ti\F\T}(\up)$.
In~light of~\e_ref{dbarbnd_e}, \e_ref{N_e},  and~\e_ref{piplusbd_e}, 
we can thus apply the Contraction Principle to the first equation in~\e_ref{Quad_e2} 
provided $|\up|$ and~$|t|$ are sufficiently small (depending on~$\pi_{\ti\F\T}(\up)$).
More precisely, for every precompact open subset $\ti{K}\!\subset\!\ti{U}_{\T,\f}^{\phi}(\cC)$ 
there exist $\de_{\nu}(\ti{K}),\ep_{\nu}(\ti{K}),C(\ti{K})\!\in\!\R^+$ such that 
$$\de_{\nu}(\ti{K})<\inf_{\ti{K}}\de_{\T}\,, \quad
\ep_{\nu}(\ti{K})<\inf_{\ti{K}}\ep_{\T}\,, \quad
2C_{\T}\de_{\nu}(\ti{K})<\ep_{\nu}(\ti{K})\,,$$
and for all $t\!\in\!\R$ and $\up\!\in\!\ti\F\T|_{\ti{K}}$
with $|t|,|\up|\!<\!\de_{\nu}(\ti{K})$  the first equation in~\e_ref{Quad_e2} 
has a unique solution $\ze_{t\nu}(\up)\!\in\!\Ga_+(u_{\bu;\up})^\phi$ with 
$\|\ze_{t\nu}(\up)\|_{\up,p,1}\!<\!\ep_{\nu}(\ti{K})$ and this solution
satisfies
\BE{zetnubnd_e} \big\|\ze_{t\nu}(\up)\big\|_{\up,p,1}
<C(\ti{K})\big(|\up|\!+\!|t|\big).\EE
If $\ti{K}$ is preserved by the $\ti\Aut^*(\T)$-action, 
then the function $\up\!\lra\!\ze_{t\nu}(\up)$ is $\ti\Aut^*(\T)$-equivariant.\\

The~maps 
$$\ti\ga_{\T;t\nu}\!:
\ti\F\T^{\si}_{\de_{\nu}(\ti{K})}\big|_{\ti{K}}\lra\X_{g,l}(X,B)^{\phi}, \quad
\up\lra 
\big[\up(\ze_{t\nu}(\up))\big]\equiv\big[\exp_{u_{\bu};\up}\!\ze_{t\nu}(\up),
\big(z^+_i(\up),z^-_i(\up)\big)_{i=1}^l\big],$$
with $|t|\!<\!\de_{\nu}(\ti{K})$ 
are the analogues of \cite[(3.16)]{gluing} in the present situation.
Let 
$$\ti\fM_{g,l}^{(t)}(\ti{K})^{\phi}\subset \X_{g,l}(X,B)^{\phi}$$
denote the image of~$\ti\ga_{\T;t\nu}$.
Since the evaluation map
$$\ev\!: \fM_{\T}(X)^\phi \lra X^l$$
is transverse to the pseudocycles~$\f$ along $\fM_{\T;\f}(\cC)^\phi$,
the maps~$\ti\ga_{\T;t\nu}$ can be adjusted for the constraints as in the proof
of \cite[Lemma~3.28]{gluing};
this adjustment is summarized~below.\\

We denote by
$$\pi_{\ti\cN^{\mu}}\!:\ti\cN^{\mu}\lra  \fM_{\T;\f}(\cC)^\phi$$
the normal bundle of $\fM_{\T;\f}(\cC)^\phi$ in~$\fM_{\T}(X)^\phi$.
Fix an $\ti\Aut^*(\T)$-equivariant identification 
$$\ti\phi_{\T}^{\mu}\!: \ti\cN^{\mu}\lra \ti\fM_{\T}(X)^\phi$$
between neighborhoods of $\fM_{\T;\f}(\cC)^\phi$ in $\ti\cN^{\mu}$
and in~$\fM_{\T}(X)^\phi$ restricting to the identity over~$\fM_{\T;\f}(\cC)^\phi$
and identifications
$$\ti\Phi_{\T}^{\mu}\!: 
\pi_{\ti\F\T^{\si}}^*\ti\cN^{\mu}\!=\!
\pi_{\ti\cN^{\mu}}^*\ti\F\T^{\si}\lra \ti\F\T^{\si}
\quad\hbox{and}\quad
\Pi_{\T}^{\mu}\!: \pi_{\ti\cN^{\mu}}^*\Obs_{\T}^{\phi}\lra\Obs_{\T}^{\phi}$$
of vector bundles lifting~$\ti\phi_{\T}^{\mu}$, 
restricting to the identity over~$\fM_{\T;\f}(\cC)^\phi$,
and respecting the splittings.\\

Let $K\!\subset\!\fM_{\T;\f}(\cC)^\phi$  be an open subset such that 
$\io_{\T,\f}^{-1}(K)\!\subset\!\ti{K}$ is precompact.
By Lemma~\ref{LMnbhd}, \e_ref{Phidfn_e}, \e_ref{zetnubnd_e}, and  
the proof of \cite[Lemma~3.28]{gluing}, 
there exist a neighborhood $U_{\nu}(K)$ of~$K$ in $\X_{g,\f}(X,B)^{\phi}$,
$\de_{\nu}(K),C_{\nu}(K)\!\in\!\R^+$, and a unique section
$$\ti\vph_{t\nu}^{\mu}\in 
\Ga\big(\ti\F\T_{\de_{\nu}(K)}^{\phi}\big|_{\io_{\T,\f}^{-1}(K)};
\ti\pi_{\ti\F\T^{\phi}}^*\ti\cN^{\mu}\big)$$
such that 
\BE{mubd_e} \big\|\ti\vph_{t\nu}^{\mu}(\up)\big\|
<C_{\nu}(K)\big(|\up|\!+\!|t|\big)\EE
and the~map
$$\Phi_{\T;t\nu}\!\equiv\!\ti\ga_{\T;t\nu}\!\circ\!\ti\Phi_{\T}^{\mu}
\!\circ\!\ti\vph_{t\nu}^{\mu}\!: 
\ti\F\T^{\si}_{\de_{\nu}(K)}|_{\io_{\T,\f}^{-1}(K)}\lra 
\ti\fM_{g,l}^{(t)}(\ti{K})^{\phi}\cap U_{\nu}(K)$$
is an $\ti\Aut^*(\T)$-invariant covering of degree  $|\ti\Aut^*(\T)|$.
By~\e_ref{zetnubnd_e} and~\e_ref{mubd_e}, $\Phi_{\T;t\nu}$
satisfies the first condition in~\e_ref{BdCond_e}.\\

We define the section  in~\e_ref{PsiDfn_e} by
$$\Psi_{\T;t\nu}(\up)= \big\{\Pi_{\T}^{\mu}|_{\ti\vph_{t\nu}^{\mu}(\up)}\big\}^{-1}
\Big(\Th_{\ti\Phi_{\T}^{\mu}(\ti\vph_{t\nu}^{\mu}(\up))}
\big(\Xi_{\T;t\nu}\big( \ti\Phi_{\T}^{\mu}(\ti\vph_{t\nu}^{\mu}(\up)),
\ze_{t\nu}\big(\ti\Phi_{\T}^{\mu}(\ti\vph_{t\nu}^{\mu}(\up))\big)\big)\big)\Big).$$
By~\e_ref{ThEst_e}, \e_ref{zetnubnd_e}, and~\e_ref{mubd_e},  
$\Psi_{\T;t\nu}(\up)$ satisfies~\e_ref{PhiEstimate_e} with~$\ve_{\T;t\nu}$ 
satisfying the second condition in~\e_ref{BdCond_e}.
By~\e_ref{Phidfn_e}, \e_ref{Quad_e2}, and the last sentence of the previous paragraph,
\e_ref{UnuK_e} is satisfied as well.
If $\nu$ is generic, then
$$\M^*_{g,\f}(X,B;J,t\nu)^\phi=\ov\M_{g,\f}(X,B;J,t\nu)^\phi\,.$$
The last statement of Proposition~\ref{PrpSection} is obtained
as in the proof of \cite[Corollary~3.26]{gluing}.
The crucial point in both cases is that the signs of the elements of
$\Psi_{\T;t\nu}^{-1}(0)$ are determined from the orientation of the deformation-obstruction
complex over each stratum associated with the moduli space~$\ov\M_{g,l}(X,B;J)^{\phi}$.

\section{Computations and applications}
\label{SecEx}

In this section, we determine the genus~$g$ degree $d\!=\!1,3,4$  
real GW- and enumerative invariants of $(\P^3,\tau_4)$ with $d$~conjugate pairs 
of point insertion; the latter is readily obtained from the former via~\e_ref{lowgenusGV_e}.
By \cite[Theorem~1.6]{RealGWsIII} and~\e_ref{FanoGV_e}, 
the genus~$g$ degree~$d$  real GW- and enumerative invariants vanish whenever $d\!-\!g\!\in\!2\Z$.
Examples~\ref{Exd1}-\ref{Exd4} apply~\cite[Theorem~4.6]{RealGWsIII} 
to compute the genus~$g$ degree $d\!=\!1,3,4$  real GW-invariants with $d\!-\!g\!\not\in\!2\Z$
by equivariant localization with the standard $(\C^*)^2$-action on~$(\P^3,\tau_4)$. 
In \cite{RealGWvsEnumApp}, these computations are carried out for $d\!=\!5,6,7,8$.
The degree~2 GW- and enumerative invariants vanish because there are no conics
passing through two generic conjugate pairs of points in~$\P^3$;
the vanishing of the GW-invariants in this case is shown by equivariant localization
in \cite[Example~4.10]{RealGWsIII}.
The results of these computations are summarized in Table~\ref{TabInvReal}. 
We conclude by combining Theorem~\ref{main_thm} with Castelnuovo bounds to obtain
conclusions about one- and two-partition Hodge integrals.

\subsection{Preliminaries}
\label{Prelim_subs}

Let $g,k\!\in\!\Z^{\ge 0}$ with $2g\!+\!k\!\ge\!3$.
We  denote~by
$$\bE\lra \ov\cM_{g,k}$$
the Hodge vector bundle of harmonic differentials over the Deligne-Mumford moduli space
of (complex) genus~$g$ curves with $k$~marked points.
For each $i\!=\!1,\ldots,k$, let
$$\psi_i\in H^2\big(\ov\cM_{g,k}\big)$$
be the first Chern class of the universal cotangent line bundle at the $i$-th marked point.
For a formal variable~$u$, we define
$$\La(u) = \sum_{r=0}^g\!
c_r(\bE^*)u^{g-r} \in H^*\big(\ov\cM_{g,k}\big)\big[u\big].$$
This is the equivariant Euler class of $\bE^*$ tensored with the trivial line bundle 
with the equivariant first Chern class equal to~$u$.\\

For formal variables $u_1$, $u_2$, and $u_3$, define
\begin{alignat*}{3}
I_{1;0}(u_1,u_2,u_3)&=1, &\quad
I_{1;g}(u_1,u_2,u_3)&=\int_{\ov\cM_{g,1}}\!\!\!\!
 \frac{\La(u_1)\La(u_2)\La(u_3)}{u_1(u_1\!-\!\psi_1)}
&~~&\forall\,g\!\in\!\Z^+,\\
I_{2;0}(u_1,u_2,u_3)&=\frac{(u_1\!+\!u_2)u_3}{u_1u_2}, &\quad
I_{2;g}(u_1,u_2,u_3)&=\frac{(u_1\!+\!u_2)^2u_3}{u_1u_2}\!\!\int_{\ov\cM_{g,2}}\!\!
 \frac{\La(u_1)\La(u_2)\La(u_3)}{(u_1\!-\!\psi_1)(u_2\!-\!\psi_2)}
&~~&\forall\,g\!\in\!\Z^+.
\end{alignat*} 
These integrals are known in the literature as \sf{one-} 
and \sf{two-partition Hodge integrals}.
We note~that 
\begin{alignat}{2}
\label{Isymm}
 I_{1;g}(u_1,u_2,u_3)&=I_{1;g}(u_1,u_3,u_2),&\qquad
 I_{2;g}(u_1,u_2,u_3)&=I_{2;g}(u_2,u_1,u_3),\\
\label{Isymm2}
I_{1;g}(u_1,u_2,u_3)&=I_{1;g}(-u_1,-u_2,-u_3),&\qquad
I_{2;g}(u_1,u_2,u_3)&=I_{2;g}(-u_1,-u_2,-u_3)
\end{alignat}
for all $g\!\in\!\Z^{\ge 0}$.

\begin{lmm}\label{LMHodge}
For $g,k\!\in\!\Z^{\ge0}$,
\begin{alignat*}{2}
 \int_{\ov\cM_{g,k}}\!\!\!\!\frac{\La(u_1)\La(u_2)\La(u_3)}{u_1(u_1\!-\!\psi_1)}
&=\big(u_1^{-1}\big)^{k-1}I_{1;g}(u_1,u_2,u_3)&\quad&\hbox{if}~\,
2g\!+\!k\!\ge\!3,\,k\!\ge\!1,\\
 \frac{(u_1\!+\!u_2)^2u_3}{u_1u_2}\int_{\ov\cM_{g,k}}
 \frac{\La(u_1)\La(u_2)\La(u_3)}{(u_1\!-\!\psi_1)(u_2\!-\!\psi_2)}
 &=\big(u_1^{-1}\!\!+\!u_2^{-1}\big)^{k-2}I_{2;g}(u_1,u_2,u_3)&\quad
 &\hbox{if}~\,2g\!+\!k\!\ge\!3,\,k\!\ge\!2.
\end{alignat*}
\end{lmm}

\begin{proof}
For $(g,k)\!=\!(0,3)$, both identities follow from the moduli space $\ov\cM_{0,3}$
being a single point.
In the remaining cases, they are consequences of the dilaton relation as shown below.
If $k\!\ge\!2$ and $2g\!+\!k\!\ge\!4$,
\begin{equation*}\begin{split}
\int_{\ov\cM_{g,k}}\!\!\!\!
 \frac{\La(u_1)\La(u_2)\La(u_3)}{u_1(u_1\!-\!\psi_1)}
&=\sum_{s=0}^{\i}\int_{\ov\cM_{g,k}}\!\!\!\!
 \La(u_1)\La(u_2)\La(u_3)\frac{\psi_1^s}{u_1^{s+2}}
=\sum_{s=1}^{\i}\int_{\ov\cM_{g,k-1}}\!\!\!\!\!\!\!
 \La(u_1)\La(u_2)\La(u_3)\frac{\psi_1^{s-1}}{u_1^{s+2}}\\
&=\sum_{s=0}^{\i}\int_{\ov\cM_{g,k-1}}\!\!\!\!\!\!\!
 \La(u_1)\La(u_2)\La(u_3)\frac{u_1^{-1}\psi_1^s}{u_1^{s+2}}
=u_1^{-1}\!\int_{\ov\cM_{g,k-1}}\!\!\!\!\!\!\!
 \frac{\La(u_1)\La(u_2)\La(u_3)}{u_1(u_1\!-\!\psi_1)}\,.
\end{split}\end{equation*}
If $k\!\ge\!3$ and $2g\!+\!k\!\ge\!4$,
\begin{equation*}\begin{split}
\int_{\ov\cM_{g,k}}\!\!
 \frac{\La(u_1)\La(u_2)\La(u_3)}{(u_1\!-\!\psi_1)(u_2\!-\!\psi_2)}
&=\sum_{s_1,s_2\ge0}\int_{\ov\cM_{g,k}}\!\!\!\!
 \La(u_1)\La(u_2)\La(u_3)\frac{\psi_1^{s_1}\psi_1^{s_2}}{u_1^{s_1+1}u_2^{s_2+1}}\\
&=\sum_{s_1,s_2\ge0}\int_{\ov\cM_{g,k-1}}\!\!\!\!\!\!\!\!\!\!
 \La(u_1)\La(u_2)\La(u_3)
\frac{(u_1^{-1}\!\!+\!u_2^{-1})\psi_1^{s_1}\psi_1^{s_2}}{u_1^{s_1+1}u_2^{s_2+1}}\\
&=\big(u_1^{-1}\!\!+\!u_2^{-1}\big)\!\!\int_{\ov\cM_{g,k-1}}\!\!
 \frac{\La(u_1)\La(u_2)\La(u_3)}{(u_1\!-\!\psi_1)(u_2\!-\!\psi_2)}\,.
\end{split}\end{equation*}
Both of the claimed identities now follow by induction from the base cases.
\end{proof}

In the next three examples, we denote by $H\!\in\!H^2(\P^3;\Q)$ the usual hyperplane class
and use the standard $(\C^*)^2$-action on~$(\P^3,\tau_4)$. 
The fixed points of this action~are
$$P_1=[1,0,0,0], \quad P_2=[0,1,0,0], \quad P_3=[0,0,1,0], \quad P_4=[0,0,0,1];$$
they satisfy $P_1\!=\!\tau_4(P_2)$ and $P_3\!=\!\tau_4(P_4)$.
Define
$$\tau_4\!:\big\{1,2,3,4\big\}\lra\big\{1,2,3,4\big\}
\qquad\hbox{by}\quad 
P_{\tau_4(i)}=\tau_4(P_i)\,.$$
We denote~by
$$\al_1=-\al_2\qquad\tn{and}\qquad\al_3=-\al_4$$
the \sf{weights} of the standard $(\C^*)^2$-action and by 
$$\x,\bfe\big(T\P^3\big),\bc\big(T\P^3\big)\in H_{\bT^2}^*(\P^3;\Q)$$ 
the equivariant hyperplane class
and the equivariant Euler and Chern classes of~$T\P^3$. 
Thus, 
\BE{xrestr}\x|_{P_i}=\al_i,  \quad
\bc\big(T\P^3\big)\big|_{P_i}=\prod_{j\neq i}\!\big(1\!+\!\x\!-\!\al_j\big)
\qquad\forall~i\!=\!1,2,3,4.\EE
For $i\!\in\!\Z$, let 
$$\lr{i}=\begin{cases}1,&\hbox{if}~i\!\not\in\!2\Z;\\
3,&\hbox{if}~i\!\in\!2\Z.\end{cases}$$
The  genus~$g$ degree~$d$ real GW-invariant of~$(\P^3,\tau_4)$ 
with $d$~pairs of conjugate point constraints is given~by
\BE{EquivInt_e}
\GW_{g,d}^{\P^3,\tau_4}\big(\underset{d}{\underbrace{H^3,\ldots,H^3}}\big)=
\int_{[\ov\fM_{g,d}(\P^3,d)^{\tau_4}]^{\vir}}
\prod_{i=1}^d\!\bigg(\!\ev_i^{\,*}\!\!\!\prod_{j\neq\lr{i}}\!(\x\!-\!\al_j)\!\bigg).\EE
By \cite[Theorem~4.6]{RealGWsIII}, this integral is the sum over contributions 
$\Cntr_{\Ga,\si}(3^d)$
from the $(\C^*)^2$-fixed loci~$\cZ_{\Ga,\si}$ corresponding to
the elements~$(\Ga,\si)$ of the collection~$\A_{g,d}(4,d)$ of \sf{admissible pairs};
these are reviewed below.\\

Each element $(\Ga,\si)$ of~$\A_{g,d}(4,d)$ consists of
a genus~$g$ \sf{$S$-marked $[4]$-labeled decorated graph}~$\Ga$ with $S$ as in~\e_ref{SeTcond_e0}
for $l\!=\!d$
and an involution~$\si$ on~$\Ga$.
The former consists of an $S$-marked decorated graph~$\T$ as in~\e_ref{Dual_e} 
and additional functions 
$$\vt\!:\Ver\lra\big\{1,2,3,4\big\} \qquad\hbox{and}\qquad
\d\!:\Edg\!\lra\!\Z^+$$
such that $\vt(v_1)\!\neq\!\vt(v_2)$ for every edge $e\!=\!\{v_1,v_2\}$
and the values of $\d$ add up to~$d$.
The bijection~$\si$ is an involution on~$\T$ as in~\e_ref{Invol_e} 
such that $\vt\!\circ\!\si\!=\!\tau_4\!\circ\!\vt$
and $\d\!\circ\!\si\!=\!\d$.
The first condition implies that $\V_{\R}^{\si}(\Ga)\!=\!\eset$.
A pair~$(\Ga,\si)$ is \sf{admissible} if $\d(e)\!\not\in\!2\Z$ for every $e\!\in\!\E_{\R}^{\si}(\Ga)$.
Since $\ev_i|_{\cZ_{\Ga,\si}}$ is the constant map with value~$P_{\vt(\fd(i))}$,
\e_ref{xrestr} implies~that 
\BE{D1fixedlocus}
\ev_i^{\,*}\!\!\!\prod_{j\neq\lr{i}}\!\!\!\big(\x\!-\!\al_j\big)
=\prod_{j\neq\lr{i}}\!\!\!\big(\al_{\vt(\fd(i^+))}\!-\!\al_j\big)
=\begin{cases}
\bfe(T_{P_{\lr{i}}}\P^3),&\hbox{if}~\vt(\fd(i^+))\!=\!\lr{i};\\
0,&\hbox{if}~\vt(\fd(i^+))\!\neq\!\lr{i};
\end{cases}\EE
for each $i\!=\!1,\ldots,d$.\\

For each edge $e\!=\!\{v_1,v_2\}$, let
$$\psi_{e;v_1}=\frac{\al_{v_2}\!-\!\al_{v_1}}{\d(e)}\,.$$
For each $v\!\in\!\Ver$, let 
$$S_v=\E_v(\Ga) \sqcup\fd^{-1}(v)$$
be the disjoint union of the edges leaving~$v$ and the marked points carried by~$v$.
If in addition $\val_v(\Ga)\!\ge\!3$ and $e\!\in\!\E_v(\Ga)$, let
$$\psi_{v;e}\in H^2\big(\ov\cM_{\g(v),S_v}\big)$$
denote the first Chern class of the universal cotangent line bundle associated 
with the marked point indexed by~$e$.\\

If $\vt(\fd(i^+))\!=\!\lr{i}$ and  $\val_v(\Ga)\!\ge\!3$, the contribution
of the vertex~$v$ to $\Cntr_{\Ga,\si}(3^d)$ is given~by \cite[(4.18)]{RealGWsIII}
as 
\BE{Contr3V_e}
(-1)^{\fs_v}\Cntr_{\Ga,\si;v}(3^d)=-(-1)^{\g(v)+|\E_v(\Ga)|}
\bfe\big(T_{P_{\vt(v)}}\P^3\big)^{\!|S_v|-1}
 \!\!\!\!\!\!\!\!\!
\int\limits_{\ov\cM_{\g(v),S_v}}\!\!
\frac{\bfe(\bE^*\!\otimes\!T_{P_{\vt(v)}}\P^3)}
{\prod\limits_{e\in\E_v(\Ga)}\!\!\!\!\!\!\psi_{e;v}\!\left(\psi_{e;v}\!+\!\psi_{v;e}\right)}\,.\EE
By the second statement in~\e_ref{xrestr}, 
\BE{EtoLa_e}\bfe\big(\bE^*\!\otimes\!T_{P_{\vt(v)}}\P^3\big)
=\prod_{j\neq\vt(v)}\!\!\!\!\La\big(\al_{\vt(v)}\!-\!\al_j\big)\,.\EE
If $\vt(\fd(i^+))\!=\!\lr{i}$ and  $\val_v(\Ga)\!\in\!\{1,2\}$, the contribution
of the vertex~$v$ to $\Cntr_{\Ga,\si}(3^d)$ is given~by \cite[(4.19)]{RealGWsIII}
as 
\BE{Contr2V_e}\begin{split}
(-1)^{\fs_v}\Cntr_{\Ga,\si;v}(3^d)&=(-1)^{|\fd^{-1}(v)|}
\bfe\big(T_{P_{\vt(v)}}\P^3\big)^{|S_v|-1}\\
&\hspace{1in}
\times\bigg(\prod\limits_{e\in\E_v(\Ga)}\!\!\!\!\!\!\psi_{e;v}\bigg)^{\!-1}
\!\bigg(\!\sum\limits_{e\in\E_v(\Ga)}\!\!\!\!\!\psi_{e;v}\!
\bigg)^{\!3-|S_v|-|\E_v(\Ga)|}\,.
\end{split}\EE 
If $\vt(\fd(i^+))\!\neq\!\lr{i}$, then $(-1)^{\fs_v}\Cntr_{\Ga,\si;v}(3^d)\!=\!0$
by the second case of~\e_ref{D1fixedlocus}.\\

If $e\!\in\!\E_{\C}^{\si}(\T)$, the contribution
of the edge~$e$ to $\Cntr_{\Ga,\si}(3^d)$ is given~by \cite[(4.22)]{RealGWsIII}
as
\BE{ContrEC_e}\begin{split}
\Cntr_{\Ga,\si;e}&=\frac{(-1)^{\d(e)}}{\d(e)\,(\d(e)!)^2}
\frac{1}
{\left(\!\frac{\al_{\vt(v_1)}-\al_{\vt(v_2)}}{\d(e)}\!\right)^{\!\!2\d(e)-2} \!\!\!\!\!\!\!\!\!\!\!
\prod\limits_{j\neq\vt(v_1),\vt(v_2)}\prod\limits_{r=0}^{\d(e)}\!\!
\left(\!\frac{(\d(e)-r)\al_{\vt(v_1)}+r\al_{\vt(v_2)}}{\d(e)}\!-\!\al_j\!\right)}\,.
\end{split}\EE
If $e\!\in\!\E_{\R}^{\si}(\Ga)$, the contribution
of the edge~$e$ to $\Cntr_{\Ga,\si}(3^d)$ is given~by \cite[(4.23)]{RealGWsIII} as
\BE{ContrER_e}
 \Cntr_{\Ga,\si;e}=
\frac{(-1)^{\frac{\d(e)-1}{2}}}{\d(e)\d(e)!}\,
\frac{1}
{\left(\!2\frac{\al_{\vt(v_1)}}{\d(e)}\!\right)^{\!\d(e)-1}\!\!\!\!\!\!\!\!\!\!\!
\prod\limits_{j\neq\vt(v_1),\vt(v_2)}
\!\!\!\!\!\!\!\prod\limits_{r=0}^{(\d(e)-1)/2}\!\!
\left(\frac{(\d(e)-2r)\al_{\vt(v_1)}}{\d(e)}\!-\!\al_j\right)}.\EE
The description of the contribution  $\Cntr_{\Ga,\si}(3^d)$ of~$\cZ_{\Ga,\si}$ to~\e_ref{EquivInt_e}
in \cite[Theorem~4.6]{RealGWsIII} involves picking any subsets
\begin{gather*}
\V_+^{\si}(\Ga)\subset\Ver  \quad\hbox{and}\quad \E_+^{\si}(\Ga)\!\subset\!\E_{\C}^{\si}(\Ga)
\qquad\hbox{s.t.}\\
\Ver=\V_+^{\si}(\Ga) \sqcup \si\big(\V_+^{\si}(\Ga)\big)
\qquad\hbox{and}\qquad
\Edg=\E_{\R}^{\si}(\Ga) \sqcup \E_+^{\si}(\Ga) \sqcup \si\big(\E_+^{\si}(\Ga)\big).
\end{gather*}
By \cite[(4.26)]{RealGWsIII},
\BE{EquivLocal_e}\Cntr_{\Ga,\si}(3^d)=\frac{1}{|\Aut(\Ga,\si)|}
\prod_{v\in\V_+^{\si}(\Ga)}\!\!\!\!\!(-1)^{\fs_v}\Cntr_{\Ga,\si;v}(3^d)
\prod_{e\in\E_{\R}^{\si}(\Ga)\sqcup\E_+^{\si}(\Ga)}
\hspace{-.33in}\Cntr_{\Ga,\si;e}\,.\EE
We will call an element~$(\Ga,\si)$ of~$\A_{g,d}(4,d)$ \sf{contributing}
if $\vt(\fd(i^+))\!=\!\lr{i}$ for every $i\!=\!1,\ldots,d$.
By \cite[Theorem~4.6]{RealGWsIII}, the number~\e_ref{EquivInt_e}
is the sum of the rational fractions~\e_ref{EquivLocal_e} over
the contributing elements of~$\A_{g,d}(4,d)$.

\subsection{Low-degree real GW-invariants}
\label{LowDegree_subs}

We next apply the above setup to evaluate the integral in~\e_ref{EquivInt_e}
for $d\!=\!1,3,4$ and $g\!\in\!\Z^{\ge0}$ such that $d\!-\!g\!\not\in\!2\Z$.

\begin{eg}[$d\!=\!1$]\label{Exd1}
We compute the genus~$g$ degree~1 real GW-invariant of $(\P^3,\tau_4)$ 
with 1~conjugate pair of point constraints for every $g\!\in\!2\Z^{\ge0}$.
This invariant is given by~\e_ref{EquivInt_e} with $d\!=\!1$.   
If $(\Ga,\si)$ is an element of~$\A_{g,1}(4,1)$, then $\Ga$ consists of a single edge 
$e\!=\!\{v_1,v_2\}$ with 
$$e\in\E_{\R}^{\si}, \qquad \d(e)=1, \qquad \g(v_1)=\g(v_2)=\frac{g}{2},
\qquad \fd(1^+)=1, \qquad \fd(1^-)=2\,.$$
If $(\Ga,\si)$ is a contributing pair, then 
\begin{gather*}
\vt(v_1)=1, \quad \vt(v_2)=2, \quad \psi_{e;v_1}=-2\al_i, \quad 
\Cntr_{\Ga,\si;e}=\frac{1}{\al_1^2\!-\!\al_3^2}\,,\\
(-1)^{\fs_{v_1}}\Cntr_{\Ga,\si;v_1}(3^1)=\al_1^2\!-\!\al_3^2 \quad\hbox{if}~g\!=\!0\,;
\end{gather*}
the last two statements follow from~\e_ref{ContrER_e} and~\e_ref{Contr2V_e}.
If $g\!=\!2g'\!>\!0$, then \e_ref{Contr3V_e}, \e_ref{EtoLa_e}, and the first statement 
of Lemma~\ref{LMHodge} give 
\begin{equation*}\begin{split}
(-1)^{\fs_{v_1}}\Cntr_{\Ga,\si;v_1}(3^1)
&=(-1)^{g'}2\al_1\big(\al_1^2\!-\!\al_3^2\big)
\int_{\ov\cM_{g',2}}\!\!\!\!\!\!
\frac{\La(2\al_1)\La(\al_1\!-\!\al_3)\La(\al_1\!+\!\al_3)}{2\al_1(2\al_1\!-\!\psi_1)}\\
&=(-1)^{g'}\big(\al_1^2\!-\!\al_3^2\big)
I_{1;g'}\big(2\al_1,\al_1\!-\!\al_3,\al_1\!+\!\al_3\big).
\end{split}\end{equation*}
Combining the last three equations with~\e_ref{EquivLocal_e}, we obtain
\BE{d1GW_e}\GW_{g,1}^{\P^3,\tau_4}\big(H^3\big) 
=(-1)^{g/2}I_{1;g/2}
\big(2\al_1,\al_1\!-\!\al_3,\al_1\!+\!\al_3\big)\qquad\forall\,g\!\in\!2\Z^{\ge 0}\,.\EE
In particular, we obtain the numbers $\GW_{0,d}^{\phi}$, $\GW_{2,d}^{\phi}$, 
and $\GW_{4,d}^{\phi}$ listed in the $d\!=\!1$ column of Table~\ref{TabInvReal}.
\end{eg}

\begin{eg}[$d\!=\!3$]\label{Exd3}
Let  $g\!\in\!2\Z^{\ge0}$.
The genus~$g$ degree~3 real GW-invariant of $(\P^3,\tau_4)$ 
with 3~conjugate pairs of point constraints 
is given by~\e_ref{EquivInt_e} with $d\!=\!3$. 
If $(\Ga,\si)\!\in\!\A_{g,3}(4,3)$ is  a contributing pair, then
\BE{D3vt}
\vt\big(\fd(1^+)\big),\vt\big(\fd(3^+)\big)=1,\quad
\vt\big(\fd(1^-)\big),\vt\big(\fd(3^-)\big)=2,\quad
 \vt\big(\fd(2^+)\big)=3,\quad \vt\big(\fd(2^-)\big)=4.\EE
There are 4 types of admissible pairs satisfying~\e_ref{D3vt};
see Figure~\ref{Figd=3}.
The number next to each vertex~$v$ in this figure, i.e.~1,2,3, or~4,  is~$\vt(v)$,
the number next to each edge is~$\d(e)$, and 
the arrows indicate the involutions.
The labels~$g_1$ and~$g_2$ next to the top vertices are the values of~$\g$ 
on these vertices.
Since the values of~$\g$ are preserved by the involution, $g_1\!+\!g_2\!=\!g/2$.
We can take $\V_+^{\si}(\Ga)$ to consist of the two top vertices and
$\E_+^{\si}(\Ga)$ of the top edge~$e_1$.
We denote by $v_1$ the left top vertex, by~$v_2$ the right top vertex, and
by~$e_2$ the vertical edge.
In the case of the first diagram in Figure~\ref{Figd=3},
\e_ref{Contr3V_e}-\e_ref{ContrER_e}  and Lemma~\ref{LMHodge} give
\begin{gather*} 
\Cntr_{\Ga_1,\si_1;e_1}=\frac{-1}{4\al_1\al_3(\al_1\!+\!\al_3)^2}\,, \qquad
\Cntr_{\Ga_1,\si_1;e_2}=\frac{1}{\al_1^2\!-\!\al_3^2}\,,\\
(-1)^{\fs_{v_1}}\Cntr_{\Ga_1,\si_1;v_1}(3^3)=-(-1)^{g_1}
2\al_1\big(\al_1^2\!-\!\al_3^2\big)\big(\al_1\!+\!\al_3\big)
I_{2;g_1}\!\big(\al_1\!-\!\al_3,2\al_1,\al_1\!+\!\al_3\big),\\
(-1)^{\fs_{v_2}}\Cntr_{\Ga_1,\si_1;v_2}(3^3)=(-1)^{g_2}
2\al_3\big(\al_3\!+\!\al_1\big)
I_{1;g_2}\!\big(\al_3\!-\!\al_1,2\al_3,\al_3\!+\!\al_1\big).
\end{gather*}
By~\e_ref{EquivLocal_e}, the contributions of the first two diagrams 
to~\e_ref{EquivInt_e} with $d\!=\!3$ are~thus
\BE{d3c12_e}\begin{split}
\Cntr_{\Ga_1,\si_1}\!(3^3)&=(-1)^{g/2}
I_{2;g_1}\!\big(\al_1\!-\!\al_3,2\al_1,\al_1\!+\!\al_3\big)
I_{1;g_2}\!\big(\al_3\!-\!\al_1,2\al_3,\al_1\!+\!\al_3\big),\\
\Cntr_{\Ga_2,\si_2}\!(3^3)&=(-1)^{g/2}I_{2;g_1}\!\big(\al_1\!+\!\al_3,2\al_1,\al_1\!-\!\al_3\big)
I_{1;g_2}\!\big(\!\al_1\!+\!\al_3,2\al_3,\al_3\!-\!\al_1\big);
\end{split}\EE
the second expression is obtained from the first by replacing $\al_3$ by~$-\al_3$
(which corresponds to interchanging~3 and~4)
and using the second identity in~\e_ref{Isymm2}.
In the case of the third diagram in Figure~\ref{Figd=3},
\e_ref{Contr3V_e}-\e_ref{ContrER_e}  and Lemma~\ref{LMHodge} give
\begin{gather*} 
\Cntr_{\Ga_3,\si_3;e_1}=\frac{-1}{4\al_1\al_3(\al_1\!+\!\al_3)^2}\,, \qquad
\Cntr_{\Ga_3,\si_3;e_2}=\frac{1}{\al_3^2\!-\!\al_1^2}\,,\\
(-1)^{\fs_{v_1}}\Cntr_{\Ga_3,\si_3;v_1}(3^3)=(-1)^{g_1}
4\al_1^2\big(\al_1\!+\!\al_3\big)^2
I_{1;g_1}\!\big(\al_1\!-\!\al_3,2\al_1,\al_1\!+\!\al_3\big),\\
(-1)^{\fs_{v_2}}\Cntr_{\Ga_3,\si_3;v_2}(3^3)=-(-1)^{g_2}
\frac{2\al_3(\al_3^2\!-\!\al_1^2)}{3\al_3\!-\!\al_1}
I_{2;g_2}\!\big(\al_3\!-\!\al_1,2\al_3,\al_3\!+\!\al_1\big).
\end{gather*}
By~\e_ref{EquivLocal_e}, the contributions of the last two diagrams 
to~\e_ref{EquivInt_e} with $d\!=\!3$ are~thus
\BE{d3c34_e}\begin{split}
\Cntr_{\Ga_3,\si_3}\!(3^3)&=\frac{(-1)^{g/2}2\al_1}{3\al_3\!-\!\al_1}
I_{1;g_1}\!\big(\al_1\!-\!\al_3,2\al_1,\al_1\!+\!\al_3\big)
I_{2;g_2}\!\big(\al_3\!-\!\al_1,2\al_3,\al_1\!+\!\al_3\big),\\
\Cntr_{\Ga_4,\si_4}\!(3^3)&=\frac{-(-1)^{g/2}2\al_1}{3\al_3\!+\!\al_1}
I_{1;g_1}\!\big(\al_1\!+\!\al_3,2\al_1,\al_1\!-\!\al_3\big)
I_{2;g_2}\!\big(\al_1\!+\!\al_3,2\al_3,\al_3\!-\!\al_1\big).
\end{split}\EE
The genus~$g$ degree~3 real GW-invariant of $(\P^3,\tau_4)$ 
with 3~conjugate pairs of point constraints is the sum of 
the four fractions in~\e_ref{d3c12_e} and~\e_ref{d3c34_e} 
over all $g_1,g_2\!\in\!\Z^{\ge0}$ with $g_1\!+\!g_2\!=\!g/2$.
In particular, we obtain the numbers $\GW_{0,d}^{\phi}$, $\GW_{2,d}^{\phi}$, 
and $\GW_{4,d}^{\phi}$ listed in the $d\!=\!3$ column of Table~\ref{TabInvReal}.
\end{eg}

\begin{figure}
\begin{center}
\begin{tikzpicture}
 \def\edgmrk{
  (-0.5,1.5)--(2,1.5)
  (-0.5,0)--(2,0)
  (1.5,0)--(1.5,1.5)
  (0,0)--(-.5,0.4)
  (0,1.5)--(-.5,1.1)
 };
 \def\varedgmrk{
  (-0.5,1.5)--(2,1.5)
  (-0.5,0)--(2,0)
  (0,0)--(0,1.5)
  (0,0)--(-.5,0.4)
  (0,1.5)--(-.5,1.1)
 };
 \def\ver{
 (0,0) circle (1.5pt)
 (0,1.5) circle (1.5pt)
 (1.5,0) circle (1.5pt)
 (1.5,1.5) circle (1.5pt)
 };
 \draw[thick,xshift=4cm]\varedgmrk;
 \draw[thick,xshift=8cm]\edgmrk;
 \draw[thick,xshift=12cm]\edgmrk;
 \draw[thick]\varedgmrk;
 \filldraw\ver;
 \filldraw[xshift=4cm]\ver;
 \filldraw[xshift=8cm]\ver;
 \filldraw[xshift=12cm]\ver;
 \draw (-0.1,0.75)node{\small 1}
 (3.9,0.75)node{\small 1}
 (9.6,0.75)node{\small 1}
 (13.6,0.75)node{\small 1}
 (0.75,1.7)node{\small 1}
 (4.75,1.7)node{\small 1}
 (8.75,1.7)node{\small 1}
 (12.75,1.7)node{\small 1}
 (0.75,-0.2)node{\small 1}
 (4.75,-0.2)node{\small 1}
 (8.75,-0.2)node{\small 1}
 (12.75,-0.2)node{\small 1};
 \draw (0,1.7)node{\small $\mathfrak 1$} 
 (4,1.7)node{\small $\mathfrak 1$}
 (8,1.7)node{\small $\mathfrak 1$}
 (12,1.7)node{\small $\mathfrak 1$}
 (1.5,1.7)node{\small $\mathfrak 3$}
 (5.5,1.7)node{\small $\mathfrak 4$}
 (9.5,1.7)node{\small $\mathfrak 3$}
 (13.5,1.7)node{\small $\mathfrak 4$}
 (1.5,-.2)node{\small $\mathfrak 4$}
 (5.5,-.2)node{\small $\mathfrak 3$}
 (9.5,-.2)node{\small $\mathfrak 4$}
 (13.5,-.2)node{\small $\mathfrak 3$}
 (0,-.2)node{\small $\mathfrak 2$}
 (4,-.2)node{\small $\mathfrak 2$}
 (8,-.2)node{\small $\mathfrak 2$}
 (12,-.2)node{\small $\mathfrak 2$};
 \draw (-0.7,0.1)node{\small $\mathit 1^-$} 
 (3.3,0.1)node{\small $\mathit 1^-$}
 (7.3,0.1)node{\small $\mathit 1^-$}
 (11.3,0.1)node{\small $\mathit 1^-$}
 (-0.7,0.55)node{\small $\mathit 3^-$}
 (3.3,0.55)node{\small $\mathit 3^-$}
 (7.3,0.55)node{\small $\mathit 3^-$}
 (11.3,0.55)node{\small $\mathit 3^-$}
 (-0.7,1.2)node{\small $\mathit 3^+$}
 (3.3,1.2)node{\small $\mathit 3^+$}
 (7.3,1.2)node{\small $\mathit 3^+$}
 (11.3,1.2)node{\small $\mathit 3^+$}
 (-0.7,1.6)node{\small $\mathit 1^+$}
 (3.3,1.6)node{\small $\mathit 1^+$}
 (7.3,1.6)node{\small $\mathit 1^+$}
 (11.3,1.6)node{\small $\mathit 1^+$}
 (2.25,1.65)node{\small $\mathit 2^+$}
 (6.25,1.65)node{\small $\mathit 2^-$}
 (10.25,1.65)node{\small $\mathit 2^+$}
 (14.25,1.65)node{\small $\mathit 2^-$}
 (2.25,0)node{\small $\mathit 2^-$}
 (6.25,0)node{\small $\mathit 2^+$}
 (10.25,0)node{\small $\mathit 2^-$}
 (14.25,0)node{\small $\mathit 2^+$};
 \draw [<->,>=stealth] (4.75,0.2)--(4.75,1.3); 
 \draw [<->,>=stealth] (8.75,0.2)--(8.75,1.3);
 \draw [<->,>=stealth] (12.75,0.2)--(12.75,1.3);
 \draw [<->,>=stealth] (0.75,0.2)--(0.75,1.3);
 \draw (0.2,1.3) node {\footnotesize $g_1$}
 (4.2,1.3) node {\footnotesize $g_1$}
 (8.2,1.3) node {\footnotesize $g_1$}
 (12.2,1.3) node {\footnotesize $g_1$}
 (1.7,1.3) node {\footnotesize $g_2$}
 (5.7,1.3) node {\footnotesize $g_2$}
 (9.7,1.3) node {\footnotesize $g_2$}
 (13.7,1.3) node {\footnotesize $g_2$};
\end{tikzpicture}
\end{center}
\caption{Admissible pairs potentially contributing to~\e_ref{EquivInt_e} with $d\!=\!3$}
\label{Figd=3}
\end{figure}

 \begin{figure}
\begin{center}
\begin{tikzpicture}
 \def\edgmrk{
  (-0.5,1.5)--(2,1.5)
  (-0.5,0)--(2,0)
  (0,0)--(0,1.5)
  (1.5,0)--(1.5,1.5)
  (0,0)--(-.5,0.4)
  (0,1.5)--(-.5,1.1)
  (1.5,0)--(2,0.4)
  (1.5,1.5)--(2,1.1)
 };
 \def\varedgmrk{
  (-0.5,1.5)--(2,1.5)
  (-0.5,0)--(2,0)
  (0,0)..controls(-0.3,0.3) and (-0.3,1.2)..(0,1.5)
  (0,0)..controls(0.3,0.3) and (0.3,1.2)..(0,1.5)
  (0,0)--(-.5,0.4)
  (0,1.5)--(-.5,1.1)
  (1.5,0)--(2,0.4)
  (1.5,1.5)--(2,1.1)
 };
 \def\ver{
 (0,0) circle (1.5pt)
 (0,1.5) circle (1.5pt)
 (1.5,0) circle (1.5pt)
 (1.5,1.5) circle (1.5pt)
 };
 \draw[thick,xshift=4cm]\edgmrk;
 \draw[thick,xshift=8cm]\edgmrk;
 \draw[thick,xshift=12cm]\edgmrk;
 \draw[thick]\varedgmrk;
 \filldraw\ver;
 \filldraw[xshift=4cm]\ver;
 \filldraw[xshift=8cm]\ver;
 \filldraw[xshift=12cm]\ver;
 \draw (-0.34,0.75)node{\small 1}
 (0.12,0.75)node{\small 1}
 (5.6,0.75)node{\small 1}
 (3.9,0.75)node{\small 1}
 (7.9,0.75)node{\small 1}
 (11.9,0.75)node{\small 1}
 (9.6,0.75)node{\small 1}
 (13.6,0.75)node{\small 1}
 (0.75,1.7)node{\small 1}
 (4.75,1.7)node{\small 1}
 (8.75,1.7)node{\small 1}
 (12.75,1.7)node{\small 1}
 (0.75,-0.2)node{\small 1}
 (4.75,-0.2)node{\small 1}
 (8.75,-0.2)node{\small 1}
 (12.75,-0.2)node{\small 1};
 \draw (0,1.7)node{\small $\mathfrak 1$} 
 (4,1.7)node{\small $\mathfrak 1$}
 (8,1.7)node{\small $\mathfrak 1$}
 (12,1.7)node{\small $\mathfrak 1$}
 (1.5,1.7)node{\small $\mathfrak 3$}
 (5.5,1.7)node{\small $\mathfrak 3$}
 (9.5,1.7)node{\small $\mathfrak 4$}
 (13.5,1.7)node{\small $\mathfrak 3$}
 (1.5,-.2)node{\small $\mathfrak 4$}
 (5.5,-.2)node{\small $\mathfrak 4$}
 (9.5,-.2)node{\small $\mathfrak 3$}
 (13.5,-.2)node{\small $\mathfrak 2$}
 (0,-.2)node{\small $\mathfrak 2$}
 (4,-.2)node{\small $\mathfrak 2$}
 (8,-.2)node{\small $\mathfrak 2$}
 (12,-.2)node{\small $\mathfrak 4$};
 \draw (-0.7,0.1)node{\small $\mathit 1^-$} 
 (3.3,0.1)node{\small $\mathit 1^-$}
 (7.3,0.1)node{\small $\mathit 1^-$}
 (11.3,0.1)node{\small $\mathit 2^-$}
 (-0.7,0.55)node{\small $\mathit 3^-$}
 (3.3,0.55)node{\small $\mathit 3^-$}
 (7.3,0.55)node{\small $\mathit 3^-$}
 (11.3,0.55)node{\small $\mathit 4^-$}
 (-0.7,1.2)node{\small $\mathit 3^+$}
 (3.3,1.2)node{\small $\mathit 3^+$}
 (7.3,1.2)node{\small $\mathit 3^+$}
 (11.3,1.2)node{\small $\mathit 3^+$}
 (-0.7,1.6)node{\small $\mathit 1^+$}
 (3.3,1.6)node{\small $\mathit 1^+$}
 (7.3,1.6)node{\small $\mathit 1^+$}
 (11.3,1.6)node{\small $\mathit 1^+$}
 (2.25,1.65)node{\small $\mathit 2^+$}
 (6.25,1.65)node{\small $\mathit 2^+$}
 (10.25,1.65)node{\small $\mathit 2^-$}
 (14.25,1.65)node{\small $\mathit 2^+$}
 (2.3,1.15)node{\small $\mathit 4^+$}
 (6.3,1.15)node{\small $\mathit 4^+$}
 (10.3,1.15)node{\small $\mathit 4^-$}
 (14.3,1.15)node{\small $\mathit 4^+$}
 (2.3,0.55)node{\small $\mathit 4^-$}
 (6.3,0.55)node{\small $\mathit 4^-$}
 (10.3,0.55)node{\small $\mathit 4^+$}
 (14.3,0.55)node{\small $\mathit 3^-$}
 (2.25,0)node{\small $\mathit 2^-$}
 (6.25,0)node{\small $\mathit 2^-$}
 (10.25,0)node{\small $\mathit 2^+$}
 (14.25,0)node{\small $\mathit 1^-$};
 \draw [<->,>=stealth] (4.75,0.4)--(4.75,1.1); 
 \draw [<->,>=stealth] (8.75,0.4)--(8.75,1.1);
 \draw [<->,>=stealth] (12.4,0.4)--(13.1,1.1);
 \draw [<->,>=stealth] (13.1,0.4)--(12.4,1.1);
 \draw [<->,>=stealth] (0.3,0.4)--(0.5,1.1);
 \draw [<->,>=stealth] (0.5,0.4)--(0.3,1.1);
 \draw [<->,>=stealth] (0.8,0.4)--(0.8,1.1);
 \draw [<->,>=stealth] (0.95,0.4)--(0.95,1.1);
 \draw (0.4,1.3) node {\footnotesize $g_1$}
 (4.2,1.3) node {\footnotesize $g_1$}
 (8.2,1.3) node {\footnotesize $g_1$}
 (12.2,1.3) node {\footnotesize $g_1$}
 (1.3,1.3) node {\footnotesize $g_2$}
 (5.3,1.3) node {\footnotesize $g_2$}
 (9.3,1.3) node {\footnotesize $g_2$}
 (13.3,1.3) node {\footnotesize $g_2$};
 \draw (0.63,.75) node{\tiny{or}};
\end{tikzpicture}
\end{center}
\caption{Admissible pairs potentially contributing to~\e_ref{EquivInt_e} with $d\!=\!4$}
\label{Figg=1d=4}
\end{figure}
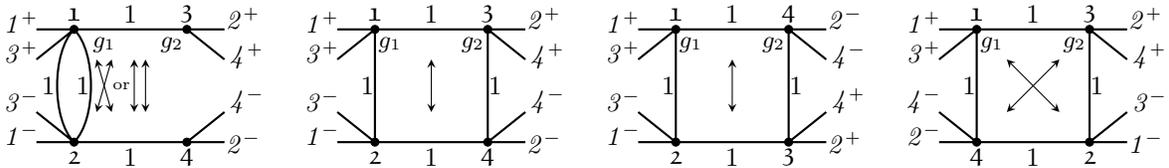

\begin{eg}[$d\!=\!4$]\label{Exd4}
Let  $g\!\in\!\Z^+\!\!-\!2\Z$.
The genus~$g$ degree~4 real GW-invariant of $(\P^3,\tau_4)$ 
with 3~conjugate pairs of point constraints 
is given by~\e_ref{EquivInt_e} with $d\!=\!4$. 
If $(\Ga,\si)\!\in\!\A_{g,4}(4,4)$ is  a contributing pair, then
\BE{D4vt} \vt\big(\fd(1^+)\big),\vt\big(\fd(3^+)\big)=1,\qquad
 \vt\big(\fd(2^+)\big),\vt\big(\fd(4^+)\big)=3.\EE
There are 11 types of admissible pairs satisfying~\e_ref{D3vt}:
the first diagram of Figure~\ref{Figg=1d=4}, 
with the 4~possible ways of labeling the vertices and 
the 2~possible involutions on the loop of edges, and 
the remaining 3 diagrams in this figure.
All 8 versions of the first diagram in Figure~\ref{Figg=1d=4} have 
the same automorphism group~$\Z_2$ and all edges of degree~1;
their  contributions thus cancel in pairs by~\cite[Corollary~4.8]{RealGWsIII}.
For each of the remaining diagrams,  
we denote  the top left and right vertices by~$v_1$ and~$v_2$, respectively,
and the top, left, and right edges by~$e_1$, $e_2$, and~$e_3$, respectively. 
We take $\V_+^{\si}(\Ga)\!=\!\{v_1,v_2\}$ in all three cases,
$\E_+^{\si}(\Ga)\!=\!\{e_1\}$ for the middle two diagrams,
and $\E_+^{\si}(\Ga)\!=\!\{e_1,e_2\}$ for the last diagram.
In the case of the second diagram in Figure~\ref{Figg=1d=4},
\e_ref{ContrEC_e},  \e_ref{ContrER_e}, \e_ref{Contr3V_e}, \e_ref{EtoLa_e}, 
and Lemma~\ref{LMHodge} give
\begin{gather*} 
\Cntr_{\Ga_2,\si_2;e_1}=\frac{-1}{4\al_1\al_3(\al_1\!+\!\al_3)^2}\,, \quad
\Cntr_{\Ga_2,\si_2;e_2}=\frac{1}{\al_1^2\!-\!\al_3^2}\,, \quad
\Cntr_{\Ga_2,\si_2;e_3}=\frac{1}{\al_3^2\!-\!\al_1^2}\,,\\
(-1)^{\fs_{v_1}}\Cntr_{\Ga_2,\si_2;v_1}(3^4)=-(-1)^{g_1}
2\al_1\big(\al_1^2\!-\!\al_3^2\big)\big(\al_1\!+\!\al_3\big)
I_{2;g_1}\!\big(\al_1\!-\!\al_3,2\al_1,\al_1\!+\!\al_3\big),\\
(-1)^{\fs_{v_2}}\Cntr_{\Ga_2,\si_2;v_2}(3^4)=-(-1)^{g_2}
2\al_3\big(\al_3^2\!-\!\al_1^2\big)\big(\al_3\!+\!\al_1\big)
I_{2;g_2}\!\big(\al_3\!-\!\al_1,2\al_3,\al_3\!+\!\al_1\big).
\end{gather*}
By~\e_ref{EquivLocal_e}, the contributions of the middle two diagrams 
to~\e_ref{EquivInt_e} with $d\!=\!4$ are~thus
\BE{d4c23_e}\begin{split}
\Cntr_{\Ga_2,\si_2}\!(3^4)&=-(-1)^{(g-1)/2}
I_{2;g_1}\!\big(\al_1\!-\!\al_3,2\al_1,\al_1\!+\!\al_3\big)
I_{2;g_2}\!\big(\al_3\!-\!\al_1,2\al_3,\al_1\!+\!\al_3\big),\\
\Cntr_{\Ga_3,\si_3}\!(3^4)&=
-(-1)^{(g-1)/2}I_{2;g_1}\!\big(\al_1\!+\!\al_3,2\al_1,\al_1\!-\!\al_3\big)
I_{2;g_2}\!\big(\al_1\!+\!\al_3,2\al_3,\al_3\!-\!\al_1\big);
\end{split}\EE
the second expression is obtained from the first by replacing $\al_3$ by~$-\al_3$
and using the second identity in~\e_ref{Isymm2}.
In the case of the last diagram in Figure~\ref{Figg=1d=4},
\e_ref{ContrEC_e},  \e_ref{Contr3V_e}, \e_ref{EtoLa_e},
 and Lemma~\ref{LMHodge} give
\begin{gather*} 
\Cntr_{\Ga_4,\si_4;e_1}=\frac{-1}{4\al_1\al_3(\al_1\!+\!\al_3)^2}\,, \qquad
\Cntr_{\Ga_4,\si_4;e_2}=\frac{1}{4\al_1\al_3(\al_1\!-\!\al_3)^2}\,,\\
(-1)^{\fs_{v_1}}\Cntr_{\Ga_4,\si_4;v_1}(3^4)=-(-1)^{g_1}
4\al_1^2\big(\al_1^2\!-\!\al_3^2\big)
I_{2;g_1}\big(\al_1\!-\!\al_3,\al_1\!+\!\al_3,2\al_1\big),\\
(-1)^{\fs_{v_2}}\Cntr_{\Ga_4,\si_4;v_2}(3^4)=-(-1)^{g_2}
4\al_3^2\big(\al_3^2\!-\!\al_1^2\big)
I_{2;g_2}\big(\al_3\!-\!\al_1,\al_3\!+\!\al_1,2\al_3\big).
\end{gather*}
By~\e_ref{EquivLocal_e}, the contribution of the last diagram 
to~\e_ref{EquivInt_e} with $d\!=\!4$ is~thus
\BE{d4c4_e}
\Cntr_{\Ga_4,\si_4}\!(3^4)=
(-1)^{(g-1)/2}I_{2;g_1}\!\big(\al_1\!-\!\al_3,\al_1\!+\!\al_3,2\al_1\big)
I_{2;g_2}\!\big(\al_3\!-\!\al_1,\al_1\!+\!\al_3,2\al_3\big).\EE
The genus~$g$ degree~4 real GW-invariant of $(\P^3,\tau_4)$ 
with 4~conjugate pairs of point constraints is the sum of 
the three expressions in~\e_ref{d4c23_e} and~\e_ref{d4c4_e} 
over all $g_1,g_2\!\in\!\Z^{\ge0}$ with $g_1\!+\!g_2\!=\!(g\!-\!1)/2$.
In particular, we obtain the numbers $\GW_{1,d}^{\phi}$, $\GW_{3,d}^{\phi}$, 
and $\GW_{5,d}^{\phi}$ listed in the $d\!=\!4$ column of Table~\ref{TabInvReal}.
\end{eg}

\subsection{Implications for Hodge integrals}
\label{Hurwitz_subs}

We put the Hodge integrals $I_{1;g}$ and $I_{2;g}$ into the generating functions
$$F_1(u_1,u_2,u_3;t)=\sum_{g=0}^{\i} I_{1;g}(u_1,u_2,u_3)t^{2g}, \qquad
F_2(u_1,u_2,u_3;t)=\sum_{g=0}^{\i} I_{2;g}(u_1,u_2,u_3)t^{2g}\,.$$

\begin{prp}\label{Hodge_prp}
With notation as above,
\begin{gather}
F_1(x\!+\!y,x,y;t)= \frac{\sin(t/2)}{t/2}\,,\label{F1}\\
\begin{split}
&\tfrac{x+y}{2x-y}F_1(x,x\!+\!y,y;t)F_2(x,x\!-\!y,-y;t)
+\tfrac{x+y}{2y-x}F_1(y,x\!+\!y,x;t)F_2(-y,x\!-\!y,x;t)\\
&-F_1(x,x\!-\!y,-y;t)F_2(x,x\!+\!y,y;t)
-F_1(-y,x\!-\!y,x;t)F_2(y,x\!+\!y,x;t)
=\bigg(\frac{\sin(t/2)}{t/2}\bigg)^{\!5}\,,
\end{split}\label{F12}\\
\begin{split}
F_2(x\!+\!y,x,y;t)F_2(x\!-\!y,x,-y;t)
&+F_2(x\!+\!y,y,x;t)F_2(x\!-\!y,-y,x;t)\\
&- F_2(x,y,x\!+\!y;t)F_2(x,-y,x\!-\!y;t)
=\bigg(\frac{\sin(t/2)}{t/2}\bigg)^{\!8}\,.\label{F2}
\end{split}
\end{gather}
\end{prp}

\begin{proof}
By the $n\!=\!1$ case of \cite[Examples~6.3]{Teh} or the $g\!=\!0$ case of Example~\ref{Exd1}, 
$$\E^{\P^3,\tau_4}_{0,1}(H^3)=\GW^{\P^3,\tau_4}_{0,1}(H^3)=1.$$
Since every degree~1 curve in~$\P^3$ is of genus~0, $\E^{\P^3,\tau_4}_{g,1}(H^3)\!=\!0$
for every $g\!\in\!\Z^+$.
Along with~\e_ref{FanoGV_e} and~\e_ref{Cdfn_e}, 
these two observations~give
$$\sum_{g=0}^{\i}(-1)^g\GW^{\P^3,\tau_4}_{2g,1}(H^3)t^{2g}
=\sum_{g=0}^{\i}\ti{C}_{0,\ell}^{\P^3}(g)(\fI t)^{2g}
=\bigg(\frac{\sin(t/2)}{t/2}\bigg)^{\!0-1+4/2}\,.$$
Combining this statement with~\e_ref{d1GW_e} and setting $x\!=\!\al_1\!+\!\al_3$ and 
$y\!=\!\al_1\!-\!\al_3$, we obtain~\e_ref{F1}.\\

By the $n\!=\!1$ case of \cite[Examples~6.4]{Teh} or the $g\!=\!0$ case of Example~\ref{Exd3}, 
$$\E^{\P^3,\tau_4}_{0,3}(H^3,H^3,H^3)=\GW^{\P^3,\tau_4}_{0,3}(H^3,H^3,H^3)=-1.$$
By~\cite[Example 4.11]{RealGWsIII} or the $g\!=\!1$ case of Example~\ref{Exd4}, 
$$\E^{\P^3,\tau_4}_{1,4}(H^3,H^3,H^3,H^3)=\GW^{\P^3,\tau_4}_{1,4}(H^3,H^3,H^3,H^3)=-1.$$
By the Castelnuovo bound, every degree~$3,4$ curve of genus~2 or higher lies in a~$\P^2$
and~so
$$\E^{\P^3,\tau_4}_{2g,3}(H^3,H^3,H^3),
\E^{\P^3,\tau_4}_{2g+1,4}(H^3,H^3,H^3,H^3)=0
\qquad\forall~g\!\in\!\Z^+.$$
Along with~\e_ref{FanoGV_e} and~\e_ref{Cdfn_e}, 
these three observations~give
\begin{equation*}\begin{split}
\sum_{g=0}^{\i}(-1)^g\GW^{\P^3,\tau_4}_{2g,3}(H^3,H^3,H^3)t^{2g}
&=-\!\sum_{g=0}^{\i}\ti{C}_{0,3\ell}^{\P^3}(g)(\fI t)^{2g}
=-\bigg(\frac{\sin(t/2)}{t/2}\bigg)^{\!0-1+12/2}\,,\\
\sum_{g=0}^{\i}(-1)^g\GW^{\P^3,\tau_4}_{1+2g,4}(H^3,H^3,H^3,H^3)t^{2g}
&=-\!\sum_{g=0}^{\i}\ti{C}_{1,4\ell}^{\P^3}(g)(\fI t)^{2g}
=-\bigg(\frac{\sin(t/2)}{t/2}\bigg)^{\!1-1+16/2}\,.
\end{split}\end{equation*}
Combining these two statements with the conclusions of Examples~\ref{Exd3}
and~\ref{Exd4} and setting $x\!=\!\al_1\!+\!\al_3$ and $y\!=\!\al_1\!-\!\al_3$, 
we obtain~\e_ref{F12} and~\e_ref{F2}.
\end{proof}

The strategy in the proof of Proposition~\ref{Hodge_prp} can be applied 
with complex GW-invariants
using \cite[Theorem~1.5]{FanoGV}.
The $d\!=\!1$ case would then give
\BE{F1sq} F_1(u,x,y;t)F_1(u,u\!-\!x,u\!-\!y;t)=\left(\frac{\sin(t/2)}{t/2}\right)^{\!2}\,.\EE
The $u\!=\!x\!+\!y$ case of this identity recovers~\e_ref{F1}.\\

It is straightforward to see that
$$F_1\big(u_1,u_2,u_3;t\big)=F_1\big(1,u_2/u_1,u_3/u_1;t\big), \qquad 
F_2\big(u_1,u_2,u_3;t\big)=F_2\big(u_1/u_3,u_2/u_3,1;t\big)$$
By Mumford's relation \cite[(5.3)]{Mu}, 
\begin{equation*}\begin{split}
F_1(-x,x,-2x;t)
&=1+\!\sum_{g=1}^{\i}(-1)^g t^{2g}\!\int_{\ov\cM_{g,1}}\!\!
 \frac{\La(-2)}{(-1)(-1\!-\!\psi_1)}
=1+\!\sum_{g=1}^{\i}\sum_{i=0}^g(-2)^i\!\int_{\ov\cM_{g,1}}\!\!\!\!\psi_i^{2g-2+i}\la_{g-i}\,.
\end{split}\end{equation*}
Thus, the $y\!=\!-2x$ case of~\e_ref{F1} is equivalent to the $k\!=\!-2$ case
of \cite[Theorem 2]{FP}.\\

\noindent
The coefficients of $F_1(1,x,y;t)$ are rational symmetric functions in $x,y$.
Thus, $F_1(1,x,y;t)$ is a function of $x\!+\!y$, $xy$, and~$t$. 
However, it does not appear to depend on~$xy$.

\begin{cnj}\label{F1_cnj}
The function $F_1\!=\!F_1(u_1,u_2,u_3;t)$ depends only on $u_1$, $u_2\!+\!u_3$, and~$t$.
\end{cnj}

If this is the case, \e_ref{F1sq} would imply that 
$$F_1(u_1,u_2,u_3;t)= 
\frac{\sin(t/2)}{t/2}\exp\Big(f_1\big((u_2\!+\!u_3)/u_1\!-\!1;t\big)\Big)$$
for some $f_1\!\in\!\Q(z)[[t^2]]$ such that $f_1(-z;t)\!=\!-f_1(z;t)$.\\

The coefficients of $F_2(x,y,1;t)$ are rational symmetric functions in $x,y$.
However, they appear to depend on $x\!+\!y$ and~$xy$ and not just up to 
a fixed uniform scaling factor.
On the other hand, the next statement appears plausible.

\begin{cnj}\label{F2_cnj}
With notation as above,
\BE{F2cnj_e}  F_2(x,y,x\!+\!y;t)F_2(x,-y,x\!-\!y;t)
 =-\left( \frac{\sin(t/2)}{t/2}\right)^{\!8} \frac{(x^2\!-\!y^2)^2}{x^2y^2}\,.\EE
\end{cnj}

Along with \e_ref{F2}, \e_ref{F2cnj_e} would imply~that 
\begin{equation*}\begin{split}
F_2(x\!+\!y,x,y;t)F_2(x\!-\!y,x,-y;t)
&+F_2(x\!+\!y,y,x;t)F_2(x\!-\!y,-y,x;t)\\
&\hspace{.5in}
=-\left( \frac{\sin(t/2)}{t/2}\right)^{\!8}\frac{x^4\!-\!3x^2y^2\!+\!y^4}{x^2y^2}\,.
\end{split}\end{equation*}

\vspace{.5in}

\noindent
{\it Department of Mathematics, Stony Brook University, Stony Brook, NY 11794-3651\\
niu@math.stonybrook.edu, azinger@math.stonybrook.edu}


\begin{thebibliography}{99}


\bibitem{Faber} C.~Faber,
{\it A non-vanishing result for the tautological ring of~$\cM_g$},
math/9711219

\bibitem{FP} C.~Faber and R.~Pandharipande,
{\it Hodge integrals and Gromov-Witten theory},
Invent.~Math.~139 (2000), 173--199

\bibitem{Teh} M.~Farajzadeh Tehrani,
{\it Counting genus zero real curves in symplectic manifolds},
math/1205.1809v4, to appear in Geom.~Top.

\bibitem{GWsAbsRel} M.~Farajzadeh Tehrani and A.~Zinger,
{\it Absolute vs.~Relative Gromov-Witten Invariants},
to appear in J.~Symplectic Geom.

\bibitem{growi} A.~Gathmann, {\it GROWI},
available on the author's website 

\bibitem{Ge2} P.~Georgieva,
{\it Open Gromov-Witten invariants in the presence of an anti-symplectic involution},
math/1306.5019v2

\bibitem{RealGWsI} P.~Georgieva and A.~Zinger,
{\it Real Gromov-Witten invariants and lower bounds in real enumerative geometry: construction},
math/1504.06617

\bibitem{RealGWsII} P.~Georgieva and A.~Zinger,
{\it Real Gromov-Witten invariants and lower bounds in real enumerative geometry: properties},
math/1507.06633

\bibitem{RealGWsIII} P.~Georgieva and A.~Zinger,
{\it Real Gromov-Witten invariants and lower bounds in real enumerative geometry: computation},
math/1510.07568 

\bibitem{GP} T.~Graber and R.~Pandharipande, {\it Localization of virtual classes},
Invent.~Math.~135 (1999), no.~2, 487--518

\bibitem{KoM} M.~Kontsevich and Yu.~Manin,
{\it Gromov-Witten classes, quantum cohomology, and enumerative geometry}, 
Comm.~Math.~Phys.~164 (1994), no.~3, 525--562

\bibitem{LiT}  J.~Li and G.~Tian, 
{\it Virtual moduli cycles and Gromov-Witten invariants of general symplectic manifolds}, 
Topics in Symplectic \hbox{$4$-Manifolds},
Internat.~Press, 1998

\bibitem{LiZ}  J.~Li and A.~Zinger, 
{\it On Gromov-Witten invariants of a quintic threefold
and a rigidity conjecture}, Pacific J.~Math.~233 (2007), no.~2, 417-480

\bibitem{McS} D.~McDuff and D.~Salamon,
{\it $J$-holomorphic Curves and Symplectic Topology}, 2nd.~Ed.,
AMS~2012

\bibitem{Mu} D.~Mumford,
{\it Towards an enumerative geometry of the moduli space of curves, Arithmetic and Geometry},
Vol. II, 271-328, Progr.~Math.~36, Birkh\"auser, 1983

\bibitem{RealGWvsEnumApp} J.~Niu and A.~Zinger,
{\it Lower bounds for enumerative counts of positive-genus real curves: appendix},
available on the authors' websites

\bibitem{P1} R.~Pandharipande,
{\it Hodge integrals and degenerate contributions}, 
Comm.~Math.~Phys.~208  (1999),  no.~2, 489--506

\bibitem{P2} R.~Pandharipande,
{\it Three questions in Gromov-Witten theory},
Proceedings of ICM, Beijing (2002),  503--512
 
\bibitem{RT}  Y.~Ruan and G.~Tian, {\it A mathematical theory of quantum cohomology},  
JDG 42 (1995),  no.~2, 259--367

\bibitem{desing} R.~Vakil and A.~Zinger, 
{\it A desingularization of the main component of the moduli 
space of genus-one stable maps into $\P^n$}
Geom.~Topol.~12 (2008), 1-95

\bibitem{Wel4}  J.-Y.~Welschinger,  
\emph{Invariants of real symplectic 4-manifolds and lower bounds in real enumerative geometry},
Invent.~Math.~162 (2005), no.~1, 195--234

\bibitem{Wel6} J.-Y.~Welschinger,  
\emph{Spinor states of real rational curves in real algebraic convex 3-manifolds 
and enumerative invariants},
Duke Math.~J.~127 (2005), no.~1, 89–-121

\bibitem{g2n2and3} A.~Zinger,
{\it Enumeration of genus-two curves with a fixed complex structure
in $\P^2$ and $\P^3$}, J.~Diff.~Geom.~65 (2003), no.~3, 341-467 

\bibitem{gluing} A.~Zinger, 
{\it Enumerative vs.~symplectic invariants and obstruction bundles},
J.~Symplectic Geom.~2 (2004), no.~4, 445--543

\bibitem{g0pr} A.~Zinger, 
{\it Counting rational curves of arbitrary shape in projective spaces},
Geom.~Top.~9 (2005), 571--697

\bibitem{g1comp} A.~Zinger,
{\it A sharp compactness theorem for genus-one pseudo-holomorphic maps}, 
Geom.\&Top.~13 (2009) 2427-2522

\bibitem{g1comp2} A.~Zinger, 
{\it Reduced genus-one Gromov-Witten invariants}, 
JDG~83 (2009), no.~2, 407--460

\bibitem{FanoGV} A.~Zinger, 
\emph{A comparison theorem for Gromov-Witten invariants in the symplectic category},
Adv.~Math.~228 (2011), no.~1, 535--574



\end{thebibliography}
\end{document}